\documentclass[11pt]{article}

\usepackage{amsfonts, amscd, amssymb, amsthm, amsmath} 
\usepackage{enumerate}

\theoremstyle{plain}
	\newtheorem{thm}{Theorem}[section]
	
	\newtheorem{prop}[thm]{Proposition}
	\newtheorem{cor}[thm]{Corollary}
\theoremstyle{definition}
	
	\newtheorem{remark}[thm]{Remark}
\theoremstyle{example}
	\newtheorem{example}{Example}
	
\theoremstyle{remark}
	
	\numberwithin{equation}{section}

\def\cA{\mathcal{A}}\def\cB{\mathcal{B}}\def\cC{\mathcal{C}}\def\cD{\mathcal{D}}\def\cF{\mathcal{F}}\def\cH{\mathcal{H}}\def\cO{\mathcal{O}}\def\cP{\mathcal{P}}\def\cT{\mathcal{T}}\def\cW{\mathcal{W}}\def\cZ{\mathcal{Z}}

  \def\CC{\mathbb{C}}               \def\RR{\mathbb{R}}        \def\ZZ{\mathbb{Z}}  

      \def\fg{\mathfrak{g}} \def\fh{\mathfrak{h}}                 

\def\fgl{\mathfrak{gl}}  \def\fsl{\mathfrak{sl}}

\def\ad{\mathrm{ad}} 

\def\Card{\mathrm{Card}} 
\def\cc{\mathbf{c}}
 
\def\dim{\mathrm{dim}} 
\def\End{\mathrm{End}} 
 
\def\ext{\textrm{ext}}

\def\id{\mathrm{id}}  
\def\Ind{\mathrm{Ind}}

\def\ring{\mathring}

\def\sgn{\mathrm{sgn}}

\def\Tr{\mathrm{Tr}} 
\def\vep{\varepsilon}

\def\half{\hbox{$\frac12$}}
\def\<{\langle}	\def\>{\rangle}

\newcommand{\comment}[1]{}

\usepackage{color}
\definecolor{dred}{rgb}{.65, 0, 0.15}
\definecolor{gray}{rgb}{0.6, .6, .6}

\usepackage[pdftex,bookmarks,linktocpage=true]{hyperref}

\usepackage{tikz}
	\usepgflibrary[patterns] 
	\usetikzlibrary{patterns} 
	\usepgflibrary{shapes.geometric}

\tikzstyle over=[draw=white,double=black,line width=2pt, double distance=.5pt]
\tikzstyle{B}=[draw, fill=black, circle, inner sep=0pt, outer sep=0pt, minimum size=5pt]
\tikzstyle{V}=[draw, fill =black, circle, inner sep=0pt, minimum size=2pt]
\tikzstyle{BoxArr}=[xscale = .2, yscale=-.2]
\tikzstyle{WhiteCircle}=[draw,circle,white, fill=white]
\tikzstyle{bV}=[draw, fill =black, circle, inner sep=0pt, minimum size=3.5pt]
\tikzstyle{cV}=[draw, fill =white, circle, inner sep=0pt, minimum size=3.5pt]
\newcommand\Cont[2]{\node at (#1) {\small #2};}

\def\Over[#1,#2][#3,#4]{ 
	\draw[style=over]   (#1,#2) .. controls ++(0,#4*.5-#2*.5) and ++(0,-#4*.5+#2*.5) .. (#3,#4);}
\def\Under[#1,#2][#3,#4]{ 
	\draw  (#1,#2) .. controls ++(0,#4*.5-#2*.5) and ++(0,-#4*.5+#2*.5) .. (#3,#4);}
\def\Cross[#1,#2][#3,#4]{
	\Under[#3,#2][#1,#4]\Over[#1,#2][#3,#4]}

\def\Tops[#1][#2][#3]{
	\foreach\x in {#1}{
		\draw (\x+.15,#2) -- (\x+.15,#2+.1) (\x-.15,#2) -- (\x-.15,#2+.1) ;
		\draw (\x+.15,#2+.1) arc (0:360:1.5mm and .75mm);}
	\foreach \x in {1,...,#3} {\draw (\x,#2)  to (\x,#2+.05); \filldraw [black] (\x,#2+.05) circle (2pt);}
	}
\def\Bottoms[#1][#2][#3]{
	\foreach\x in {#1}{
		\draw (\x+.15,#2) -- (\x+.15,#2-.1) (\x-.15,#2) -- (\x-.15,#2-.1) ;
		\draw (\x+.15,#2-.1) arc (0:-180:1.5mm and .75mm);}
	\foreach \x in {1,...,#3} {\draw (\x,#2)  to (\x,#2-.05); \filldraw [black] (\x,#2-.05) circle (2pt);}
	}
\def\Caps[#1][#2,#3][#4]{
	\Tops[#1][#3][#4]
	\Bottoms[#1][#2][#4]
	}
\def\Pole[#1][#2,#3]{
	\shade[left color=white,right color=white] (#1+.15,#2) rectangle (#1-.15,#3);
	\draw[over] (#1+.15,#2) to (#1+.15,#3) (#1-.15,#2) to (#1-.15,#3) ;}
\def\Label[#1,#2][#3][#4]{
	\node[above] at (#3,#2+.1) {#4};
	\node[below] at (#3,#1-.1) {#4};		}
\def\Nodes[#1][#2]{
	 \foreach \x in {1,...,#2} {\filldraw [black] (\x,#1) circle (2pt);	}
	}

\def\PoleTwist[#1,#2]{
	\foreach \x/\y in {-1/1L, -.7/1R, 0/2L, .3/2R}{\coordinate(T\y) at (\x,#2); \coordinate(B\y) at (\x,#1);}
	\draw[thin] (B1R) .. controls ++(0,#2*.5-#1*.5-.1) and ++(0,-#2*.5+#1*.5-.1) ..  (T2R)
			(B1L)   .. controls ++(0,#2*.5-#1*.5+.1) and ++(0,-#2*.5+#1*.5+.1) ..    (T2L) ;
	\draw[line width=2pt, white]
			(.15,#1)  .. controls +(0,#2*.5-#1*.5) and +(0,-#2*.5+#1*.5) ..   (-.85,#2) ;
	\draw[thin,over] 
		(B2R) .. controls ++(0,#2*.5-#1*.5+.1) and ++(0,-#2*.5+#1*.5+.1) ..  (T1R) 
			(B2L)  .. controls +(0,#2*.5-#1*.5-.1) and +(0,-#2*.5+#1*.5-.1) ..   (T1L) ;
			}

\newcounter{r}
\newcommand\Part[1]{
        \setcounter{r}{1}
	 \foreach \x in {#1}{
 	{\ifnum\value{r}=1
		\draw (0,\value{r}-1)--(\x,\value{r}-1); 
		\fi}
	\draw (0,\value{r}) to (\x,\value{r});
   	\foreach \y in {0, ..., \x} {\draw (\y,\value{r})--(\y,\value{r}-1);}
	\addtocounter{r}{1}
 }}
 
  \def\PartUNIT{.175}

\newcommand{\BOX}[2]{
	\draw[fill=white] (#1) to ++(1,0) to ++(0, 1) to ++(-1, 0) to ++(0, -1);
	\begin{scope}[xshift=.5cm, yshift=.5cm]\node at (#1) {\tiny #2};\end{scope}
	}

\newcommand{\DOT}[1]{
\filldraw [red] (#1) circle (4pt);
}
\newcommand{\bDOT}[1]{
\filldraw [red] (#1) circle (6pt);
}

\title{Two boundary Hecke Algebras and \\combinatorics of type $C$}
\author{
Zajj Daugherty\\
Department of Mathematics\\
The City College of New York\\
NAC 8/133\\
New York, NY 10031\\
zdaugherty@ccny.cuny.edu
\and
Arun Ram \\
Department of Mathematics and Statistics \\
University of Melbourne \\
Parkville VIC 3010 Australia \\
aram@unimelb.edu.au }
\date{}

\begin{document}
\maketitle 

\begin{abstract} 
This paper gives a Schur-Weyl duality approach to the representation theory of the 
affine Hecke algebras of type C with unequal parameters.  The first step is to realize the affine braid group
of type $C_k$ as the group of braids on $k$ strands with two poles.  Generalizing familiar methods from 
the one pole (type A) case, this provides commuting actions
of the quantum group $U_q\fg$ and the affine braid group of type $C_k$ on a tensor space $M\otimes N \otimes V^{\otimes k}$.
Special cases provide Schur-Weyl pairings 
between the affine Hecke algebra of type $C_k$ and the quantum group of type $\fgl_n$, resulting in
natural labelings of many representations of the affine Hecke algebras of type C by partitions.  
Following an analysis
of the structure of weights of affine Hecke algebra representations (extending the one parameter
case to the three parameter case necessary for affine Hecke algebras of type C), we provide an explicit identification
of the affine Hecke algebra representations that appear in tensor space (essentially by identifying their Langlands parameters).

\smallskip

\emph{AMS 2010 subject classifications:}  	20C08 (17B10, 17B37, 05E10)

\end{abstract}

\setcounter{tocdepth}{2}
\tableofcontents

\section{Introduction}

This paper explores a Schur-Weyl duality approach to the representations of the affine Hecke algebras of type C with unequal parameters.  Following Kazhdan-Lusztig \cite{KL}, the irreducible representations of the affine Hecke algebra are usually constructed via the K-theory of generalized Springer fibers. This method works well when an algebraic group is available, which is only for special cases of the three parameters $t, t_0, t_k$ of the affine Hecke algebras of type C.  

G.\ Lusztig gave a general approach to the unequal parameter case using Kazhdan-Lusztig bases and
cells.  In \cite{Lu2}, the challenges for pushing this method through in type C are outlined in a set of 
conjectures, many of which have now been settled in work of Geck, Bonnaf\'e, and others (see \cite{Ge, Bo, Gu} and references there). Another analytic approach, closer to the original classification and construction of Kazhdan-Lusztig, is given by Opdam and Solleveld (see \cite{OS} and \cite{So} and the references there).  In the type C case, Kato \cite{Kt} explained that the ``exotic nilpotent cone" can be used to replace the Kazhdan-Lusztig geometry and obtain a complete geometric classification of the irreducible representations of affine Hecke algebras (with mild restrictions on parameters).

In the type A case, there is a powerful alternative to the geometric method via
Schur-Weyl duality (see for example \cite{AS, OR, VV}).
In this paper we provide an analogue of this Schur-Weyl duality approach for the type C case, with unequal
parameters.  This is a generalization of the degenerate case studied by Daugherty \cite{Da}.

The method is the following:
Let $U_q\fgl_n$ be the Drinfeld-Jimbo quantum group corresponding to the general linear Lie algebra,
and let $V=\CC^n$ be the standard representation of $U_q\fgl_n$. Write $L(\lambda)$ for the irreducible polynomial representation of $U_q\fgl_n$ indexed by the partition $\lambda$, let $M=L((a^c))$ and $N=L((b^d))$ be irreducible representations of $U_q\fgl_n$ indexed by $a \times c$ and $b \times d$ rectangles.  
There is an action of an extension of the affine Hecke algebra 
 of type $C_k$, denoted $H_k^{\mathrm{ext}}$, with parameters
$$t^{\frac12} = q, \qquad t_0^{\frac12} = -iq^{b+d}, \qquad \text{ and } \qquad t_k^{\frac12} = -iq^{a+c} 
\qquad\hbox{(where $i=\sqrt{-1}$)},$$
such that 
$$M\otimes N\otimes V^{\otimes k} \qquad \hbox{is a $(U_q\fgl_n, H_k^{\mathrm{ext}})$-bimodule.}$$
We show that the commuting actions of $U_q\fgl_n$ and $H_k^{\mathrm{ext}}$ provide a Schur-Weyl duality, which can be used to derive the representation theory of $H_k^{\mathrm{ext}}$ from the quantum group $U_q\fgl_n$.  We work out the combinatorics of this correspondence, relating the natural indexing of $H_k^{\mathrm{ext}}$-modules coming from the Schur-Weyl duality to the other indexings, by describing the weights for the action of the polynomial part (generated by Bernstein generators) on each irreducible module.

A significant portion of the work in identifying the centralizer of the $U_q\fgl_n$ action on $M\otimes N\otimes V^{\otimes k}$ as 
an extended affine Hecke algebra of type C is in relating Coxeter and Bernstein presentations, and 
putting the parameter conversions into focus. 
The relationships between these presentations are given in Theorem \ref{thm:BraidPres} for the affine braid group of type C, 
and in Theorem \ref{thm:2bdryHeckePresentation} for the affine Hecke algebra of type C. 
Sections 3, 4 and 5 could, perhaps have stood as papers on their own.
In Section \ref{sec:CalibratedRepsOfHk}, we give the combinatorics of local regions and standard tableaux 
for the case of type C with unequal parameters (following the equal parameter case done in \cite{Ra2}).
The main result of Section 3, Theorem \ref{thm:calibconst}, 
provides a classification and a construction of all irreducible calibrated $H_k^{\mathrm{ext}}$-modules.
As in \cite{Ra2}, this classification is via \emph{skew local regions}, whose precise definition of skew local regions depends on the careful analysis of the
structure of the irreducible representations of rank two affine Hecke algebras. This analysis was done in the single parameter case in \cite{Ra1}.  
Since the corresponding analysis for \emph{three distinct parameters in the type $C_2$ case} is, to our knowledge, not available in the literature, we have provided it in Section 4.  
This will ensure that our classification of calibrated irreducible representations for $H_k^{\mathrm{ext}}$ with distinct parameters, as given in Theorem \ref{thm:calibconst}, is on firm footing.
The construction of the action of $H_k^{\mathrm{ext}}$ on tensor space is completed in 
Theorems \ref{thm:BraidRep} and \ref{thm:HeckeActionOnTensorSpace}. 
Finally, armed with these tools we prove the main result, Theorem \ref{thm:partitions-to-SLRs},
which determines exactly which representations of $H_k^{\mathrm{ext}}$ appear in tensor space,  
comparing the natural indexing from the highest weight theory for $\fgl_n$  to 
the combinatorics of the weights of the action of the polynomial part of $H_k^{\mathrm{ext}}$.  

Following the schematic from \cite{OR}, one would like to generalize the analysis in this paper by replacing finite-dimensional $M$ and $N$ with, for example, other modules from category $\cO$. In the finite-dimensional case, the key is that $R$-matrices for $M\otimes V$ and  $N \otimes V$ have only two eigenvalues. This strongly restricts the choices for $M$ and $N$. Non-finite-dimensional choices of modules $M$ and $N$ that satisfy these conditions exist
in category $\cO$,  but additional work toward understanding the combinatorics of $M\otimes N\otimes V^{\otimes k}$ in these cases is needed. This understanding would yield an interesting generalization of the work in this paper.

The seeds of the idea for this paper were sown during conversations of A.\ Ram with P.\ Pyatov and V.\ Rittenberg in Bonn in 2005.  They suggested that one should analyze two boundary spin chains by $R$-matrices, thus implying the possibility for  Schur-Weyl duality approach to representations of affine braid groups of type C. This idea was completed in the degenerate case in \cite{Da}, and significant information was obtained in the Temperley-Lieb case in \cite{GN} (see also references there). 
In \cite{DR} we shall complete the connection to  the statistical mechanics by using the results of this paper to identify the representations of the two boundary Temperley-Lieb algebra given, in a diagrammatic form, by de Gier and Nichols in \cite{GN}.

\smallskip\noindent
\textbf{Acknowledgements.}  We thank the Australian Research Council and the National Science Foundation for support of our research under grants DP130100674 and DMS-1162010. Much of the research for this paper was completed during residency at the special semester on ``Automorphic forms, Combinatorial representation theory, and Multiple Dirichlet series" at ICERM in 2013. We thank ICERM, all the ICERM staff and the organizers of the special semester for providing a wonderful and stimulating working environment.

\section{The two boundary Hecke algebra}

In this section we define the two boundary braid group and Hecke algebras and establish multiple presentations of each.
The conversion between presentations is important for matching the algebraic approach to the representation theory 
with the Schur-Weyl duality approach that we give in Section 5.

\medskip

For generators $g_i, g_j$, encode relations graphically by
\begin{equation}\label{braidlengths}
\begin{array}{cl}
\begin{tikzpicture}
	\draw[fill=white] (0,0) circle (2.5pt) node[above=1pt] {\small $g_i$};
	\draw[fill=white] (1,0) circle (2.5pt) node[above=1pt] {\small $g_j$}; 
\end{tikzpicture}
&\hbox{means $g_ig_j =g_jg_i$,} 
\\ \\
\begin{tikzpicture}
	\draw (0,0)--(1,0);
	\draw[fill=white] (0,0) circle (2.5pt) node[above=1pt] {\small $g_i$};
	\draw[fill=white] (1,0) circle (2.5pt) node[above=1pt] {\small $g_j$}; 
\end{tikzpicture}
&\hbox{means $g_ig_jg_i = g_jg_ig_j$, and} \\ \\
\begin{tikzpicture}
	\draw[double distance = 2pt] (0,0)--(1,0);
	\draw[fill=white] (0,0) circle (2.5pt) node[above=1pt] {\small $g_i$};
	\draw[fill=white] (1,0) circle (2.5pt) node[above=1pt] {\small $g_j$}; 
\end{tikzpicture}
&\hbox{means $g_ig_jg_ig_j = g_jg_ig_jg_i$.} 
\end{array}
\end{equation}
For example, the group of signed permutations,
\begin{equation}
\label{eq:Weyl-group}
\cW_0 = \left\{ 
\begin{matrix}
\hbox{bijections $w\colon \{-k, \ldots, -1, 1, \ldots, k\} \to \{ -k, \ldots, -1, 1, \ldots, k\}$} \\
\hbox{such that $w(-i) = -w(i)$ for $i=1, \ldots, k$}
\end{matrix} \right\},
\end{equation}
has a presentation by generators $s_0, s_1,\ldots, s_{k-1},$ with relations
\begin{equation}
\begin{tikzpicture}
	\foreach \x in {0,1, 2, 4,5}{
		\draw (\x,0) circle (2.5pt);
		\node(\x) at (\x,0){};
		}
	\foreach \x in {1, 2}{
		\node[label=above:$s_{\x}$] at (\x,0){};
		\node[label=above:$s_{k-\x}$] at (6-\x,0){};}
	\node[label=above:$s_0$] at (0){};
	\draw[double distance = 2pt] (0)--(1);
	\draw (1)--(2) (4)--(5);
	\draw[dashed] (2) to (4);
\end{tikzpicture}
\qquad\hbox{and}\qquad s_i^2=1\ \hbox{for $i=0,1,2,\ldots, k-1$.}
\label{W0defn}
\end{equation}

\subsection{The two boundary braid group}

The \emph{two boundary braid group} is the group $\cB_k$ generated by 
$\bar{T}_0, \bar{T}_1, \ldots, \bar{T}_k$,
with relations
\begin{equation}\label{Bdefrels}
\begin{matrix}
\begin{tikzpicture}
	\foreach \x in {0, 1, 2, 4,5,6}{
		\draw (\x,0) circle (2.5pt);
		\node(\x) at (\x,0){};
		}
	\foreach \x in {0, 1, 2}
		{\node[label=above:$\bar{T}_\x$] at (\x){};}
	\node[label=above:$\bar{T}_{k-2}$] at (4){};
	\node[label=above:$\bar{T}_{k-1}$] at (5){};
	\node[label=above:$\bar{T}_{k}$] at (6){};
	\draw[double distance = 2pt] (0)--(1) (5)--(6);
	\draw (1)--(2) (4)--(5);
	\draw[dashed] (2) to (4);
\end{tikzpicture}
\end{matrix}\ .
\end{equation}
\noindent Pictorially, the generators of $\cB_k$ are identified with the braid diagrams
$$
{\def\TOP{2}\def\K{6}
\bar{T}_k=
\begin{matrix}
\begin{tikzpicture}[scale=.5]
\Pole[.15][0,2]
\Under[\K,0][\K+1.3,1]
\Pole[\K+.85][0,1][\K]
\Pole[\K+.85][1,2][\K]
\Over[\K+1.3,1][\K,2]
 \foreach \x in {1,...,5} {
	 \draw[thin] (\x,0) -- (\x,\TOP);
	 }
\Caps[.15,\K+.85][0,\TOP][\K]
\end{tikzpicture}\end{matrix},
	\qquad 
\bar{T}_0=
\begin{matrix}
\begin{tikzpicture}[scale=.5]
	\Pole[\K+.85][0,2][\K]
	\Pole[.15][0,1]
	\Over[1,0][-.3,1]
	\Under[-.3,1][1,2]
	\Pole[.15][1,2]
	 \foreach \x in {2,...,\K} {
	 \draw[thin] (\x,0) -- (\x,\TOP);
	 }
\Caps[.15,\K+.85][0,\TOP][\K]
\end{tikzpicture}\end{matrix},\qquad  \text{and}
}$$
\begin{equation}\label{LRpics}
{\def\TOP{2} \def\K{6}
\bar{T}_i=
\begin{matrix}
\begin{tikzpicture}[scale=.5]
	\Pole[\K+.85][0,2][\K]
	\Pole[.15][0,2]
	\Under[3,0][4,2]
	\Over[4,0][3,2]
	 \foreach \x in {1,2,5,\K} {
		 \draw[thin] (\x,0) -- (\x,\TOP);
		 }
	\Caps[.15,\K+.85][0,\TOP][\K]
	\Label[0,\TOP][3][\footnotesize $i$]
	\Label[0,\TOP][4][\footnotesize$i$+1]
\end{tikzpicture}\end{matrix}
\qquad \text{for $i=1, \dots, k-1$,}
}
\end{equation}
and the multiplication of braid diagrams is given by placing one diagram on top of another.

To make explicit the Schur-Weyl duality approach to representations of $\cB_k$ appearing in Section \ref{sec:braids-on-tensor-space}, it is useful to move the 
rightmost  pole to the left by conjugating by the diagram
\begin{equation}\label{eq:sigma}
\sigma = 
\begin{matrix}
\begin{tikzpicture}[scale=.5]
	\draw (-.7,1) .. controls (-.7,.15) .. (0,.15) -- (6,.15) .. controls (7,.15) .. (7,-1);
	\draw (-1,1) .. controls (-1,-.15) .. (0,-.15)-- (6,-.15) .. controls (7-.3,-.15) .. (7-.3,-1);
\Pole[.15][-1,1]
 \foreach \x in {1,...,6} {\draw[style=over] (\x,-1) -- (\x,1);}
\Tops[.15, -.85][1][6]
\Bottoms[.15, 6+.85][-1][6]
\end{tikzpicture}\end{matrix}\ .
\end{equation}
Define
\begin{equation}\label{DefnTiY1}
{\def\TOP{2}\def\K{6}
T_i= \sigma \bar{T}_i \sigma^{-1}=
\begin{matrix}
\begin{tikzpicture}[scale=.5]
	\Pole[-.85][0,2]
	\Pole[.15][0,2]
	\Under[3,0][4,2]
	\Over[4,0][3,2]
	 \foreach \x in {1,2,5,\K} {
		 \draw[thin] (\x,0) -- (\x,\TOP);
		 }
	\Caps[.15,-.85][0,\TOP][\K]
	\Label[0,\TOP][3][\footnotesize $i$]
	\Label[0,\TOP][4][\footnotesize$i$+1]
\end{tikzpicture}\end{matrix}\ , \qquad 
Y_1= \sigma \bar{T}_0 \sigma^{-1} =
\begin{matrix}
\begin{tikzpicture}[scale=.5]
	\Pole[-.85][0,2][\K]
	\Pole[.15][0,1]
	\Over[1,0][-.3,1]
	\Under[-.3,1][1,2]
	\Pole[.15][1,2]
	 \foreach \x in {2,...,\K} {
	 \draw[thin] (\x,0) -- (\x,\TOP);
	 }
	\Caps[.15,-.85][0,\TOP][\K]
\end{tikzpicture}\end{matrix}\ ,}
\end{equation}
and
\begin{equation}\label{DefnX1}
{\def\TOP{2} \def\K{6}
X_1 =   T_1^{-1} T_2^{-1} \cdots T^{-1}_{k-1} \sigma \bar{T}_k \sigma^{-1} T_{k-1} \cdots T_1 =   
\begin{matrix}
\begin{tikzpicture}[scale=.5]
	\Pole[-.85][0,1]
	\Over[1,0][-1.3,1]
	\Under[-1.3,1][1,2]
	\Pole[-.85][1,2]
	\Pole[.15][0,2]
	 \foreach \x in {2,...,\K}  {
		 \draw[thin] (\x,0) -- (\x,\TOP);
		 }
	\Caps[.15,-.85][0,\TOP][\K]
	\end{tikzpicture}\end{matrix}\ .
}\end{equation}
\noindent 
Define 
\begin{equation}\label{BraidMurphy}{
\def\TOP{2}\def\K{6}
Z_1=X_1Y_1\quad\hbox{and}\quad
		 Z_i = T_{i-1} T_{i-2} \cdots T_1 X_1 Y_1 T_1 \cdots T_{i-1} =
\begin{matrix}
\begin{tikzpicture}[scale=.5]
		\Pole[-.85][0,1]
		\Pole[.15][0,1]
		 \foreach \x in {1,2} {\draw[thin] (\x,0) -- (\x,1);}
		\Over[3,0][-1.3,1]
		\Under[-1.3,1][3,2]
		\Pole[-.85][1,2]
		\Pole[.15][1,2]
		\foreach \x in {1,2} { \draw[thin, style=over] (\x,1) -- (\x,2); }
		\foreach \x in {4,...,\K} {\draw[thin] (\x,0) -- (\x,\TOP);}
		\Caps[.15,-.85][0,\TOP][\K]
		\Label[0,\TOP][3][{\footnotesize $i$}]
\end{tikzpicture}
\end{matrix}\ ,}\end{equation}
for $i=2, \ldots, k$.

\begin{thm}\label{thm:BraidPres}
The two boundary braid group $\cB_k$ is presented in the following three ways, using the notation defined in \eqref{braidlengths}.

\begin{enumerate}[(a)]
\item $\cB_k$ is presented by generators $X_1, Y_1, Z_1, T_1, \dots, T_{k-1}$ and relations 
\begin{equation}\label{eq:BraidRels-X}
\begin{matrix}
\begin{tikzpicture}
	\draw (-.25,0) circle (2.5pt); \node(0) at (-.25,0){};
	\foreach \x in {1, 2, 4,5}{\draw (\x,0) circle (2.5pt); \node(\x) at (\x,0){};}
	\foreach \x in {1, 2}
		{\node[above] at (\x){$T_\x$};}
	\node[above] at (0){$X_1$};
	\node[above] at (4){$T_{k-2}$};
	\node[above] at (5){$T_{k-1}$};
	\draw[double distance = 2pt] (0)--(1);
	\draw (1)--(2) (4)--(5);
	\draw[dashed] (2) to (4);
\end{tikzpicture}\end{matrix}
\tag{a1}
\end{equation}
\begin{equation}\label{eq:BraidRels-Y}
\begin{matrix}
\begin{tikzpicture}
	\draw (-.25,0) circle (2.5pt); \node(0) at (-.25,0){};
	\foreach \x in {1, 2, 4,5}{\draw (\x,0) circle (2.5pt); \node(\x) at (\x,0){};}
	\foreach \x in {1, 2}
		{\node[above] at (\x){$T_\x$};}
	\node[above] at (0){$Y_1$};
	\node[above] at (4){$T_{k-2}$};
	\node[above] at (5){$T_{k-1}$};
	\draw[double distance = 2pt] (0)--(1);
	\draw (1)--(2) (4)--(5);
	\draw[dashed] (2) to (4);
\end{tikzpicture}\end{matrix}
\tag{a2}
\end{equation}
\begin{equation}\label{eq:BraidRels-Z}
\begin{matrix}
\begin{tikzpicture}
	\draw (-.25,0) circle (2.5pt); \node(0) at (-.25,0){};
	\foreach \x in {1, 2, 4,5}{\draw (\x,0) circle (2.5pt); \node(\x) at (\x,0){};}
	\foreach \x in {1, 2}
		{\node[above] at (\x){$T_\x$};}
	\node[above] at (0){$Z_1$};
	\node[above] at (4){$T_{k-2}$};
	\node[above] at (5){$T_{k-1}$};
	\draw[double distance = 2pt] (0)--(1);
	\draw (1)--(2) (4)--(5);
	\draw[dashed] (2) to (4);
\end{tikzpicture}\end{matrix}
\tag{a3}
\end{equation}
and 
\begin{equation}\label{eq:BraidRels-A}
Z_1=X_1Y_1.
\tag{a4}
\end{equation}

\item $\cB_k$ is presented by generators $X_1, Y_1, T_1, \dots, T_{k-1}$ and relations \eqref{eq:BraidRels-X}, \eqref{eq:BraidRels-Y}, 
and 
\begin{equation}\label{eq:BraidRels-B}
(T_1 X_1 T_1^{-1}) Y_1 = Y_1 (T_1 X_1 T_1^{-1}).	
\tag{b3}
\end{equation}

\item $\cB_k$ is presented by generators  $Z_1$, $\dots$, $Z_k$, $Y_1$, $T_1$, $\dots$, $T_{k-1}$, and relations \eqref{eq:BraidRels-Y}, 
\begin{equation}\label{eq:BraidRels-C1}
	Z_i Z_j = Z_j Z_i  \qquad \text{ for } i,j = 1, \dots, k,	\tag{c1}
\end{equation}
\begin{equation}\label{eq:BraidRels-C2}
	Y_1 Z_i = Z_i Y_1 \qquad \text{ for } i = 2, \dots, k, \text{ and } 	\tag{c2}
\end{equation}
\begin{equation}\label{eq:BraidRels-C3}
	T_i Z_j = Z_j T_i \qquad \text{ for } j \neq i, i+1, \text{ with } i=1, \dots, k-1, \text{ and } j = 1, \dots, k, 
	\tag{c3}
\end{equation}
and
\begin{equation}\label{eq:BraidRels-C4}
	Z_{i+1} = T_i Z_i T_i \qquad \text{ for } i=1, \dots, k-1.
	\tag{c4}
\end{equation}
\end{enumerate}
\end{thm}

\begin{proof}

With $\sigma$ as in \eqref{eq:sigma} let $T_k = \sigma \bar{T}_k \sigma^{-1}$, so that the original generators are the $\sigma$-conjugates of 
\begin{equation}
T_0, T_1, \dots, T_k.
\tag{o}
\label{eq:generators-o}
\end{equation}
Conjugate the relations in \eqref{Bdefrels} by $\sigma$ 
to rewrite them in the form
\begin{equation}
\begin{tikzpicture}
	\foreach \x in {0,1, 2, 4,5}{
		\draw (\x,0) circle (2.5pt);
		\node (\x) at (\x,0){};
		}
	\foreach \x in {1, 2}{
		\node[label=above:$T_\x$] at (\x,0){};
		\node[label=above:$T_{k-\x}$] at (6-\x,0){};}
	\node[label=above:$Y_1$] at (0){};
	\draw[double distance = 2pt] (0)--(1);
	\draw (1)--(2) (4)--(5);
	\draw[dashed] (2) to (4);
\end{tikzpicture},
\qquad
T_kT_{k-1}T_kT_{k-1}=T_{k-1}T_kT_{k-1}T_k,
\tag{o1}\label{eq:BraidRels-o1}
\end{equation}
\begin{equation}
T_kY_1=Y_1T_k,
\qquad\hbox{and}\qquad
T_kT_i = T_iT_k,\ \ \hbox{for $i=1,\ldots, k-2$.}
\tag{o2}\label{eq:BraidRels-o2}
\end{equation}

The conversions between the generators in presentations (a), (b), and (c)
are given in \eqref{DefnTiY1},  \eqref{DefnX1}, and \eqref{BraidMurphy}. 
For generators (a) and (b) in terms of generators \eqref{eq:generators-o}, the key relations are
$$Y_1=T_0, \qquad X_1 = T_1^{-1}\cdots T_{k-1}^{-1}T_kT_{k-1}\cdots T_1
\qquad\hbox{and}\qquad
T_k = T_{k-1}\cdots T_1 X_1 T_1^{-1}\cdots T_{k-1}^{-1}.
$$

\paragraph{Relations (a) from relations (b):}
Relation \eqref{eq:BraidRels-A} is the conversion from generators (b) to generators (a). The relations in  \eqref{eq:BraidRels-Z} then follow from
\begin{align*}
T_iZ_1 = T_iX_1Y_1=X_1T_iY_1=X_1Y_1T_i=Z_1T_i,
\qquad\hbox{for $i=2,\ldots, k-1$,}
\end{align*}
and
\begin{align*}
T_1Z_1T_1Z_1 
&= T_1X_1Y_1T_1X_1Y_1
= T_1X_1(Y_1T_1X_1T_1^{-1})T_1Y_1 
= T_1X_1(T_1X_1T_1^{-1}Y_1)T_1Y_1 \\
&=X_1T_1X_1T_1T_1^{-1}Y_1T_1Y_1
=X_1T_1X_1T_1^{-1}T_1Y_1T_1Y_1
=X_1T_1X_1T_1^{-1}Y_1T_1Y_1T_1 \\
&=X_1Y_1T_1X_1T_1^{-1}T_1Y_1T_1
=Z_1T_1Z_1T_1.
\end{align*}

\paragraph{Relations (b) from relations (a):}
Multiplying
\begin{align*}
T_1X_1(T_1X_1T_1^{-1}Y_1)T_1Y_1
&=X_1T_1X_1T_1T_1^{-1}Y_1T_1Y_1
=X_1T_1X_1T_1^{-1}T_1Y_1T_1Y_1 \\
&=X_1T_1X_1T_1^{-1}Y_1T_1Y_1T_1 
=X_1Y_1T_1X_1T_1^{-1}T_1Y_1T_1
=Z_1T_1Z_1T_1 \\
&=T_1Z_1T_1Z_1 
= T_1X_1Y_1T_1X_1Y_1
= T_1X_1(Y_1T_1X_1T_1^{-1})T_1Y_1
\end{align*}
on the left by $(T_1X_1)^{-1}$ and on the right by $(T_1Y_1)^{-1}$ gives
$T_1X_1T_1^{-1}Y_1 = Y_1T_1X_1T_1^{-1}$, establishing \eqref{eq:BraidRels-B}.

\paragraph{Relations (b) from relations (o):}
The pictorial computations 
$${\def\TOP{3} \def\K{6}
\begin{matrix}\begin{tikzpicture}[scale=.3]
	\Pole[-.85][0,1]
	\Over[1,0][-1.3,1]
	\Under[-1.3,1][1,2]
	\Pole[-.85][1,\TOP]
	\Pole[.15][0,\TOP]
	\foreach \x in {2,5,\K}  { \draw[thin] (\x,0) -- (\x,\TOP);}
	\foreach \x in {3,4}  { \draw[thin] (\x,0) -- (\x,2);}  \draw[thin] (1,2) -- (1,3);
	\Cross[3,3][4,2]
	\Caps[-.85,.15][0,\TOP][\K]
	\Label[0,\TOP][3][\tiny $i$]
\end{tikzpicture}\end{matrix}
=
\begin{matrix}\begin{tikzpicture}[scale=.3]
	\Pole[-.85][0,2]
	\Over[1,1][-1.3,2]
	\Under[-1.3,2][1,3]
	\Pole[-.85][2,\TOP]
	\Pole[.15][0,\TOP]
	\foreach \x in {2,5,\K}  { \draw[thin] (\x,0) -- (\x,\TOP);}
	\foreach \x in {3,4}  { \draw[thin] (\x,1) -- (\x,3);}  \draw[thin] (1,0) -- (1,1);
	\Cross[3,1][4,0]
	\Caps[-.85,.15][0,\TOP][\K]
	\Label[0,\TOP][3][\tiny $i$]
\end{tikzpicture}\end{matrix}}\ , \qquad 
{\def\TOP{4} \def\K{6}
\begin{matrix}\begin{tikzpicture}[scale=.3]
	\Under[-1.3,1][2,2]
	\Under[-.3,3][1,4]
	\Pole[-.85][0,\TOP]
	\Over[2,0][-1.3,1]
	\Pole[.15][0,\TOP]
	\Over[1,2][-.3,3]
	\draw[over](1,0)--(1,2) (2,2)--(2,4);
	\foreach \x in {3,4,5,\K}  { \draw[thin] (\x,0) -- (\x,\TOP);}
	\Caps[-.85,.15][0,\TOP][\K]
\end{tikzpicture}\end{matrix}
=
\begin{matrix}\begin{tikzpicture}[scale=.3]
	\Under[-1.3,3][2,4]
	\Under[-.3,1][1,2]
	\Pole[-.85][0,\TOP]
	\Over[2,2][-1.3,3]
	\Pole[.15][0,\TOP]
	\Over[1,0][-.3,1]
	\draw[over](1,2)--(1,\TOP) (2,0)--(2,2);
	\foreach \x in {3,4,5,\K}  { \draw[thin] (\x,0) -- (\x,\TOP);}
	\Caps[-.85,.15][0,\TOP][\K]
\end{tikzpicture}\end{matrix}}\ , \qquad \text{ and }
$$
$${\def\TOP{4} \def\K{6}
\begin{matrix}\begin{tikzpicture}[scale=.3]
	\Under[-1.3,1][2,2]
	\Under[-1.3,3][1,4]
	\Pole[-.85][0,\TOP]
	\draw[over](1,0)--(1,2) (2,2)--(2,4);
	\Over[2,0][-1.3,1]
	\Over[1,2][-1.3,3]
	\Pole[.15][0,\TOP]
	\foreach \x in {3,4,5,\K}  { \draw[thin] (\x,0) -- (\x,\TOP);}
	\Caps[-.85,.15][0,\TOP][\K]
\end{tikzpicture}\end{matrix}
=
\begin{matrix}\begin{tikzpicture}[scale=.3]
	\Under[-1.3,3][2,4]
	\Under[-1.3,1][1,2]
	\Pole[-.85][0,\TOP]
	\draw[over](1,2)--(1,\TOP) (2,0)--(2,2);
	\Over[2,2][-1.3,3]
	\Over[1,0][-1.3,1]
	\Pole[.15][0,\TOP]
	\foreach \x in {3,4,5,\K}  { \draw[thin] (\x,0) -- (\x,\TOP);}
	\Caps[-.85,.15][0,\TOP][\K]
\end{tikzpicture}\end{matrix}}$$
show that $X_1T_i=T_iX_1$ for $i=1,2,\ldots, k-1$, 
$Y_1T_1X_1T_1^{-1}=T_1X_1T_1^{-1}Y_1$, and  $X_1T_1X_1T_1=T_1X_1T_1X_1$.
Hence the relations \eqref{eq:BraidRels-X} and \eqref{eq:BraidRels-Y} follow from the relations in \eqref{eq:BraidRels-o1} and \eqref{eq:BraidRels-o2}.

\paragraph{Relations (o) from relations (b):}
The first set of relations in \eqref{eq:BraidRels-o1} are the same as the relations in \eqref{eq:BraidRels-Y}. 
Let $A=T_{k-1}\cdots T_1$ and $B=T_{k-1}\cdots T_2$.  Since $X_1$ commutes with $T_i$ for $i=2, \ldots, k-1$, then
$BX_1B^{-1}=X_1$ so that
$${\def\TOP{8} \def\K{6}
ABX_1B^{-1}A^{-1} = 
\begin{matrix}\begin{tikzpicture}[scale=.3]
	\Over[6,8][1,6] 	\Over[5,8][6,7]
	\Under[6,7][2,5] 
	\Under[6,0][1,2] \Over[6,1][5,0]
	\Under[6,1][2,3] 
	\Under[-1.3,4][1,5]
	\Pole[-.85][0,\TOP]
	\Over[1,3][-1.3,4]
	\Pole[.15][0,\TOP]
	\foreach \x in {1,2,3,4}  { 
		\draw[over, rounded corners] (\x,8) to (\x+2,5) to (\x+2,3) to (\x,0);}
	\draw[thin] (1,5)--(1,6) (1,3)--(1,2) (2,5)--(2,3);
	\Caps[-.85,.15][0,\TOP][\K]
\end{tikzpicture}\end{matrix}
=
\begin{matrix}\begin{tikzpicture}[scale=.3]
	\Under[-1.3,2][6,4]
	\Pole[-.85][0,4]
	\Over[-1.3,2][6,0]
	\Pole[.15][0,4]
	\foreach \x in {1,...,5} { \draw[over] (\x,0) to (\x,4);}
	\Caps[-.85,.15][0,4][\K]
\end{tikzpicture}\end{matrix}=T_k,}$$
and
$$ABT_1B^{-1}A^{-1}=
{\def\TOP{7} \def\K{6}
\begin{matrix}\begin{tikzpicture}[scale=.3]
	\Over[6,7][1,5] 	\Over[5,7][6,6]
	\Under[6,6][2,4] 
	\Cross[1,4][2,3]
	\Under[6,0][1,2] \Over[6,1][5,0]
	\Under[6,1][2,3] 
	\foreach \x in {1,2,3,4}  { 
		\draw[over, rounded corners] (\x,7) to (\x+2,4) to (\x+2,3) to (\x,0);}
	\draw[thin] (1,5)--(1,4) (1,3)--(1,2) ;
	\Pole[-.85][0,\TOP]
	\Pole[.15][0,\TOP]
	\Caps[-.85,.15][0,\TOP][\K]
\end{tikzpicture}\end{matrix}}
=
{\def\TOP{3} \def\K{6}
\begin{matrix}\begin{tikzpicture}[scale=.3]
	\Cross[5,3][6,2] \Cross[5,2][6,1] \Cross[6,1][5,0] 
	\foreach \x in {1,...,4} { \draw[over] (\x,0) to (\x,\TOP);}
	\Pole[-.85][0,\TOP] \Pole[.15][0,\TOP]
	\Caps[-.85,.15][0,\TOP][\K]
\end{tikzpicture}\end{matrix}}=
{\def\TOP{3} \def\K{6}
\begin{matrix}\begin{tikzpicture}[scale=.3]
	\Cross[5,2][6,1] 
	\draw[thin] (5,3)--(5,2) (5,1)--(5,0) (6,3)--(6,2) (6,1)--(6,0);
	\foreach \x in {1,...,4} { \draw[over] (\x,0) to (\x,\TOP);}
	\Pole[-.85][0,\TOP] \Pole[.15][0,\TOP]
	\Caps[-.85,.15][0,\TOP][\K]
\end{tikzpicture}\end{matrix}}
=T_{k-1}.$$
 Thus, by conjugation by $AB$, the relation 
 $X_1T_1X_1T_1=T_1X_1T_1X_1$ becomes $T_kT_{k-1}T_kT_{k-1}=T_{k-1}T_kT_{k-1}T_k$, establishing the
 second relation in \eqref{eq:BraidRels-o1}. 
For $i=1,\ldots, k-2$,
\begin{align*}
T_iT_k 
&= T_iT_{k-1}\cdots T_1X_1T_1^{-1}\cdots T_k^{-1} 
= T_{k-1}\cdots T_{i+2}T_iT_{i+1}T_i\cdots T_1X_1T_1^{-1}\cdots T_k^{-1} \\
&= T_{k-1}\cdots T_{i+2}T_{i+1}T_iT_{i+1}T_{i-1}\cdots T_1X_1T_1^{-1}\cdots T_k^{-1} \\
&= T_{k-1}\cdots T_1X_1T_1^{-1}\cdots T_{i-1}^{-1}T_{i+1}T_i^{-1}T_{i+1}^{-1}\cdots T_k^{-1} \\
&= T_{k-1}\cdots T_1X_1T_1^{-1}\cdots T_{i-1}^{-1}T_i^{-1}T_{i+1}^{-1}T_iT_{i+2}^{-1}\cdots T_k^{-1} \\
&= T_{k-1}\cdots T_1X_1T_1^{-1}\cdots \cdots T_k^{-1}T_i = T_kT_i.
\end{align*}
Similarly, \eqref{eq:BraidRels-B} gives 
\begin{align*}
Y_1 T_k 
&=  Y_1T_{k-1} \cdots T_2T_1X_1T_1^{-1} T_2^{-1} \cdots T^{-1}_{k-1} 
=  T_{k-1} \cdots T_2 (Y_1T_1X_1T_1^{-1}) T_2^{-1} \cdots T^{-1}_{k-1} \\
&= T_{k-1} \cdots T_2 (T_1X_1T_1^{-1}Y_1) T_2^{-1} \cdots T^{-1}_{k-1} 
= T_{k-1} \cdots T_2 T_1X_1T_1^{-1} T_2^{-1} \cdots T^{-1}_{k-1}Y_1 
= T_k Y_1,
\end{align*}
giving the relations in \eqref{eq:BraidRels-o2}.

\paragraph{Relations (c) from relations (o):}
The first set of relations in \eqref{eq:BraidRels-o1} are the same as the relations in \eqref{eq:BraidRels-Y}. Relations \eqref{eq:BraidRels-C4} are exactly the definitions in the second part of \eqref{BraidMurphy}.
The pictorial computation
\begin{equation*}\label{Zscommute}{
\def\TOP{2}\def\K{6}
Z_jZ_i=
\begin{matrix}
\begin{tikzpicture}[scale=.3]
		\Pole[-.85][0,1]
		\Pole[.15][0,1]
		 \foreach \x in {1} { \draw[thin] (\x,0) -- (\x,1); }
		\Over[2,0][-1.3,1]
		\Under[-1.3,1][2,2]
		\Pole[-.85][1,2]
		\Pole[.15][1,2]
		 \foreach \x in {1} { \draw[thin, style=over] (\x,1) -- (\x,2); }
		 \foreach \x in {3,...,\K} { \draw[thin] (\x,0) -- (\x,\TOP); }
	\pgftransformyshift{\TOP cm}
		\Pole[-.85][0,1]
		\Pole[.15][0,1]
		\foreach \x in {1,2,3} {\draw[thin] (\x,0) -- (\x,1); }
		\Over[4,0][-2,1]
		\Under[-2,1][4,2]
		\Pole[-.85][1,2]
		\Pole[.15][1,2]
		 \foreach \x in {1,2,3} { \draw[thin, style=over] (\x,1) -- (\x,2); }
		 \foreach \x in {5,...,\K} { \draw[thin] (\x,0) -- (\x,\TOP); }
		\Caps[-.85,.15][-\TOP,\TOP][\K]
		\Label[-\TOP,\TOP][4][{\tiny $j$}]
		\Label[-\TOP,\TOP][2][{\tiny $i$}]
\end{tikzpicture}\end{matrix}
=
\begin{matrix}
\begin{tikzpicture}[scale=.3]
		\Pole[-.85][0,1]
		\Pole[.15][0,1]
		 \foreach \x in {1,2,3} {\draw[thin] (\x,0) -- (\x,1); }
		\Over[4,0][-2,1]
		\Under[-2,1][4,2]
		\Pole[-.85][1,2]
		\Pole[.15][1,2]
		\foreach \x in {1,2,3} {\draw[thin, style=over] (\x,1) -- (\x,2); }
		 \foreach \x in {5,...,\K} { \draw[thin] (\x,0) -- (\x,\TOP); }
	\pgftransformyshift{\TOP cm}
		\Pole[-.85][0,1]
		\Pole[.15][0,1]
		 \foreach \x in {1} {\draw[thin] (\x,0) -- (\x,1);}
		\Over[2,0][-1.3,1]
		\Under[-1.3,1][2,2]
		\Pole[-.85][1,2]
		\Pole[.15][1,2]
		 \foreach \x in {1} { \draw[thin, style=over] (\x,1) -- (\x,2);}
		 \foreach \x in {3,...,\K} {\draw[thin] (\x,0) -- (\x,\TOP);}
		\Caps[-.85,.15][-\TOP,\TOP][\K]
		\Label[-\TOP,\TOP][4][{\tiny $j$}]
		\Label[-\TOP,\TOP][2][{\tiny $i$}]
\end{tikzpicture}
\end{matrix}= Z_iZ_j}
\end{equation*}
give relations \eqref{eq:BraidRels-C1}. Similarly, pictorial computations readily show that $Y_1 Z_i = Z_i Y_1$ for $i>1$ and $T_i Z_j = Z_j T_i$ for $i \neq j, j+1$, proving relations \eqref{eq:BraidRels-C2} and \eqref{eq:BraidRels-C3}. 

\paragraph{Generators (o) from generators (c):}  The key formula for the generator $T_k$ is 
\begin{align*}
T_k 
&= T_{k-1}\cdots T_1(T_1^{-1}\cdots T_{k-1}^{-1}T_kT_{k-1}\cdots T_1)Y_1(T_1\cdots T_{k-1})(T_{k-1}^{-1}\cdots T_1^{-1})Y_1^{-1}(T_1^{-1}\cdots T_{k-1}^{-1}) \\
&= (T_{k-1}\cdots T_1) X_1Y_1(T_1\cdots T_{k-1})T_{s_{\varphi}} =Z_kT_{s_\varphi}^{-1},
\end{align*}
where
$$T_{s_\varphi}=T_{k-1}T_{k-2}\cdots T_1Y_1T_1\cdots T_{k-2}T_{k-1}
	={\def\TOP{2}\def\K{6}
		\begin{matrix}
		\begin{tikzpicture}[scale=.4]
			\Caps[-.85,.15][0,\TOP][\K]
			\Under[-.3,1][\K,\TOP]
			\Pole[-.85][0,\TOP]\Pole[.15][0,\TOP]
			\foreach \x in {1,...,5} { \draw[over] (\x,0) -- (\x,\TOP); }
			\Over[\K,0][-.3,1]
		\end{tikzpicture}
		\end{matrix}}.$$

\paragraph{Relations (o) from relations (c):}
The first set of relations in \eqref{eq:BraidRels-o1} are the same as the relations in \eqref{eq:BraidRels-Y}.
The relations
\begin{equation}
T_{s_\varphi}Y_1 = Y_1T_{s_\varphi} 
\qquad \text{ and } \qquad 
T_{s_\varphi}T_i = T_i T_{s_\varphi}, \qquad \text{ for } i = 1, \dots, k-2,
\label{pastTsphi}
\end{equation}
are verified pictorially by
$${\def\TOP{2}\def\K{6}
	\begin{matrix}
	\begin{tikzpicture}[scale=.3]
		\Caps[-.85,.15][0,2*\TOP][\K]			
			\Under[-.3,3][\K,2*\TOP]
			\Under[-.3,1][1,2]
			\Pole[-.85][0,2*\TOP]\Pole[.15][0,2*\TOP]
			\foreach \x in {1,...,5} { \draw[over] (\x,2*\TOP) -- (\x,\TOP); }
			\foreach \x in {2,3,4,5,6} { \draw[over] (\x,0) -- (\x,\TOP); }
			\Over[-.3,3][\K,2]
			\Over[-.3,1][1,0]
	\end{tikzpicture}
	\end{matrix}=
	\begin{matrix}
	\begin{tikzpicture}[scale=.3]
		\Caps[-.85,.15][0,2*\TOP][\K]
			\Under[-.3,1][\K,\TOP]
			\Under[-.3,3][1,2*\TOP]
			\Pole[-.85][0,2*\TOP]\Pole[.15][0,2*\TOP]
			\foreach \x in {1,...,5} { \draw[over] (\x,0) -- (\x,\TOP); }
			\foreach \x in {2,3,4,5,6} { \draw[over] (\x,2*\TOP) -- (\x,\TOP); }
			\Over[\K,0][-.3,1]
			\Over[1,\TOP][-.3,3]
	\end{tikzpicture}
	\end{matrix}
	\qquad \text{ and } \qquad 
	\begin{matrix}
	\begin{tikzpicture}[scale=.3]
		\Caps[-.85,.15][0,2*\TOP][\K]
			\Under[-.3,3][\K,2*\TOP]
			\Pole[-.85][0,2*\TOP]\Pole[.15][0,2*\TOP]
			\foreach \x in {1,...,5} { \draw[over] (\x,2*\TOP) -- (\x,\TOP); }
			\foreach \x in {1,2,5,6} { \draw[over] (\x,0) -- (\x,\TOP); }
			\Over[-.3,3][\K,2]
			\Cross[3,\TOP][4,0]
	\end{tikzpicture}
	\end{matrix}=
	\begin{matrix}
	\begin{tikzpicture}[scale=.3]
		\Caps[-.85,.15][0,2*\TOP][\K]
			\Under[-.3,1][\K,\TOP]
			\Pole[-.85][0,2*\TOP]\Pole[.15][0,2*\TOP]
			\foreach \x in {1,...,5} { \draw[over] (\x,0) -- (\x,\TOP); }
			\foreach \x in {1,2,5,6} { \draw[over] (\x,2*\TOP) -- (\x,\TOP); }
			\Over[\K,0][-.3,1]
			\Cross[3,2*\TOP][4,\TOP]
	\end{tikzpicture}
	\end{matrix}}.$$
or by direct computation using the relations in \eqref{eq:BraidRels-Y}.

By \eqref{DefnX1} and \eqref{BraidMurphy}, $Z_k = T_k T_{s_\varphi}$ and, 
by \eqref{eq:BraidRels-C3} and \eqref{eq:BraidRels-C2} respectively,
\begin{align}
	T_kT_i &= Z_kT_{s_\varphi}^{-1}T_i = Z_kT_iT_{s_\varphi}^{-1}=T_iZ_kT_{s_\varphi}^{-1} = T_iT_k,
	\qquad\hbox{for $i=1,\ldots, k-2$, and} \label{pastTk} \\
	T_kY_1 &= Z_kT_{s_\varphi}^{-1}Y_1 = Z_kY_1T_{s_\varphi}^{-1}=Y_1Z_kT_{s_\varphi}^{-1}=Y_1T_k,
\nonumber\end{align}
which proves the relations in \eqref{eq:BraidRels-o2}.

By the relations in \eqref{pastTk} and the second set of relations in \eqref{pastTsphi},
	$$(T_{k-1}^{-1}T_{s_\varphi}T_{k-1}^{-1})T_k
		=T_k(T_{k-1}^{-1}T_{s_\varphi}T_{k-1}^{-1})
	\qquad\hbox{and}\qquad
	(T_{k-1}^{-1}T_{s_\varphi}T_{k-1}^{-1})T_{s_\varphi}
		=T_{s_\varphi}(T_{k-1}^{-1}T_{s_\varphi}T_{k-1}^{-1}),$$
so that 
	$(T_{k-1}^{-1}T_{s_\varphi}T_{k-1}^{-1})(T_kT_{s_\varphi})
		=(T_kT_{s_\varphi})(T_{k-1}^{-1}T_{s_\varphi}T_{k-1}^{-1}).$
Using these and the equality
	$$T_{k-1} Z_{k} Z_{k-1}= T_{k-1} Z_{k-1} Z_k = Z_k T_{k-1}^{-1}Z_k = Z_kZ_{k-1}T_{k-1},$$
we have
\begin{align*}
	T_{k-1}Z_kZ_{k-1} 
	&= T_{k-1}Z_k(T_{k-1}^{-1}Z_kT_{k-1}^{-1})
		= T_{k-1}(T_kT_{s_\varphi})T_{k-1}^{-1}(T_kT_{s_\varphi})T_{k-1}^{-1} \\
	&= T_{k-1}T_kT_{k-1}(T_{k-1}^{-1}T_{s_\varphi}T_{k-1}^{-1})T_kT_{s_\varphi}T_{k-1}^{-1} 
		= (T_{k-1}T_kT_{k-1}T_k)(T_{k-1}^{-1}T_{s_\varphi}T_{k-1}^{-1}T_{s_\varphi}T_{k-1}^{-1}) \\
	=Z_kZ_{k-1}T_{k-1} 
	&= Z_k(T_{k-1}^{-1}Z_kT_{k-1}^{-1})T_{k-1} = (T_kT_{s_\varphi})T_{k-1}^{-1}(T_kT_{s_\varphi}) 
		= T_kT_{k-1}(T_{k-1}^{-1}T_{s_\varphi}T_{k-1}^{-1})(T_kT_{s_\varphi}) \\
	&= T_kT_{k-1}(T_kT_{s_\varphi})(T_{k-1}^{-1}T_{s_\varphi}T_{k-1}^{-1}) 
		= (T_k T_{k-1}T_kT_{k-1})(T_{k-1}^{-1}T_{s_\varphi}T_{k-1}^{-1}T_{s_\varphi}T_{k-1}^{-1}).
\end{align*}
Multiplying on the right by
	$(T_{k-1}^{-1}T_{s_\varphi}T_{k-1}^{-1}T_{s_\varphi}T_{k-1}^{-1})^{-1}$
gives $T_kT_{k-1}T_kT_{k-1} = T_{k-1}T_kT_{k-1}T_k$, establishing the last relation in \eqref{eq:BraidRels-o1}.
\end{proof}
\medskip

If
\begin{equation}\label{DefnPhalf}
{\def\TOP{1.5} \def\K{6}
P^{1/2} = 
\begin{matrix}
\begin{tikzpicture}[scale=.5]
		\PoleTwist[0,\TOP]
		\foreach \x in {1,...,\K} {\draw[thin] (\x,0) -- (\x,\TOP);}
		\Caps[-.85,.15][0,\TOP][\K]
\end{tikzpicture}\end{matrix}}
\end{equation}
then 
\begin{equation}\label{PYP}
{\def\TOP{4} \def\K{6}
P^{1/2}Y_1P^{-1/2} = \begin{matrix}
\begin{tikzpicture}[scale=.3]
		\Caps[-.85,.15][0,\TOP][\K]
		\PoleTwist[.67*\TOP,\TOP]
		\foreach \x in {2,...,\K} {\draw[thin] (\x,0) -- (\x,\TOP);}
		{\draw[thin] (1,0) -- (1,.2*\TOP);}
		{\draw[thin] (1,.8*\TOP) -- (1,\TOP);}
		\Pole[-.85][.3*\TOP,.67*\TOP]
		\PoleTwist[0,.33*\TOP]
		\Under[1,.8*\TOP][-.4,.5*\TOP]
		\Pole[.15][.33*\TOP, .67*\TOP]
		\Over[-.4,.5*\TOP][1,.2*\TOP]
		\end{tikzpicture}\end{matrix}
	 = \begin{matrix}
\begin{tikzpicture}[scale=.3]
		\Caps[-.85,.15][0,\TOP][\K]
		\Pole[-.85][0,.5*\TOP]
		\Under[-.3,.83*\TOP][1,.67*\TOP]
		\Under[1,.67*\TOP][-1.3,.5*\TOP]
		\Under[1,.33*\TOP][-.3,.17*\TOP]
		\Pole[-.85][\TOP,.5*\TOP]
		\Over[-1.3,.5*\TOP][1,.33*\TOP]
		\Pole[.15][\TOP,0]
		\Over[1,\TOP][-.3,.83*\TOP]
		\Over[1,0][-.3,.17*\TOP]
		\foreach \x in {2,...,\K} {\draw[thin] (\x,0) -- (\x,\TOP);}
\end{tikzpicture}\end{matrix}= Y_1^{-1}X_1Y_1}
\end{equation}
and
\begin{equation}\label{PXP}
{\def\TOP{4} \def\K{6}
P^{1/2}X_1P^{-1/2} = \begin{matrix}
\begin{tikzpicture}[scale=.3]
		\Caps[-.85,.15][0,\TOP][\K]
		\PoleTwist[.67*\TOP,\TOP]
		\foreach \x in {2,...,\K} {\draw[thin] (\x,0) -- (\x,\TOP);}
		{\draw[thin] (1,0) -- (1,.33*\TOP);}
		{\draw[thin] (1,.7*\TOP) -- (1,\TOP);}
		\Pole[-.85][.5*\TOP,.33*\TOP]
		\PoleTwist[0,.33*\TOP]
		\Under[1,.7*\TOP][-1.3,.5*\TOP]
		\Over[-1.3,.5*\TOP][1,.33*\TOP]
		\Pole[.15][.33*\TOP, .67*\TOP]
		\Pole[-.85][.5*\TOP,.67*\TOP]
\end{tikzpicture}\end{matrix}= 
\begin{matrix}
\begin{tikzpicture}[scale=.3]
		\Caps[-.85,.15][0,\TOP][\K]
		\Pole[-.85][0,\TOP]
		\Under[1,.8*\TOP][-.3,.5*\TOP]
		\Pole[.15][\TOP,.5*\TOP]
		\Pole[.15][.5*\TOP,0]
		\Over[-.3, .5*\TOP][1,.2*\TOP]
		\foreach \x in {2,...,\K} {\draw[thin] (\x,0) -- (\x,\TOP);}
		{\draw[thin] (1,0) -- (1,.2*\TOP);}
		{\draw[thin] (1,.8*\TOP) -- (1,\TOP);}
		\end{tikzpicture}\end{matrix}= Y_1}
\end{equation}
Following these pictorial computations, the \emph{extended affine braid group} is the group $\cB_k^{\mathrm{ext}}$ 
generated by $\cB_k$ and $P$ with the additional relations
\begin{align}
PX_1P^{-1} = Z_1^{-1}X_1Z_1, \qquad
PY_1P^{-1} = Z_1^{-1}Y_1Z_1, 
\label{Pcomm1} 
\\
PZ_1P^{-1} = Z_1,\qquad \text{ and } \qquad PT_iP^{-1} = T_i\ \hbox{for $i=1, \ldots, k-1$.}
\label{Pcomm2} 
\end{align}
Note that the element 
\begin{equation}\label{eq:Z0iscentral}
Z_0 = P Z_1 \cdots Z_k \quad\hbox{is central in $\cB_k^\mathrm{ext}$}
\tag{c0}
\end{equation}
since the group $\cB_k^{\mathrm{ext}}$ is a subgroup of the braid group on $k+2$ strands,
and
$Z_0$ is the generator of the center of the braid group on $k+2$ strands (see \cite[Theorem 4.2]{GM}).
So
$$\hbox{if}\ \cD = \{ Z_0^{j}\ |\ j\in \ZZ\}
\qquad\hbox{then}\qquad
\cB_k^{\mathrm{ext}} = \cD \times \cB_k,
\quad\hbox{with $\cD \cong \ZZ$.}
$$

\subsection{The two boundary Hecke algebra $H_k^{\mathrm{ext}}$}\label{Heckedefnsubsection}

In this subsection we define the two boundary Hecke algebras and relate it to the presentation of the 
affine Hecke algebra of type C that is found, for example, in 
\cite[Proposition 3.6]{Lu1} and \cite[(4.2.4)]{Mac2}.

Fix $a_1, a_2, b_1, b_2, t ^{\frac12}\in \CC^\times$. 
The \emph{extended two boundary Hecke algebra} $H_k^\ext$ is the quotient of $\cB_k^\ext$ by the relations 
\begin{equation}
(X_1 - a_1)(X_1-a_2) = 0,
\quad
(Y_1 - b_1)(Y_1-b_2) = 0,
\quad\hbox{and}\quad 
(T_i - t^{\half})(T_i+ t^{-\half}) = 0,
\label{Heckedefn} \tag{h}
\end{equation}
for $i = 1, \dots, k-1$. 
Let
\begin{equation}\label{eq:ab-to-t0tk}
	t_k^{\frac12} =  a_1^{\frac12}(-a_2)^{-\frac12}
\quad\hbox{and}\quad
t_0^{\frac12} = b_1^{\frac12}(-b_2)^{-\frac12}.
\end{equation}
With $Z_i\in H_k^{\mathrm{ext}}$ as in \eqref{BraidMurphy}, define
\begin{equation}
T_0 = b_1^{-\frac12}(-b_2)^{-\frac12} Y_1, \qquad W_i = -(a_1a_2b_1b_2)^{-\frac12}Z_i \quad 
\hbox{for $i = 1, \dots, k$, and} \label{eq:normalized_gens}
\end{equation}
\begin{equation}
W_0 = PW_1\cdots W_k = (-1)^k(a_1a_2b_1b_2)^{-\frac{k}2}PZ_1\cdots Z_k = (-1)^k(a_1a_2b_1b_2)^{-\frac{k}2}Z_0.
\end{equation} 
Then
\begin{equation}
X_1 = Z_1Y_1^{-1} = a_1^{\frac12}(-a_2)^{\frac12}W_1T_0^{-1}.
\label{X1renorm}
\end{equation}

\begin{thm}\label{thm:2bdryHeckePresentation} Fix $t_0, t_k, t \in \CC^\times$ and use notations for relations as defined in \eqref{braidlengths}. The
extended affine Hecke algebra $H^\ext_k$ defined in 
\eqref{Heckedefn}
is presented by generators, 
$T_0$, $T_1$, \dots, $T_{k-1}$, $W_0$, $W_1$, \dots, $W_k$ and relations
\begin{equation}
W_0\in Z(H_k^{\mathrm{ext}}), \qquad\qquad
\begin{matrix}\begin{tikzpicture}
	\foreach \x in {0,1, 2, 4,5}{
		\draw (\x,0) circle (2.5pt);
		\node (\x) at (\x,0){};
		}
	\foreach \x in {1, 2}{
		\node[label=above:$T_\x$] at (\x,0){};
		\node[label=above:$T_{k-\x}$] at (6-\x,0){};}
	\node[label=above:$T_0$] at (0){};
	\draw[double distance = 2pt] (0)--(1);
	\draw (1)--(2) (4)--(5);
	\draw[dashed] (2) to (4);
\label{PresBraidRelsB1}\tag{B1}
\end{tikzpicture}\end{matrix};
\end{equation}
\begin{equation}
W_iW_j = W_j W_i, \qquad\hbox{for $i,j= 0, 1,\ldots, k$;}
\label{WcommuteB2}\tag{B2}
\end{equation}
\begin{equation}
T_0W_j = W_j T_0,\quad\hbox{for $j\ne 1$;}
\label{T0WcommuteB3}\tag{B3}
\end{equation}
\begin{equation}
T_iW_j = W_j T_i\ \hbox{for $i=1, \ldots, k-1$ and $j=1,\ldots, k$ with $j\ne i, i+1$;}
\label{TiWcommuteB4}\tag{B4}
\end{equation}
\begin{equation}
(T_0- t_0^{\frac12})(T_0 + t_0^{-\frac12}) = 0, \ 
\quad \hbox{and}\quad 
(T_i - t^{\frac12})(T_i +  t^{-\frac12})=0 \ \hbox{for $i = 1, \dots, k-1$.}
\label{quadraticT0andTi}\tag{H}
\end{equation}
For $i=1,\ldots, k-1$,  
\begin{equation}
T_iW_i = W_{i+1}T_i + (t^{\frac12} - t^{-\frac12})\frac{W_i - W_{i+1}}{1-W_iW_{i+1}^{-1}},
\qquad
T_iW_{i+1} = W_iT_i +(t^{\frac12}-t^{-\frac12})\frac{W_{i+1}-W_i}{1-W_iW_{i+1}^{-1}}, 
\label{TipastWi}\tag{C1}
\end{equation}
\begin{align}\text{and}\qquad
T_0W_1 &= W_1^{-1}T_0 + \left((t_0^{\frac12}-t_0^{-\frac12})  + (t_k^{\frac12} -t_k^{-\frac12})W_1^{-1}\right) \frac{W_1 - W_1^{-1}}{1-W_1^{-2}}.
\label{T0pastW1}\tag{C2}
\end{align}
\end{thm}

\begin{proof} The conversion between the different sets of generators of $H_k^{\mathrm{ext}}$ is
provided by \eqref{eq:normalized_gens}.

\paragraph{Equivalence between (\ref{eq:Z0iscentral}--\ref{eq:BraidRels-C4})and the second
and third relations of (\ref{Heckedefn}) with the relations (\ref{PresBraidRelsB1}--\ref{TiWcommuteB4}) and (\ref{quadraticT0andTi}).}
Since $T_0$ and $Y_1$ differ by a constant, and $W_i$ and $Z_i$ differ by a constant, 
the relations in (\ref{eq:Z0iscentral}--\ref{eq:BraidRels-C4}) are equivalent
to the relations in (\ref{PresBraidRelsB1}--\ref{TiWcommuteB4}), respectively.
Since
\begin{align*}
0 &= (Y_1 - b_1)(Y_1 - b_2) 
=b_1^{\frac12}(-b_2)^{\frac12}(T_0 - b_1^{\frac12}(-b_2)^{-\frac12})
b_1^{\frac12}(-b_2)^{\frac12}(T_0 + b_1^{-\frac12}(-b_2)^{\frac12})  \\
&= -b_1 b_2 (T_0- t_0^{\frac12})(T_0 + t_0^{-\frac12}),
\end{align*}
the relations \eqref{quadraticT0andTi} are equivalent to the second and third relations in
\eqref{Heckedefn}.

\smallskip\noindent
\paragraph{Relations (\ref{TipastWi}--\ref{T0pastW1}) from relations (\ref{eq:Z0iscentral}--\ref{eq:BraidRels-C4}) and (\ref{Heckedefn}):}
From \eqref{BraidMurphy} and \eqref{eq:normalized_gens}, $W_{i+1}=T_iW_iT_i$,
and by the last relation in \eqref{Heckedefn}, 
$T_i^{-1} = T_i - (t^{\frac12} - t^{-\frac12})$. So
\begin{align*}
T_iW_i 
&
= W_{i+1}T_i^{-1} 
= W_{i+1}(T_i - (t^{\frac12} - t^{-\frac12})) 
= W_{i+1}T_i + (t^{\frac12} - t^{-\frac12})\frac{W_i - W_{i+1}}{1-W_iW_{i+1}^{-1}} \quad\hbox{and} \\
T_iW_{i+1} &= T_i^2W_iT_i = (t^{\frac12}-t^{-\frac12})W_{i+1}+W_iT_i 
=W_iT_i+ (t^{\frac12}-t^{-\frac12})\frac{W_{i+1}-W_i}{1-W_iW_{i+1}^{-1}},
\end{align*}
which establishes the relations in \eqref{TipastWi}. 

\smallskip\noindent
By the first relation in \eqref{Heckedefn},
$X_1^{-1} = -a_1^{-1}a_2^{-1}X_1+(a_1^{-1}+a_2^{-1})$.
Since $W_1=a_1^{-\frac12}(-a_2)^{-\frac12}X_1T_0$ and
$T_0 - T_0^{-1} = t_0^{\frac12} - t_0^{-\frac12} $,
\begin{align*}
T_0 W_1 - W_1^{-1} T_0
&=a_1^{-\frac12} (-a_2)^{-\frac12}(T_0 X_1 T_0 - a_1 (-a_2) T_0^{-1}X_1^{-1}T_0)\\
&=a_1^{-\frac12} (-a_2)^{-\frac12}(T_0 X_1 T_0 + a_1 a_2 T_0^{-1}(-a_1^{-1}a_2^{-1}X_1+(a_1^{-1}+a_2^{-1}))T_0)\\
&= a_1^{-\frac12} (-a_2)^{-\frac12}((T_0-T_0^{-1}) X_1 T_0 +(a_1- (-a_2)))\\
&=(t_0^{\frac12} - t_0^{-\frac12} )W_1 + (t_k^{\frac12} - t_k^{-\frac12}),
\end{align*}
which establishes \eqref{T0pastW1}.

\paragraph{
The first relation in (\ref{Heckedefn}) from the relations (\ref{PresBraidRelsB1}--\ref{TiWcommuteB4}), (\ref{quadraticT0andTi}) and (\ref{TipastWi}--\ref{T0pastW1}).}

By \eqref{T0pastW1},
\begin{align*}
a_1^{-\frac12} (-a_2)^{-\frac12}&(T_0 X_1 T_0 - a_1 (-a_2) T_0^{-1}X_1^{-1}T_0)
=T_0 W_1 - W_1^{-1} T_0
=(t_0^{\frac12} - t_0^{-\frac12})W_1 + (t_k^{\frac12} - t_k^{-\frac12}) \\
&= a_1^{-\frac12} (-a_2)^{-\frac12}((T_0-T_0^{-1}) X_1 T_0 + (a_1-(-a_2))) \\
&=a_1^{-\frac12} (-a_2)^{-\frac12}(T_0 X_1 T_0 + a_1 a_2 T_0^{-1}(-a_1^{-1}a_2^{-1}X_1+(a_1^{-1}+a_2^{-1}))T_0),
\end{align*}
giving $X_1^{-1} = -a_1^{-1}a_2^{-1}X_1+(a_1^{-1}+a_2^{-1})$, which establishes the first relation in \eqref{Heckedefn}.
\end{proof}

As vector spaces, 
\begin{equation}
H_k^{\mathrm{ext}} 
= \CC[W_0^{\pm 1}, W_1^{\pm1}, \ldots, W_k^{\pm1}] \otimes H_k^{\mathrm{fin}},
\label{translsplitting}
\end{equation}
where $H_k^{\mathrm{fin}}$ is the subalgebra of $H_k^{\mathrm{ext}}$ generated by $T_0, T_1, \ldots, T_{k-1}$.
The algebra $H_k^{\mathrm{fin}}$ is the Iwahori-Hecke algebra of finite type $C_k$.  If 
$s_0, s_1,\ldots, s_{k-1}$ are the generators of $\cW_0$ as given in \eqref{W0defn}, write 
$T_w = T_{s_{i_1}}\cdots T_{s_{i_\ell}}$ for a reduced expression $w = s_{i_1}\cdots s_{i_\ell}$,
so that
$$\{ T_w\ |\ w\in \cW_0\}\qquad\hbox{is a $\CC$-basis of $H_k^{\mathrm{fin}}$.}$$
Thus \eqref{translsplitting} means that any element $h\in H_k^{\mathrm{ext}}$ can be written uniquely as
$$h = \sum_{w\in \cW_0} h_w T_w, \qquad\hbox{with}\quad
h_w\in \CC[W_0^{\pm1}, W_1^{\pm1}, \ldots, W_k^{\pm1}].$$

Let 
\begin{equation}\label{eq:Wlambdadef}
W^\lambda = W_0^{\lambda_0} W_1^{\lambda_1}W_2^{\lambda_2}\cdots W_k^{\lambda_k}
\qquad\hbox{for $\lambda = (\lambda_0, \lambda_1, \ldots, \lambda_k)\in \ZZ^{k+1}$.}
\end{equation}
Relations \eqref{TipastWi} and \eqref{T0pastW1} produce an action of $\cW_0$ on 
$$\CC[W_0^{\pm 1}, W_1^{\pm1}, \ldots, W_k^{\pm1}]
= \mathrm{span}_\CC\{ W^\lambda\ |\ \lambda = (\lambda_0, \lambda_1, \ldots, \lambda_k)\in \ZZ^{k+1}\}.
$$
Namely, for $w\in \cW_0$ and $\lambda\in \ZZ^{k+1}$,
\begin{equation*}
	wW^\lambda = W^{w\lambda}, \qquad \text{where} \qquad s_0\lambda 
		= s_0(\lambda_0, \lambda_1, \ldots, \lambda_k)= (\lambda_0, -\lambda_1, \ldots, \lambda_k),
\quad\hbox{and}
\end{equation*}
\begin{equation}
	s_i\lambda 
	= s_i(\lambda_0, \lambda_1, \ldots, \lambda_k) 
	= (\lambda_0, \lambda_1, \ldots, \lambda_{i-1}, \lambda_{i+1}, \lambda_i, \lambda_{i+2}, \ldots, \lambda_k),
\label{W0actiononW}
\end{equation}
for $i=1,2,\ldots, k-1$ (see \cite[(1.12)]{Ra2}).  With this notation,
for  $\lambda\in \ZZ^{k+1}$, the relations \eqref{TipastWi} and \eqref{T0pastW1} give
\begin{align}
T_iW^\lambda &= W^{s_i\lambda}T_i + (t^{\frac12}-t^{-\frac12})\frac{W^\lambda - W^{s_i\lambda}}{1-W_iW_{i+1}^{-1}}
\qquad\hbox{and} \label{eq:TiW}\\
T_0W^\lambda &= W^{s_0\lambda}T_0 
+ \left((t_0^{\frac12}-t_0^{-\frac12}) + (t_k^{\frac12} -t_k^{-\frac12})W_1^{-1}\right) \frac{W^\lambda - W^{s_0\lambda}}{1-W_1^{-2}},
\label{eq:T0W}
\end{align}
and, replacing $s_i\lambda$ by $\mu$,
\begin{align}
W^{\mu}T_i &= T_i W^{s_i\mu} + (t^{\frac12}-t^{-\frac12})\frac{W^\mu - W^{s_i\mu}}{1-W_iW_{i+1}^{-1}}
\qquad\hbox{and} \label{eq:WTi}\\
W^{\mu}T_0 &= T_0W^{s_0\mu} 
+ \left((t_0^{\frac12}-t_0^{-\frac12}) + (t_k^{\frac12} -t_k^{-\frac12})W_1^{-1}\right) \frac{W^\mu - W^{s_0\mu}}{1-W_1^{-2}}, 
\qquad\hbox{for $\mu\in \ZZ^{k+1}$.}
\label{eq:WT0}
\end{align}

The subalgebra $H_k \subseteq H_k^\ext$ 
generated by $W_1,\ldots, W_k$ and $T_0, \ldots, T_{k-1}$ is the affine Hecke algebra of 
type C considered, for example, in \cite{Lu1}.
The following theorem determines the center of $H_k^\ext$ and shows that, as algebras, 
$H_k^\ext$ is a tensor product of $H_k$ by the algebra of Laurent polynomials in one variable.
It follows that the irreducible representations of $H_k^\ext$ are indexed by $\CC^\times  \times \hat{H}_k$,
where $\hat{H}_k$ is an indexing set for the irreducible representations of $H_k$.

\begin{thm}\label{thm:Center_of_Hext}  Let $H_k$ be the subalgebra of $H_k^\ext$ 
generated by $W_1,\ldots, W_k$ and $T_0, \ldots, T_{k-1}$.
As algebras, 
\begin{equation}
H_k^\ext \cong \CC[W_0^{\pm1}] \otimes H_k, 
\label{Hextasacentralextension}
\end{equation}
The center of $H_k^{\mathrm{ext}}$ is 
$$Z(H_k^{\mathrm{ext}}) 
= \CC[W_0^{\pm1}]\otimes \CC[W_1^{\pm1},\ldots, W_k^{\pm1}]^{\cW_0},$$ 
and $H_k^{\mathrm{ext}}$ is a free module of rank $\Card(\cW_0)^2=2^{2k}(k!)^2$ over
$Z(H_k^{\mathrm{ext}})$.
\end{thm}
\begin{proof} As observed in \eqref{eq:Z0iscentral},  $Z_0$ is central in $\cB_k^\ext$, and therefore $W_0=(-1)^k(a_1 a_2 b_1 b_2)^{k/2} Z_0$ is central in $H_k^\ext$.  Thus
\begin{equation}
H_k^\ext = \CC[W_0^{\pm1}] \otimes H_k.
\label{Hextasacentralextension}
\end{equation}
By the formulas in \eqref{W0actiononW}, the Laurent polynomial ring $\CC[W_1^{\pm1}, \ldots, W_k^{\pm1}]$ is a $\cW_0$-submodule of 
$
\CC[W_0^{\pm1},W_1^{\pm1},\ldots, W_k^{\pm1}]$, 
and
\begin{align}
\CC[W_0^{\pm1}, W_1^{\pm1},\ldots, W_k^{\pm1}]^{\cW_0}= \CC[W_1^{\pm1},\ldots, W_k^{\pm1}]^{\cW_0}\otimes \CC[W_0^{\pm1}].
\label{W0invariants}
\end{align}
The proof that $Z(H_k^\ext) = \CC[W_0^{\pm1},W_1^{\pm1},\ldots,W_k^{\pm1}]^{\cW_0}$ is exactly as in  \cite[Thm. 4.12]{RR}.
The fact that $H_k^{\mathrm{ext}}$ is a free module of rank $\Card(\cW_0)^2$ over $\CC[\CC[W_0^{\pm1},W_1^{\pm1},\ldots,W_k^{\pm1}]^{\cW_0}]$ follows from \eqref{translsplitting} and \cite[Theorem1.17]{Ra2}.
\end{proof}

\subsection{Weights of representations and intertwiners}

Let $t^{\frac12}\in \CC^\times$ be such that $(t^{\frac12})^\ell\ne 1$ for $\ell\in \ZZ$.
All irreducible complex representations $\gamma$ of the algebra
$\CC[W_0^{\pm1}, W_1^{\pm1}, \ldots, W_k^{\pm1}]$ are one-dimensional. 
Identify the sets
\begin{align}
\cC &=\{\hbox{irreducible representations $\gamma$ of $\CC[W_0^{\pm1}, W_1^{\pm1}, \ldots, W_k^{\pm1}]$}\}
 \label{eq:identify_gamma_c}\\
&\qquad \leftrightarrow \quad
\{ \hbox{sequences\quad $(z, \gamma_1, \ldots, \gamma_k)\in (\CC^\times)^{k+1}$ }\} \nonumber \\
&\qquad \leftrightarrow \quad \{ \hbox{sequences\quad $(\zeta, c_1, \ldots, c_k)\in \CC^{k+1}$ }\}  \nonumber
\end{align}
via
\begin{equation}\label{eq:gamma-to-c}
\gamma(W_0) = z = (-1)^k t^{\zeta} \qquad\hbox{and}\qquad
\gamma(W_i) = \gamma_i=-t^{c_i}\ \hbox{for $i=1,\ldots, k$}
\end{equation}
(the choice of sign in the last equation is an artifact of equations \eqref{Weigenvals} and \eqref{W0eigenvals} and an effort to make the 
combinatorics of contents of boxes Section 5 optimally helpful). 
The action of $\cW_0$ from \eqref{W0actiononW} induces an action of $\cW_0$ on $\cC$ by
\begin{equation}
	(w\gamma)(W^\lambda) = \gamma(W^{w^{-1}\lambda}),
		\qquad\hbox{for $w\in \cW_0$ and $\lambda\in \ZZ^{k+1}$.}
\end{equation} 
Equivalently, 
on sequences $(\zeta, c_1, \dots, c_k)$, this action is given by
\begin{align}
w(\zeta, c_1, \ldots, c_k) 
= (\zeta, c_{w^{-1}(1)}, \ldots, c_{w^{-1}(k)}), \qquad\hbox{for $w\in \cW_0$. }
\label{W0actionsiform}
\end{align}

Let $\tilde{H}_k^{\mathrm{ext}}$ be the extensions of ${H}_k^{\mathrm{ext}}$ by the rational functions in $W_1, \ldots, W_k$:
$$
\tilde{H}_k^{\mathrm{ext}} = \CC[W_0^{\pm1}]\otimes \CC(W_1,\ldots, W_k)\otimes H_k^{\mathrm{fin}} ,$$
where $H_k^{\mathrm{fin}}$ is the subalgebra of $H_k^{\mathrm{ext}}$ generated by $T_0, T_1,\ldots, T_{k-1}$.
The \emph{intertwining operators} for $\tilde{H}_k^{\mathrm{ext}}$ are 
\begin{equation}\label{intertwinerdefs}
\tau_0 
= T_0 - \frac{ (t_0^{\frac12} - t_0^{-\frac12}) + (t_k^{\frac12} - t_k^{-\frac12})W_1^{-1} }
{1-W_1^{-2}} 
\qquad\hbox{and}\qquad
\tau_i = T_i - \frac{t^{\frac12} - t^{-\frac12} }{1-W_iW_{i+1}^{-1} }
\end{equation}
for $i=1,2,\ldots, k-1$.
Proposition \ref{prop:intertwiners} shows that these operators satisfy
$\tau_0 W^\lambda = W^{s_0\lambda}\tau_0$ and $\tau_i W^\lambda = W^{s_i\lambda}\tau_i$
so that, for $w\in \cW_0$ and $\lambda = (\lambda_0, \ldots, \lambda_k)\in \ZZ^{k+1}$,
\begin{equation}
\tau_w W^\lambda = W^{w\lambda}\tau_w, \qquad\hbox{where $\tau_w = \tau_{i_1}\ldots \tau_{i_\ell}$}
\label{tauwpastW}
\end{equation}
for a reduced expression $w= s_{i_1}\cdots s_{i_\ell}$.

Each $H_k^{\mathrm{ext}}$-module $M$ can be written as
$\displaystyle{ 
M = \bigoplus_{\gamma\in \cC} M_\gamma^{\mathrm{gen}},}$ where for each $\gamma = (z, \gamma_1, \ldots, \gamma_k) \in \cC$,
\begin{equation}
M_\gamma^{\mathrm{gen}} = \left\{ m\in M\ \left|\ {
\text{there exists 
$N\in \ZZ_{>0}$ such that $(W_0 - z)^N m = 0$}
\atop 
\text{ and $(W_i-\gamma_i)^N m=0$ for $i=1,\ldots, k$}
}\right.\right\}
\label{eq:Mcgen}
\end{equation}
is the \emph{generalized weight space} associated to $\gamma$. 
The intertwiners \eqref{intertwinerdefs} define vector space homomorphisms
\begin{equation}
\tau_0\colon M_\gamma^{\mathrm{gen}}\longmapsto M_{s_0\gamma}^{\mathrm{gen}}
\qquad\hbox{and}\qquad
\tau_i\colon M_\gamma^{\mathrm{gen}}\longmapsto M_{s_i\gamma}^{\mathrm{gen}}
\ \ \hbox{for $i=1,\ldots, k-1$,}
\label{intertwinermaps}
\end{equation}
where 
\begin{align*}
&\text{$\tau_0$ is defined only when $\gamma_1\ne 1$,  so that $(1-W_1^{-1})^{-1}$
is well-defined on $M_{\gamma}^{\mathrm{gen}}$ and}\\
&\text{$\tau_i$ is defined only when
$\gamma_i\ne \gamma_{i+1}$, so that $(1-W_iW_{i+1}^{-1})^{-1}$ is
well-defined on $M_\gamma^{\mathrm{gen}}$}
\end{align*}
for $i=1,\ldots, k-1$.

\begin{prop}\label{prop:intertwiners}
\textbf{(Intertwiner presentation)}  The algebra $\tilde{H}_k^{\mathrm{ext}}$
is generated by $\tau_0, \ldots, \tau_k$, $W_0$, and $\CC(W_1,\ldots, W_k)$ with relations
\begin{equation}
\begin{tikzpicture}
	\foreach \x in {0,1, 2, 4,5}{
		\draw (\x,0) circle (2.5pt);
		}
	\foreach \x in {1, 2}{
		\node[label=above:$\tau_{\x}$] at (\x,0){};
		\node[label=above:$\tau_{k-\x}$] at (6-\x,0){};}
	\node[label=above:$\tau_0$] at (0){};
	\draw[double distance = 2pt] (0)--(1);
	\draw (1)--(2) (4)--(5);
	\draw[dashed] (2) to (4);
\end{tikzpicture}
\label{taubraidrels}
\end{equation}
in the notation of \eqref{braidlengths};
\begin{equation}
\tau_0 W_1 = W_1^{-1}\tau_0 \qquad\hbox{and}\qquad \tau_0W_j=W_j\tau_0\  
\hbox{for $j\ne 1$};
\label{tau0pastW}
\end{equation}
for $i=1,\ldots, k-1$,
\begin{equation}
\tau_i W_i = W_{i+1}\tau_i \quad \text{and} \quad
\tau_i W_{i+1} = W_i\tau_i \quad \text{for }i>0, \qquad\hbox{and}\quad
\tau_iW_j = W_j \tau_i\quad \hbox{for $j \neq i, i+1$};
\label{tauipastW}
\end{equation}
\begin{align}
\tau_0^2
&=
\frac{(1-t_0^{\frac12}t_k^{\frac12}W_1^{-1})}{1-W_1^{-1}}
\frac{(1+t_0^{\frac12}t_k^{-\frac12}W_1^{-1})}{1+W_1^{-1}}
\frac{(1+t_0^{-\frac12}t_k^{\frac12}W_1^{-1})}{1+W_1^{-1}}
\frac{(1-t_0^{-\frac12}t_k^{-\frac12}W_1^{-1})}{1-W_1^{-1}}; 
\label{tau0sq}
\end{align} 
\begin{align}\text{ and } \qquad 
\tau_i^2 
&=\frac{(t^{\frac12} - t^{-\frac12}W_i^{-1}W_{i+1})
(t^{\frac12} - t^{-\frac12}W_{i+1}^{-1}W_i)}
{(1 - W_i^{-1}W_{i+1})(1 - W_{i+1}^{-1}W_i)}
\quad\hbox{for $i = 1,\ldots, k-1$.}
\label{tauisq}
\end{align}
\end{prop}

\begin{proof}~
The proof of the relations in \eqref{taubraidrels} is accomplished exactly as in the proof of 
\cite[Proposition 2.14(e)]{Ra2}; relation \eqref{tauisq} is \cite[Proposition 2.14(c)]{Ra2}.  Let us check the relations in \eqref{tauipastW} and \eqref{tau0sq}.

Using \eqref{TipastWi},
\begin{align*}
\tau_i W_i
&= \left(T_i - \frac{t^{\frac12} - t^{-\frac12}}{1-W_iW_{i+1}^{-1} }\right)W_i 
= W_{i+1}T_i + (t^{\frac12}-t^{-\frac12})\frac{W_i - W_{i+1}}{1-W_iW_{i+1}^{-1}}
- (t^{\frac12} - t^{-\frac12})\frac{W_i}{1-W_iW_{i+1}^{-1} } \\
&= W_{i+1}\left(T_i - \frac{t^{\frac12} - t^{-\frac12}}{1-W_iW_{i+1}^{-1} }\right) = W_{i+1}\tau_i.
\end{align*}
Similarly, using \eqref{T0pastW1},
\begin{align*}
\tau_0 W_1
&= \left(T_0 - \frac{ (t_0^{\frac12} - t_0^{-\frac12}) + (t_k^{\frac12} - t_k^{-\frac12})W_1^{-1} }
{1-W_1^{-2}} \right)W_1 \\
&= W_1^{-1}T_0 + (t_0^{\frac12} - t_0^{-\frac12})W_1
	+  (t_k^{\frac12} - t_k^{-\frac12})
	- \frac{ (t_0^{\frac12} - t_0^{-\frac12})W_1 + (t_k^{\frac12} - t_k^{-\frac12})}{1-W_1^{-2}} \\
&= W_1^{-1}\left(T_0 - \frac{ (t_0^{\frac12} - t_0^{-\frac12}) + (t_k^{\frac12} - t_k^{-\frac12})W_1^{-1} }
{1-W_1^{-2}} \right)
= W_1^{-1}\tau_0.
\end{align*}
For $i = 0, \dots, k-1$ and $j \neq i, i+1$,  $\tau_i$ and $W_j$ commute by the second set of relations in 
\eqref{TipastWi}.  These computations establish the relations in \eqref{tau0pastW} and \eqref{tauipastW}.

By the first relation in \eqref{quadraticT0andTi}, $T_0 = T_0^{-1} + (t_0^{\frac12} - t_0^{-\frac12}) $, so that
\begin{align*}
\tau_0
&= T_0 - \frac{(t_0^{\frac12} - t_0^{-\frac12}) + (t_k^{\frac12} - t_k^{-\frac12})W_1^{-1}}{1-W_1^{-2}}  
=  T_0^{-1} + (t_0^{\frac12} - t_0^{-\frac12})
	+ \frac{ (t_0^{\frac12} - t_0^{-\frac12})W_1^2
		+ (t_k^{\frac12}  - t_k^{-\frac12})W_1}{1-W_1^{2}}  \\
&= T_0^{-1} + \frac{ (t_0^{\frac12} - t_0^{-\frac12}) + (t_k^{\frac12} - t_k^{-\frac12})W_1}{1-W_1^{2}}.
\end{align*}
Then
\begin{align*}
\tau_0^2
&= \tau_0 \left(T_0 - \frac{(t_0^{\frac12} - t_0^{-\frac12}) 
+ (t_k^{\frac12} - t_k^{-\frac12})W_1^{-1}}{1-W_1^{-2}}\right) 
= \tau_0T_0 
- \left(\frac{(t_0^{\frac12} - t_0^{-\frac12}) + (t_k^{\frac12} - t_k^{-\frac12})W_1}{1-W_1^2}\right) \tau_0 \\
&=
\left(T_0^{-1} + \frac{ (t_0^{\frac12} - t_0^{-\frac12}) + (t_k^{\frac12} - t_k^{-\frac12})W_1}
{1-W_1^{2}}\right) T_0 \\
&\qquad\qquad
- \left(\frac{ (t_0^{\frac12} - t_0^{-\frac12}) + (t_k^{\frac12} - t_k^{-\frac12})W_1}{1-W_1^{2}}\right) 
\left(T_0 - \frac{(t_0^{\frac12} - t_0^{-\frac12}) + (t_k^{\frac12} - t_k^{-\frac12})W_1^{-1}}{1-W_1^{-2}}\right) \\
&= 1 + \left(
\frac{ (t_0^{\frac12} - t_0^{-\frac12}) + (t_k^{\frac12} - t_k^{-\frac12})W_1}{1-W_1^{2}}\right) 
\left(\frac{ (t_0^{\frac12} - t_0^{-\frac12}) + (t_k^{\frac12} - t_k^{-\frac12})W_1^{-1}}{1-W_1^{-2}} \right) \\
&= 1 - 
\left(\frac{ (t_0^{\frac12} - t_0^{-\frac12})W_1^{-2}
+ (t_k^{\frac12} - t_k^{-\frac12})W_1^{-1}
 }
{1-W_1^{-2}}\right) \left(
\frac{(t_0^{\frac12} - t_0^{-\frac12})
+  (t_k^{\frac12} - t_k^{-\frac12})W_1^{-1}
 }
{1-W_1^{-2}} \right) \\
&=\frac{\left(
\begin{array}{l}
1-2W_1^{-2}+W_1^{-4} - ((t_0^{\frac12}-t_0^{-\frac12})^2+(t_k^{\frac12}-t_k^{-\frac12})^2)W_1^{-2} \\
\quad - (t_0^{\frac12}-t_0^{-\frac12})(t_k^{\frac12}-t_k^{-\frac12})W_1^{-1}
- (t_0^{\frac12}-t_0^{-\frac12})(t_k^{\frac12}-t_k^{-\frac12})W_1^{-3}
\end{array}
\right)
}
{(1-W_1^{-2})^2} \\
&=
\frac{(1-t_0^{\frac12}t_k^{\frac12}W_1^{-1})}{1+W_1^{-1}}
\frac{(1+t_0^{\frac12}t_k^{-\frac12}W_1^{-1})}{1-W_1^{-1}}
\frac{(1+t_0^{-\frac12}t_k^{\frac12}W_1^{-1})}{1-W_1^{-1}}
\frac{(1-t_0^{-\frac12}t_k^{-\frac12}W_1^{-1})}{1+W_1^{-1}}, 
\end{align*} 
establishing \eqref{tau0sq}. 

\end{proof}

\section{Calibrated representations of $H_k^{\mathrm{ext}}$}\label{sec:CalibratedRepsOfHk}

A \emph{calibrated $H_k^{\mathrm{ext}}$-module} is an $H_k^{\mathrm{ext}}$-module $M$ such 
that $W_0, W_1, \ldots, W_k$ are simultaneously diagonalizable as operators on $M$.  In the
context of \eqref{eq:Mcgen}, $M$ is calibrated if
\begin{equation}
M = \bigoplus_{\gamma\in \cC} M_\gamma,
\quad\hbox{where}\quad M_\gamma 
= \{ m\in M\ |\ \hbox{$W_0m = zm$ and $W_im = \gamma_i m$ for 
$i=1,\ldots, k$} \}
\label{calibrateddefn}
\end{equation}
for $\gamma = (z, \gamma_1,\ldots, \gamma_k)\in \cC$.
Another formulation is that $M$ is calibrated if $M$ has a basis of simultaneous eigenvectors for
$W_0, \ldots, W_k$.
This section follows the framework of \cite{Ra2} in developing combinatorial tools for describing
the structure and the classification of irreducible calibrated $H_k^{\mathrm{ext}}$-modules.  In Section 5
we will use this combinatorics to analyze and classify the $H_k^{\mathrm{ext}}$-modules
arising in the Schur-Weyl duality settings.

With notations as in the definition of $\cW_0$ in \eqref{eq:Weyl-group}, 
the \emph{reflection representation} of $\cW_0$ is the action of $\cW_0$ 
on $\fh_\RR = \RR^k$ given by 
$$w(c_1, \ldots, c_k) = (c_{w^{-1}(1)}, \ldots, c_{w^{-1}(k)}),
\qquad\hbox{where $c_{-i} = - c_i$ for $i=1, 2,\ldots, k$.}
$$
The dual space $\fh^*_\RR$ has basis $\vep_1, \ldots, \vep_k$, where
$\vep_i\colon \fh_\RR\to \RR$ is the $\RR$-linear map given by $\vep_i(\gamma_1, \ldots, \gamma_k) = \gamma_i$.
With $\vep_{-i} = -\vep_i$, the action of $\cW_0$ on $\RR^k$
 produces an action on $\fh^*_\RR$ given by $w\vep_i = \vep_{w^{-1}(i)}$. 
 
 Let
\begin{align*}
R^+ 
&= \{ \vep_1, \ldots, \vep _k\} \sqcup
\{ \vep_j - \vep_i, \vep_j+\vep_i\ |\ 1\le i< j\le k\} \\
&= \{ \vep_1, \ldots, \vep _k\} 
\sqcup \{ \vep_j - \vep_i \ |\ 1\le i< j\le k\}
\sqcup \{ \vep_j-\vep_{-i}\ |\ 1\le i< j\le k\} \\
&= \{ \vep_1, \ldots, \vep _k\} 
\sqcup \{ \vep_j - \vep_i \ |\ i,j\in \{-k, \ldots, -1, 1, \ldots, k\}, i < j, i \neq -j\}.
\end{align*}
If $w\in \cW_0$, the \emph{inversion set of $w$} is 
\begin{align}
R(w) 
&= \{ \alpha\in R^+\ |\ w\alpha\not\in R^+\}   \label{R(w)origdefn}\\
&= \{ \vep_i\ |\ \hbox{if $i>0$ and $w(i) < 0$}\}
\sqcup 
\{\vep_j-\vep_i\ |\ \hbox{if $0< i < j$ and $w(i)>w(j)$}\} \\
&\qquad\sqcup
\{\vep_j+\vep_i\ |\ \hbox{if $0< i < j$ and $-w(i)>w(j)$}\}.
\nonumber
\end{align}
The chambers are the connected components of $\fh_\RR \backslash \bigcup_{\alpha\in R^+} \fh^\alpha$, where
$\fh^\alpha = \{ \gamma\in \fh_\RR\ |\ \alpha(\gamma)=0\}$.
The fundamental chamber in $\fh_\RR$ is
\begin{align*}
C &= \{ \cc\in \fh_\RR\ |\ \hbox{$\alpha(\gamma)\in \RR_{>0}$ for $\alpha\in R^+$}\} 
= \{ (c_1, \ldots, c_k)\in \RR^k \ |\ 0< c_1 < c_2 <\cdots < c_k \},
\end{align*}
and the group $\cW_0$ can be identified with the set of chambers  via the bijection
$$\begin{matrix}
\cW_0 &\longleftrightarrow &\{\hbox{chambers}\} \\
w &\longmapsto &w^{-1}C
\end{matrix}.
\qquad\hbox{Since}\quad
w^{-1}C = \left\{ \cc \in \fh_\RR\ \left|\ \begin{array}{l}
\hbox{$\alpha(\cc)\in \RR_{<0}$ if $\alpha\in R(w)$ and} \\
\hbox{ $\alpha(\cc)\in \RR_{>0}$ if $\alpha\in R^+\backslash R(w)$}
\end{array} \right.\right\},
$$ 
the set $R(w)$ determines $w$.

\subsection{Local regions}\label{sec:LocalRegions}

For $\gamma=(\gamma_1, \ldots, \gamma_k)\in (\CC^\times)^k$, define
\begin{align}
Z(\gamma) 
&= \{ \vep_i \ |\ \gamma_i = \pm 1\} 
\sqcup \{\vep_j-\vep_i\ |\ \hbox{$0<i<j$, $\gamma_i\gamma_j^{-1}=1$} \} 
\sqcup \{ \vep_j+\vep_i\ |\ \hbox{$0<i<j$, $\gamma_i\gamma_j=1$} \},
\nonumber
 \\ 
P(\gamma) 
&= \{ \vep_i \ |\ \gamma_i\in \{ (t_0^{\frac12}t_k^{\frac12})^{\pm1}, 
(-t_0^{-\frac12}t_k^{\frac12})^{\pm1}\} \} 
\sqcup \{\vep_j-\vep_i\ |\ 0<i<j, \gamma_i\gamma_j^{-1} = t^{\pm1}\} \nonumber \\
&\quad\qquad
\sqcup \{\vep_j+\vep_i\ |\ 0<i<j, \gamma_i\gamma_j = t^{\pm1}\}.
\label{P(gamma)origdefn}
\end{align}

Using the conversion from $\gamma_i$ to $c_i$ as in \eqref{eq:gamma-to-c}, let
\begin{equation}
\gamma_i = -t^{c_i}, \quad\hbox{and set}\quad
\hbox{$-t^{r_1}$} = -t_k^{\frac12}t_0^{-\frac12} \quad\hbox{ and }\quad \hbox{$-t^{r_2}$}=t_k^{\frac12}t_0^{\frac12},
\label{eq:r1andr2}
\end{equation}
so that $-t^{\pm r_1}$ and $-t^{\pm r_2}$ are the eigenvalues of $W_1$ that cause $\tau_0^2$ to have a nonzero kernel (see \eqref{tau0sq}).
Then, for $\cc =(c_1, \ldots, c_k)\in \CC^k$ let $c_{-i}=-c_i$ and define
\begin{align}
Z(\cc) 
&= \{ \vep_i \ |\ c_i = 0\}
\sqcup \{\vep_j-\vep_i\ |\ \hbox{$0<i<j$ and $c_j-c_i=0$} \} \nonumber \\
&\hspace{1.6in}
\sqcup \{ \vep_j+\vep_i\ |\ \hbox{$0<i<j$ and $c_j+c_i=0$} \},
\label{Z(c)origdefn} \\
P(\cc) 
&= \{ \vep_i \ |\ c_i \in \{ \pm r_1, \pm r_2\} \}
\sqcup \{\vep_j-\vep_i\ |\ 0<i<j\ \hbox{and}\ c_j-c_i = \pm1\} \nonumber\\
&\hspace{1.6in}
\sqcup \{\vep_j+\vep_i\ |\ 0<i<j\ \hbox{and}\ c_j+c_i = \pm1\}.
\label{P(c)origdefn}
\end{align}
A \emph{local region} is a pair $(\cc, J)$ with $\cc\in \CC^k$ and $J\subseteq P(\cc)$.
The set of \emph{standard tableaux} of shape $(\cc,J)$ is 
\begin{equation}
\cF^{(\cc,J)} = \{ w\in \cW_0\ |\ R(w)\cap Z(\cc) = \emptyset,\ R(w)\cap P(\cc) = J\}.
\label{stdtaborigdefn}
\end{equation}

As in \cite[\S 5 and \S 8]{Ra2} the local regions $(\cc, J)$ and standard tableaux $w\in \cF^{(\cc,J)}$ can be converted to
configurations of boxes $\kappa$ and standard tableaux $S$ of shape $\kappa$ similar to those that are familiar in 
the literature on irreducible representations of Weyl groups of classical types.  
 As explained in \cite[\S 5.11]{Ra2}, the definitions of $Z(\cc)$ and
$P(\cc)$ make it possible to view the general case $\cc\in \CC^k$ as pieced together from the 
cases $\cc\in (\ZZ+\beta)^k$ where $\beta$ runs over a set of representatives of the $\ZZ$-cosets in $\CC$.
Below we make the conversion between local regions and configurations of boxes
explicit for the cases when $\cc\in \ZZ^k$ and $\cc\in (\ZZ+\frac12)^k$.   These are the cases that appear in the Schur-Weyl duality
approach to the representations of $H_k^\ext$ explored in Section \ref{sec:braids-on-tensor-space}.  
As in \cite[\S 8]{Ra2}, it is also true that these cases are sufficient to determine the general $\cc\in (\ZZ+\beta)^k$ setting.

Let $(\cc, J)$ be a local region with $\cc=(c_1,\ldots, c_k)$,
\begin{equation}
\cc\in \ZZ^k \quad\hbox{or}\quad \cc\in (\ZZ+\hbox{$\frac12$})^k,
\qquad \hbox{and}\qquad 0\le c_1\le \cdots\le c_k.
\label{boxconfigcases}
\end{equation}
Start with an infinite arrangement of NW to SE diagonals, numbered consecutively from $\ZZ$ or $\ZZ + \half$, 
increasing southwest to northeast (see Example \ref{ex:smallconfiguration}).
The \emph{configuration} $\kappa$ of boxes corresponding to the local region $(\cc,J)$ 
has $2k$ boxes 
(labeled $\mathrm{box}_{-k}, \ldots, \mathrm{box}_{-1}, \mathrm{box}_1, \ldots, \mathrm{box}_k$) 
with the following conditions.
\begin{enumerate}[\quad($\kappa$1)]
\item Location: $\mathrm{box}_i$ is on  diagonal $c_i$, where $c_{-i} = -c_{i}$ for $i\in \{-k, \ldots, -1\}$. 
\item Same diagonals: $\mathrm{box}_i$ is NW of $\mathrm{box}_j$ if
$i<j$ and $\mathrm{box}_i$ and $\mathrm{box}_j$ are on the same diagonal. 
\item Adjacent diagonals: \\
		{\quad} If $\vep_j-\vep_i\in J$, then $\mathrm{box}_j$ is NW (strictly north and weakly west) of
			 $\mathrm{box}_i$:
			\quad $	\begin{matrix}\begin{tikzpicture}[xscale=.4, yscale=-.4]
						\Part{1,1}
						\node at (.5,.5) {\tiny$j$}; \node at (.5,1.5) {\tiny$i$}; 
			\end{tikzpicture}\end{matrix}$\\
		If $\vep_j-\vep_i\in P(\cc)-J$, then $\mathrm{box}_j$ is SE (weakly south and strictly east) of 
			$\mathrm{box}_i$:
			\quad $	\begin{matrix}\begin{tikzpicture}[xscale=.4, yscale=-.4]
						\Part{2}
						\node at (.5,.5) {\tiny$i$}; \node at (1.5,.5) {\tiny$j$}; 
					\end{tikzpicture}\end{matrix}$
\item Markings: There is a marking on each of the diagonals $r_1$, $-r_1$, $r_2$ and $-r_2$. \\
		If $\vep_i\in J$, $\mathrm{box}_i$ is NW of the marking on diagonal $c_i$:  
				\quad $	\begin{matrix}\begin{tikzpicture}[xscale=.4, yscale=-.4]
						\Part{1}
						\filldraw [red] (1,1) circle (4pt);
						\node at (.5,.5) {\tiny$i$};
					\end{tikzpicture}\end{matrix}$\\
		If $\vep_i\in P(\cc) - J$, then $\mathrm{box}_i$ is SE of the marking in diagonal $c_i$ :
				\quad $	\begin{matrix}\begin{tikzpicture}[xscale=.4, yscale=-.4]
					\Part{1}
					\filldraw [red] (0,0) circle (4pt);
					\node at (.5,.5) {\tiny$i$};
			\end{tikzpicture}\end{matrix}$
\end{enumerate}
Condition ($\kappa$1) enables the values $(c_{-k},\ldots, c_{-1}, c_1, \ldots, c_k)$ to be read off of configuration $\kappa$. 
The sets $Z(\cc)$, $P(\cc)$, and $J$ can also be determined from the configuration $\kappa$ since
\begin{align*}
Z(\cc) &= 
\{ \vep_i ~|~ \hbox{$0<i$ and $\mathrm{box}_i$ is in diagonal $0$} \} \\
&\qquad\sqcup
\{\vep_j - \vep_i ~|~ 0<i<j \text{ and $\mathrm{box}_i$ and $\mathrm{box}_j$ are in the same diagonal} \} \\
&\qquad\sqcup
\{\vep_j + \vep_i ~|~ 0<i<j \text{ and $\mathrm{box}_i$ and $\mathrm{box}_j$ are both in diagonal $0$} \},
\\
P(\cc) &= 
\{ \vep_i ~|~ \hbox{$0<i$ and $\mathrm{box}_i$ is in diagonal $r_1$ or $r_2$} \}, \\
&\qquad \sqcup 
\{ \vep_j - \vep_i ~|~ \hbox{$0<i<j$ and $\mathrm{box}_i$ and $\mathrm{box}_j$ are in adjacent diagonals} \}  \\
&\qquad \sqcup 
\{ \vep_j + \vep_i ~|~ \hbox{$0<i<j$ and $\mathrm{box}_{-i}$ and $\mathrm{box}_j$ are in adjacent diagonals} \}, \quad \text{ and }  \\
J &= 
\{\vep_i \in P(\cc) ~|~  \hbox{$\mathrm{box}_i$ is NW of the marking} \} \\
&\qquad \sqcup 
\left\{ \vep_j - \vep_i \in P(\cc) ~|~ \hbox{$\mathrm{box}_j$ is northwest of $\mathrm{box}_i$}\right\} \\
&\qquad \sqcup 
\left\{ \vep_j + \vep_i \in P(\cc) ~|~ \hbox{$\mathrm{box}_j$ is northwest of $\mathrm{box}_{-i}$}\right\}.
\end{align*}

\noindent 

\noindent
A \emph{standard filling} of the boxes of $\kappa$ is a bijective function $S\colon \kappa \to \{-k, \ldots, -1, 1, \ldots k\}$ 
such that
\begin{enumerate}[\quad(S1)]
\item \label{S1} Symmetry: $S(\mathrm{box}_{-i}) = -S(\mathrm{box}_i)$.
\item \label{S2} Same diagonals: \\
If $0<i<j$ and $\mathrm{box}_i$ and $\mathrm{box}_j$ are on the same diagonal then $S(\mathrm{box}_i) <S(\mathrm{box}_j)$. 
\item \label{S3} Adjacent diagonals: \\
If $0<i<j$, $\mathrm{box}_i$ and $\mathrm{box}_j$ are on adjacent diagonals, and $\mathrm{box}_j$ is NW of $\mathrm{box}_i$, then $S(\mathrm{box}_j) <S(\mathrm{box}_i)$. \\
If $0<i<j$, $\mathrm{box}_i$ and $\mathrm{box}_j$ are on adjacent diagonals, and $\mathrm{box}_j$ is SE of $\mathrm{box}_i$, then $S(\mathrm{box}_j) >S(\mathrm{box}_i)$. 
\item \label{S4} Markings: \\
	If $\mathrm{box}_i$ is on a marked diagonal and is SE of the marking, then $S(\mathrm{box}_i) >0$.\\
	If $\mathrm{box}_i$ is on a marked diagonal and is NW of the marking, then $S(\mathrm{box}_i) <0$.
\end{enumerate}
The \emph{identity filling} of a configuration $\kappa$ is the filling $F$ of the boxes of $\kappa$ given by
$F(\mathrm{box}_i) = i$,
for $i=-k,\ldots, -1, 1, \ldots, k$.
The identity filling of $\kappa$ is usually not a standard filling of $\kappa$
(see Example \ref{ex:smallconfiguration}).

\begin{example}\label{ex:smallconfiguration}
Let $k=4$, $r_1 = 1$, and $r_2 = 3$. Consider $\cc = (-3,-2,-2,2,2,3)$. Then 
$$Z(\cc) = \{\vep_2 - \vep_1 \}  \qquad \text{ and } \qquad 
	P(\cc) = \left\{  \begin{matrix}
		\vep_3,\ \vep_3 - \vep_1,\ \vep_3 - \vep_2
		\end{matrix}
		\right\}.$$ 
The box configurations corresponding to $J = \{ \vep_3-\vep_2 \}$ and 
$J = \{ \vep_3, \vep_3-\vep_1, \vep_3-\vep_2 \}$
(filled with their identity fillings) are
\medskip

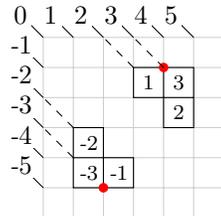
\centerline{\begin{tabular}{ccc}
\begin{tikzpicture}[xscale=.4, yscale=-.4]
	\draw [step=1cm, very thin, black!20!white] (-3,-3) grid (3,3);
	\draw (-1,2)--(-1,0)--(-2,0)--(-2,2)--(0,2)--(0,1)--(-2,1);
	\draw (1,-2)--(1,0)--(2,0)--(2,-2)--(0,-2)--(0,-1)--(2,-1);
	\filldraw [red] (1,-2) circle (4pt);	\filldraw [red] (-1,2) circle (4pt);
	\draw[dashed]
		(0,-2)--(-1,-3)	(1,-2)--(0,-3)
		(-2,0)--(-3,-1)	(-2,1)--(-3,0);
	\foreach \x in {0, 1, ..., 5} { \draw (\x-3, -3) to +(-.75,-.75) node[fill=white, inner sep=1.5pt]{\small $\x$};}
	\foreach \x in { 1, ..., 5} { \draw (-3, \x-3) to +(-.75,-.75) node[fill=white, inner sep=1.5pt]{\small -$\x$};}
	\node at (.5,-1.5) {\scriptsize $1$}; 	\node at (-.5,1.5) {\scriptsize -$1$}; 
	\node at (1.5,-.5) {\scriptsize $2$}; 	\node at (-1.5,.5) {\scriptsize -$2$}; 
	\node at (1.5,-1.5) {\scriptsize $3$}; 	\node at (-1.5,1.5) {\scriptsize -$3$}; 
\end{tikzpicture}&&
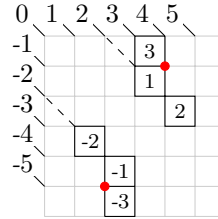
\begin{tikzpicture}[xscale=.4, yscale=-.4]
	\draw [step=1cm, very thin, black!20!white] (-3,-3) grid (3,3);
	\draw (1,-3)--(1,0)--(2,0)--(2,-1)--(0,-1)--(0,-3)--(1,-3)	(0,-2)--(1,-2);
	\draw (-1,3)--(-1,0)--(-2,0)--(-2,1)--(0,1)--(0,3)--(-1,3)	(0,2)--(-1,2);
	\filldraw [red] (1,-2) circle (4pt);	\filldraw [red] (-1,2) circle (4pt);
	\draw[dashed]	(0,-2)--(-1,-3)	(-2,0)--(-3,-1);
	\foreach \x in {0, 1, ..., 5} { \draw (\x-3, -3) to +(-.75,-.75) node[fill=white, inner sep=1.5pt]{\small $\x$};}
	\foreach \x in { 1, ..., 5} { \draw (-3, \x-3) to +(-.75,-.75) node[fill=white, inner sep=1.5pt]{\small -$\x$};}
	\node at (.5,-1.5) {\scriptsize $1$}; 	\node at (-.5,1.5) {\scriptsize -$1$}; 
	\node at (1.5,-.5) {\scriptsize $2$}; 	\node at (-1.5,.5) {\scriptsize -$2$}; 
	\node at (.5,-2.5) {\scriptsize $3$}; 	\node at (-.5,2.5) {\scriptsize -$3$}; 
\end{tikzpicture} \\ 
$\displaystyle J = \left\{ \vep_3-\vep_2 \right\}$ && $\displaystyle J = \left\{ \vep_3, \vep_3-\vep_1, \vep_3-\vep_2\right\}$
\end{tabular}}

\noindent
For both configurations, the identity filling is not a standard filling. Examples of standard fillings of the configuration corresponding to $J = \{\vep_3 - \vep_2\}$ include
$$
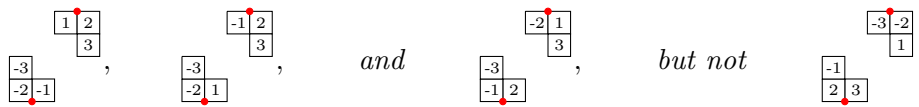
\begin{matrix}
\begin{tikzpicture}[xscale=.3, yscale=-.3]
	\BOX{0,-2}{$1$} \BOX{1,-2}{$2$} \BOX{1,-1}{$3$}
	\BOX{-1,1}{-$1$} \BOX{-2,1}{-$2$} \BOX{-2,0}{-$3$}
	\filldraw [red] (1,-2) circle (4pt);	\filldraw [red] (-1,2) circle (4pt);
\end{tikzpicture}\end{matrix}, \qquad 
\begin{matrix}
\begin{tikzpicture}[xscale=.3, yscale=-.3]
	\BOX{0,-2}{-$1$} \BOX{1,-2}{$2$} \BOX{1,-1}{$3$}
	\BOX{-1,1}{$1$} \BOX{-2,1}{-$2$} \BOX{-2,0}{-$3$}
	\filldraw [red] (1,-2) circle (4pt);	\filldraw [red] (-1,2) circle (4pt);
\end{tikzpicture}\end{matrix}, \qquad \text{ and } \qquad 
\begin{matrix}
\begin{tikzpicture}[xscale=.3, yscale=-.3]
	\BOX{0,-2}{-$2$} \BOX{1,-2}{$1$} \BOX{1,-1}{$3$}
	\BOX{-1,1}{$2$} \BOX{-2,1}{-$1$} \BOX{-2,0}{-$3$}
	\filldraw [red] (1,-2) circle (4pt);	\filldraw [red] (-1,2) circle (4pt);
\end{tikzpicture}\end{matrix}, \qquad \text{ but not } \qquad 
\begin{matrix}
\begin{tikzpicture}[xscale=.3, yscale=-.3]
	\BOX{0,-2}{-$3$} \BOX{1,-2}{-$2$} \BOX{1,-1}{$1$}
	\BOX{-1,1}{$3$} \BOX{-2,1}{$2$} \BOX{-2,0}{-$1$}
	\filldraw [red] (1,-2) circle (4pt);	\filldraw [red] (-1,2) circle (4pt);
\end{tikzpicture}\end{matrix}.$$
\end{example}

The proof of the following proposition is a straightforward, though slightly tedious, check that the conditions 
$R(w)\cap Z(\cc)=\emptyset$ and $R(w)\cap P(\cc)=J$ from
\eqref{stdtaborigdefn} convert to the conditions (S2), (S3), (S4) on standard fillings of shape $\kappa$. 
The proof is similar to the proof of \cite[Thm.\ 5.9]{Ra2}.

\begin{prop} \label{prop:stdtabbijection}
Let $\kappa$ be a configuration of boxes corresponding to a local region $(\cc, J)$ with $\cc\in \ZZ^k$ or $\cc\in (\ZZ+\frac12)^k$.
For $w\in \cW_0$ let $S_w$ be the filling of 
the boxes of $\kappa$ given by 
$$S_w(\mathrm{box}_i) = w(i),
\quad\hbox{for $i=-k, \ldots, -1, 1, \ldots, k$.}
$$
The map 
$$\begin{matrix}
\cF^{(\cc,J)} &\longrightarrow &\{ \hbox{standard fillings $S$ of the boxes of $\kappa$}\} \\
w &\longmapsto &S_w
\end{matrix}
\qquad\hbox{is a bijection}.
$$
\end{prop}

\def\ShapeA{
\draw[blue, densely dotted, thick]  (7,6)--(2,1) node[above left]{\scriptsize $\frac12$};
\draw[blue, densely dotted, thick]  (6,6)--(1,1) node[above left]{\scriptsize $-\frac12$};
\draw[blue, densely dotted, thick]  (2,10)--(-2,6) node[left]{\scriptsize $-\frac{17}{2}$};
\draw[blue, densely dotted, thick]  (10,1)--(6,-3) node[above left]{\scriptsize $\frac{17}{2}$};
\draw[blue, densely dotted, thick]  (2,9)--(-2,5) node[above left]{\scriptsize $-\frac{15}{2}$};
\draw[blue, densely dotted, thick]  (10,2)--(6,-2) node[left]{\scriptsize $\frac{15}{2}$};
\filldraw[white] (7,-2) -- (8, -2) -- (8, 0) -- (9,0) -- (9,2) -- (6,2) -- (6,1) -- (7,1) -- (7,-2)
			(2,2) rectangle (6,3)
			(2,4) rectangle (6,5)
			(2,3) rectangle (3,4)
			(5,3) rectangle (6,4)
			(1,9) to ++(-1,0) to ++(0,-2) to ++(-1,0) to ++(0,-2) to ++(3,0) to ++(0,1) to ++(-1,0) to ++(0,3);
\draw(7,-2)--(8,-2) (7,-1)--(8,-1) (7,0)--(9,0) (6,1)--(9,1) (2,2)--(9,2) (2,3)--(6,3) (2,4)--(6,4) (-1,5)--(6,5) (-1,6)--(2,6) (-1,7)--(1,7) (0,8)--(1,8) (0,9)--(1,9);
\draw (-1,5)--(-1,7) (0,5)--(0,9) (1,5)--(1,9) (2,2)--(2,6) (3,2)--(3,5) (4,4)--(4,5) (4,2)--(4,3) (5,2)--(5,5) (6,1)--(6,5) (7,-2)--(7,2) (8,-2)--(8,2) (9,0)--(9,2);
\node[bV,red] at (8,0){};
\node[bV,red!50!blue] at (5,3){};
\node[bV,red!50!blue] at (3,4){};
\node[bV,red] at (0,7){};
	\node (12) at (7.5,-1.5) {}; 	\node (n12) at (.5,8.5) {};
	\node (11) at (8.5,.5) {};	\node (n11) at (-.5,6.5) {};
	\node (10) at (7.5,-.5) {};	\node (n10) at (.5,7.5) {};
	\node (9) at (8.5,1.5) {};	\node (n9) at (-.5,5.5) {};
	\node (8) at (7.5,.5) {};	\node (n8) at (.5,6.5) {};
	\node (7) at (7.5,1.5) {};	\node (n7) at (.5,5.5) {};
	\node (6) at (6.5,1.5) {};	\node (n6) at (1.5,5.5) {};
	\node (5) at (5.5,2.5) {};	\node (n5) at (2.5,4.5) {};
	\node (4) at (5.5,3.5) {};	\node (n4) at (2.5,3.5) {};
	\node (3) at (4.5,2.5) {};	\node (n3) at (3.5,4.5) {};
	\node (2) at (5.5,4.5) {};	\node (n2) at (2.5,2.5) {};
	\node (1) at (3.5,2.5) {};	\node (n1) at (4.5,4.5) {};
}

\begin{example}\label{ex:stdtabbijection}
Let $k=12$, $r_1 = \frac{3}{2}$, $r_2 = \frac{15}{2}$,
$\cc = (\hbox{$\frac{1}{2}, \frac{1}{2}, \frac{3}{2}, \frac{3}{2}, \frac{5}{2}, \frac{9}{2}, 
\frac{11}{2}, \frac{13}{2}, \frac{13}{2}, \frac{15}{2}, \frac{15}{2}, \frac{17}{2}$})$
and
$$
J = \left\{
\begin{array}{l}
\vep_3, \vep_{10}, \vep_3-\vep_2, \vep_4-\vep_2,
\vep_5-\vep_4, \vep_8-\vep_7, \\
\vep_{10}-\vep_8, 
\vep_{10}-\vep_9, \vep_{11}-\vep_9, \vep_{12}-\vep_{10},
\vep_{12}-\vep_{11}
\end{array}
\right\}$$
Let
$$w = \left( \begin{array}{cccccccccccc} 1 &2 &3 &4 &5 &6 &7 &8 &9 &10 &11 &12 \\ 
-9 &10 &-8 &7 &6 &3 &4 &1 &5 &-11 &2 &-12
\end{array}\right)\in \cF^{(\cc,J)}.
$$
Then, for the corresponding configuration of boxes $\kappa$,
the identity filling $F$,  and the standard filling $S_w$ corresponding to $w$ are
$$
F=
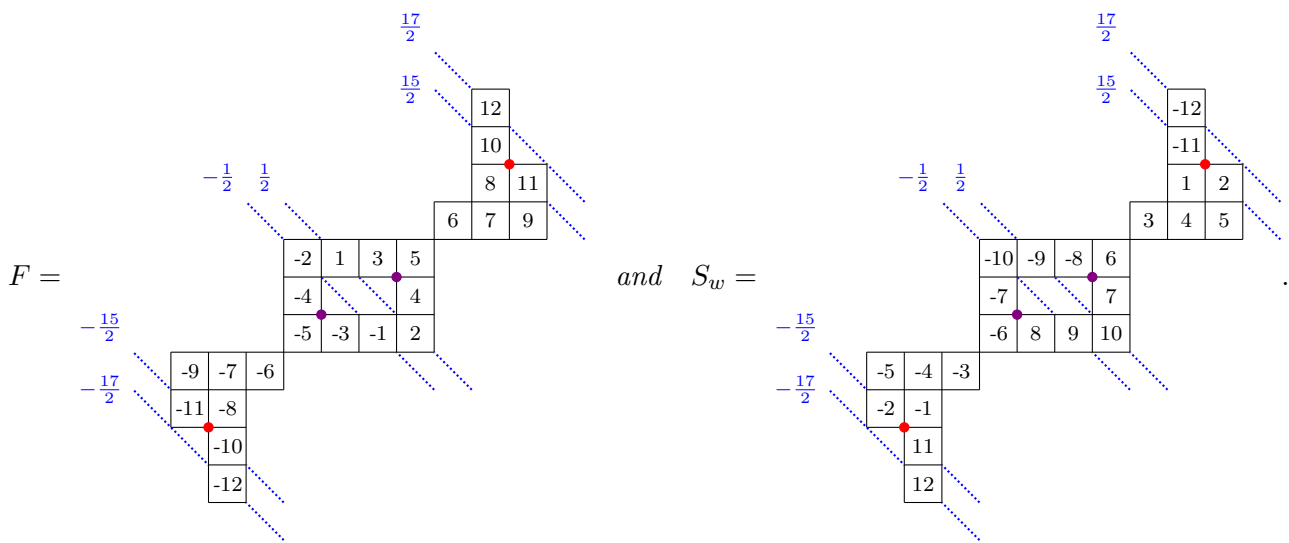
\begin{matrix}
\begin{tikzpicture}[xscale=.5,yscale=-.5]
\ShapeA
	\node at (12) {\scriptsize $12$}; 	\node at (n12) {\scriptsize -$12$};
	\node at (11) {\scriptsize $11$};	\node at (n11) {\scriptsize -$11$};
	\node at (10) {\scriptsize $10$};	\node at (n10) {\scriptsize -$10$};
	\node at (9) {\scriptsize $9$};	\node at (n9) {\scriptsize -$9$};
	\node at (8) {\scriptsize $8$};	\node at (n8) {\scriptsize -$8$};
	\node at (7) {\scriptsize $7$};	\node at (n7) {\scriptsize -$7$};
	\node at (6) {\scriptsize $6$};	\node at (n6)  {\scriptsize -$6$};
	\node at (5) {\scriptsize $5$};	\node at (n5) {\scriptsize -$5$};
	\node at (4) {\scriptsize $4$};	\node at (n4) {\scriptsize -$4$};
	\node at (3) {\scriptsize $3$};	\node at (n3) {\scriptsize -$3$};
	\node at (2) {\scriptsize $2$};	\node at (n2) {\scriptsize -$2$};
	\node at (1) {\scriptsize $1$};	\node at (n1) {\scriptsize -$1$};
\end{tikzpicture}\end{matrix}
\quad\text{and}\quad
S_w=\begin{matrix}
\begin{tikzpicture}[xscale=.5,yscale=-.5]
\ShapeA
	\node at (12) {\scriptsize -$12$}; 	\node at (n12) {\scriptsize $12$};
	\node at (11) {\scriptsize $2$};	\node at (n11) {\scriptsize -$2$};
	\node at (10) {\scriptsize -$11$};	\node at (n10) {\scriptsize $11$};
	\node at (9) {\scriptsize $5$};	\node at (n9) {\scriptsize -$5$};
	\node at (8) {\scriptsize $1$};	\node at (n8) {\scriptsize -$1$};
	\node at (7) {\scriptsize $4$};	\node at (n7) {\scriptsize -$4$};
	\node at (6) {\scriptsize $3$};	\node at (n6)  {\scriptsize -$3$};
	\node at (5) {\scriptsize $6$};	\node at (n5) {\scriptsize -$6$};
	\node at (4) {\scriptsize $7$};	\node at (n4) {\scriptsize -$7$};
	\node at (3) {\scriptsize -$8$};	\node at (n3) {\scriptsize $8$};
	\node at (2) {\scriptsize $10$};	\node at (n2) {\scriptsize -$10$};
	\node at (1) {\scriptsize -$9$};	\node at (n1) {\scriptsize $9$};
\end{tikzpicture}\end{matrix}.$$
\end{example}

\begin{remark} Borrowing a physical intuition, configurations are invariant under sliding boxes along diagonals like beads on an abacus, so long as boxes that run into each other are not allowed to exchange places, i.e.\ for most $c \in \ZZ$,
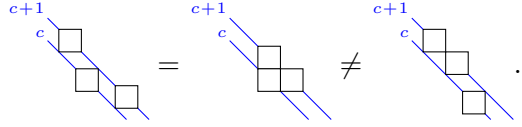
$$
\begin{matrix}
	\begin{tikzpicture}[xscale = .3, yscale=-.3]
		\draw[thin, blue] (-.5,-.5) node[inner sep=1pt, above left, blue] {\tiny$c\!+\!1$} to (4,4);
		\draw[thin, blue] (-.5,.5) node[inner sep=1pt, above left, blue] {\tiny$c$} to (3,4);
		\BOX{0,0}{} \BOX{2.5,2.5}{} \BOX{.75,1.75}{} 
		\end{tikzpicture}\end{matrix} 
	=\begin{matrix}
	\begin{tikzpicture}[xscale = .3, yscale=-.3]
		\draw[thin, blue] (-.5,-.5) node[inner sep=1pt, above left, blue] {\tiny$c\!+\!1$} to (4,4);
		\draw[thin, blue] (-.5,.5) node[inner sep=1pt, above left, blue] {\tiny$c$} to (3,4);
		\BOX{.75,.75}{} \BOX{1.75,1.75}{} \BOX{.75,1.75}{} 
		\end{tikzpicture}\end{matrix} 
	\neq
	\begin{matrix}
	\begin{tikzpicture}[xscale = .3, yscale=-.3]
		\draw[thin, blue] (-.5,-.5) node[inner sep=1pt, above left, blue] {\tiny$c\!+\!1$} to (4,4);
		\draw[thin, blue] (-.5,.5) node[inner sep=1pt, above left, blue] {\tiny$c$} to (3,4);
		\BOX{0,0}{} \BOX{1,1}{} \BOX{1.75,2.75}{} 
		\end{tikzpicture}\end{matrix}.
		$$
Then by arranging configurations so that the boxes are packed together, standard fillings of configurations are exactly analogous to standard tableaux for partitions. 

The only exception to this physical intuition is for boxes on the diagonals $\pm \half$. 
Note that if $c_i = \half$, then $\mathrm{box}_i$ and $\mathrm{box}_{-i}$ are on adjacent diagonals. However, since $2\vep_i = \vep_i - \vep_{-i} \notin R^+$
and therefore never in $P(\cc)$, the relative positions of $\mathrm{box}_i$ and $\mathrm{box}_{-i}$ will never be recorded in the set $J$. 
For example, in Figure \ref{fig:rank2nonregulars}, the point where $(c_1, c_2) = (\half, \half)$ has two configurations, each with two boxes overlapping in indication that $\mathrm{box}_i$ and $\mathrm{box}_{-i}$ may ``slide past each other''. The drawing 
$$
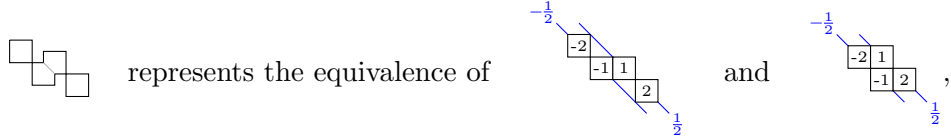
\begin{matrix}
	\begin{tikzpicture}[xscale = .3, yscale=-.3]
		\draw[fill=white]	 (0,0)--(.5,0)--(.5,-.5)--(1.5,-.5)--(1.5,.5)--(1,.5)--(1,1)--(0,1)--(0,0);
		\draw[thin, black!30] (.5,0)--(1,.5);
		\BOX{-1,-1}{} \BOX{1.5,.5}{}
		\end{tikzpicture}\end{matrix} 
		\quad\text{ represents the equivalence of }\quad 
		\begin{matrix}
		\begin{tikzpicture}[xscale = .3, yscale=-.3]
		\draw[thin, blue] (-1.5,-1.5) node[inner sep=0pt, above left, blue] {\tiny$-\!\frac12$} to (2.5,2.5);
		\draw[thin, blue] (-.5,-1.5) to (3.5,2.5)  node[inner sep=0pt, below right, blue] {\tiny$\frac12$};
		\BOX{-1,-1}{\tiny{-$2$}}\BOX{0,0}{\tiny{-$1$}}\BOX{1,0}{\tiny{$1$}} \BOX{2,1}{\tiny{$2$}}
		\end{tikzpicture}\end{matrix} 
		\quad\text{ and }\quad
		\begin{matrix}
		\begin{tikzpicture}[xscale = .3, yscale=-.3]
		\draw[thin, blue] (-1.5,-1.5) node[inner sep=0pt, above left, blue] {\tiny$-\!\frac12$} to (1.5,1.5);
		\draw[thin, blue] (-.5,-1.5) to (2.5,1.5)  node[inner sep=0pt, below right, blue] {\tiny$\frac12$};
		\BOX{-1,-1}{\tiny{-$2$}}\BOX{0,0}{\tiny{-$1$}}\BOX{0,-1}{\tiny{$1$}} \BOX{1,0}{\tiny{$2$}}
		\end{tikzpicture}\end{matrix},
		$$
(with boxes filled in the identity filling) where $\mathrm{box}_1$ and $\mathrm{box}_{-1}$ can move freely past each other, and 
$$
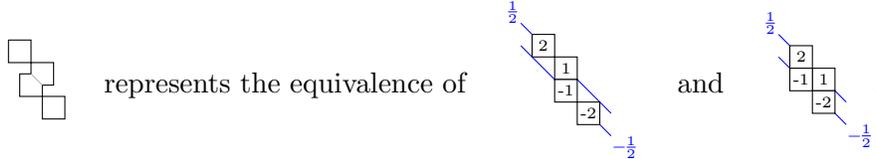
\begin{matrix}
	\begin{tikzpicture}[xscale = .3, yscale=-.3]
		\draw[fill=white]	 (0,0)--(.5,0)--(.5,-.5)--(1.5,-.5)--(1.5,.5)--(1,.5)--(1,1)--(0,1)--(0,0);
		\draw[thin, black!30] (.5,0)--(1,.5);
		\BOX{-.5,-1.5}{}	\BOX{1,1}{}
		\end{tikzpicture}\end{matrix} 
		\quad\text{ represents the equivalence of }\quad 
		\begin{matrix}
		\begin{tikzpicture}[xscale = .3, yscale=-.3]
		\draw[thin, blue] (-1.5,-1.5) node[inner sep=0pt, above left, blue] {\tiny$\frac12$} to (2.5,2.5);
		\draw[thin, blue] (-1.5,-.5) to (2.5,3.5)  node[inner sep=0pt, below right, blue] {\tiny$-\!\frac12$};
		\BOX{-1,-1}{\tiny{$2$}}\BOX{0,0}{\tiny{$1$}}\BOX{0,1}{-\tiny{$1$}} \BOX{1,2}{-\tiny{$2$}}
		\end{tikzpicture}\end{matrix} 
		\quad\text{ and }\quad
		\begin{matrix}
		\begin{tikzpicture}[xscale = .3, yscale=-.3]
		\draw[thin, blue] (-1.5,-1.5) node[inner sep=0pt, above left, blue] {\tiny$\frac12$} to (1.5,1.5);
		\draw[thin, blue] (-1.5,-.5) to (1.5,2.5)  node[inner sep=0pt, below right, blue] {\tiny$-\!\frac12$};
		\BOX{-1,-1}{\tiny{$2$}}\BOX{0,0}{\tiny{$1$}}\BOX{-1,0}{-\tiny{$1$}} \BOX{0,1}{-\tiny{$2$}}
		\end{tikzpicture}\end{matrix},
		$$
where $\mathrm{box}_2$ and $\mathrm{box}_{-2}$ can move freely past each other. In these two examples 
$\vep_1-\vep_{-2}\in P(\cc)$ and $\vep_2-\vep_{-1}\in P(\cc)$ and so 
the relative orientation of $\mathrm{box}_2$ and $\mathrm{box}_{-1}$
and the relative orientation of $\mathrm{box}_1$ and $\mathrm{box}_{-2}$
are recorded in $J$. 
Each configuration has exactly two standard fillings. 
\end{remark}

\subsection{Classifying and constructing calibrated representations}

Theorem \ref{thm:calibconst} below provides an indexing of the 
calibrated irreducible $H_k^{\mathrm{ext}}$-modules by skew local regions.
A \emph{skew local region} is a local region $(\cc, J)$, $\cc= (c_1,\ldots, c_k)$, such that \hfil\break
\indent if $w \in \cF^{(\cc,J)}$ then $w\cc = ((w\cc)_1, \ldots, (w\cc)_n)$ satisfies
\begin{equation}
\begin{array}{c}
(w\cc)_1\ne 0, \quad (w\cc)_2\ne 0, \quad (w\cc)_1\ne -(w\cc)_2, \\
\\
(w\cc)_i\ne (w\cc)_{i+1}\ \hbox{for $i=1,\ldots, k-1$,}\quad\hbox{and}\quad
(w\cc)_i\ne (w\cc)_{i+2}\ \hbox{for $i=1,\ldots, k-2$.}
\end{array}
\label{skewlocalregiondefn}
\end{equation}
Theorem \ref{thm:calibconst} is completely analogous to the same theorem for the case 
$t^{\frac12}=t_0^{\frac12}=t_k^{\frac12}$ in \cite[Theorem 3.5]{Ra2}.  As explained in the discussion
and remarks before \cite[Lemma 3.1]{Ra2} in \cite[\S 3]{Ra2},  getting exactly the right definition of
skew local region for the purpose of Theorem \ref{thm:calibconst} is accomplished by a detailed computation of the irreducible representations
in rank two cases.  More specifically, for $I \subseteq \{0, \dots, k\}$, let $H_I$ be the subalgebra of $H^{\mathrm{ext}}_k$
generated by $\{T_i\}_{i \in I}$ and $\CC[W_1^{\pm1}, \ldots, W_k^{\pm1}]$. 
Then the  conditions in \eqref{skewlocalregiondefn} guarantee that for $w\in \cF^{(\cc,J)}$ and $i,j\in \{0,1, \ldots, k-1\}$, 
$$\hbox{there
exists a calibrated $H_{\{i,j\}}$-module $M$ with $M_{w\cc}^{\mathrm{gen}}\ne 0$}.
$$
The cases where $H_{\{i,j\}}$ is of type $A\times A_1$ or of type $A_2$ are
checked in \cite{Ra1}.  However, when $H_{\{i,j\}}$ is of type $C_2$ and there are \emph{three distinct parameters},
we do not know a reference for this. So in the effort to provide a more complete presentation,
we have done the appropriate analysis in Section \ref{sec:H2Classification} for all generic choices of the three parameters
$t^{\frac12}$, $t_0^{\frac12}$, and  $t_k^{\frac12}$, as given in the following theorem (see also \eqref{eq:genericcond})

\begin{thm}\label{thm:calibconst} 
Assume $t^{\frac12}$, $t_0^{\frac12}$, and  $t_k^{\frac12}$ are invertible, $t^{\frac12}$ is not a root of unity, and
$$
t_0^{\frac12}t_k^{\frac12},-t_0^{-\frac12}t_k^{\frac12}\not\in \{1, -1, t^{\pm\frac12}, -t^{\pm\frac12}, t^{\pm1}, -t^{\pm1}\}
\quad \hbox{and}\quad
t_0^{\frac12}t_k^{\frac12}\ne (-t_0^{-\frac12}t_k^{\frac12})^{\pm1}.
$$
\begin{enumerate}[(a)]
\item Let  $(\cc,J)$ be a skew local region and let $z\in \CC^\times$. Define 
\begin{equation}
H_k^{(z,\cc,J)}  = \mathrm{span}_\CC \{ v_w \ |\  w \in \cF^{(\cc,J)}\},
\label{HgammaJ}
\end{equation}
so that the symbols $v_w$ are a labeled basis of the vector space $H_k^{(z,\cc,J)}$.  Let
$$\gamma_i = -t^{c_i}
\quad\hbox{for $i=1,2,\ldots, k$,} \qquad \text{and} \qquad
\gamma_0 = z\gamma_{w^{-1}(1)}^{-1}\cdots \gamma_{w^{-1}(k)}^{-1}.
$$
Then the following formulas make $H_k^{(z,\cc,J)}$ into an irreducible $H^\ext_k$-module:
\begin{align}
PW_1&\cdots W_k v_w = zv_w, \qquad P v_w = \gamma_0 v_w, \qquad
W_i v_w = \gamma_{w^{-1}(i)} v_w, \label{snormActionPW} \\
T_i v_w &= [T_i]_{ww} v_w + \sqrt{-([T_i]_{ww}-t^{\frac12})([T_i]_{ww}+t^{-\frac12})}\  v_{s_iw},
\quad\hbox{for $i=1, \dots,k-1$}, \label{snormActionT}\\
T_0v_w &= [T_0]_{ww}v_w+
			 \sqrt{-([T_0]_{ww}-t_0^{\frac12})([T_0]_{ww}+t_0^{-\frac12})}\  v_{s_0w},
\label{snormActionX}
\end{align}
where $v_{s_iw} = 0$ if $s_iw\not\in \cF^{(\cc,J)}$, and
\begin{equation}
[T_i]_{ww} = \frac{t^{\frac12}-t^{-\frac12}}{1-\gamma_{w^{-1}(i)}\gamma^{-1}_{w^{-1}(i+1)}} 
\quad \text{ and } \quad 
[T_0]_{ww} = \frac{ (t_0^{\frac12}-t_0^{-\frac12}) + (t_k^{\frac12} -t_k^{-\frac12})\gamma^{-1}_{w^{-1}(1)}} 
{1-\gamma^{-2}_{w^{-1}(1)}}.
\end{equation}

\item  
The map
$$
\begin{matrix}
\CC^\times\times\{\hbox{skew local regions $(\cc, J)$}\}
&\longleftrightarrow 
&\{\hbox{irreducible calibrated $H_k^\ext$-modules}\} \\
(z,\cc, J) &\longmapsto &H_k^{(z,\cc, J)}
\end{matrix}
$$
\end{enumerate}
is a bijection.
\end{thm}

\begin{proof}  This result follows from \cite[Theorems 3.2 and 3.5]{Ra2}.  It is only necessary to establish that the
formulas in \eqref{snormActionPW}, \eqref{snormActionT}, and \eqref{snormActionX} are correct.  These are derived in a similar manner to
\cite[Proposition 3.3]{Ra2} as follows.  As in \cite[Theorem 3.2]{Ra2}, if $M$ is an irreducible calibrated $H_k^\ext$-module then
$$M = \bigoplus_{w\in \cW_0} M_{w\gamma}^{\mathrm{gen}},
\qquad \hbox{with $\dim(M_{w\gamma}^{\mathrm{gen}})=1$ if $M_{w\gamma}^{\mathrm{gen}}\ne 0$.}
$$
For $w\in \cW_0$, if $M_{w\gamma}^{\mathrm{gen}}\ne 0$, let $v_w$ be a nonzero vector in $M_{w\gamma}^{\mathrm{gen}}$; otherwise if $M_{w\gamma}^{\mathrm{gen}}=0$, let $v_w=0$. By \eqref{intertwinermaps},
$\tau_iv_w = [T_i]_{s_iw,w}v_{s_iw}$ for some constant $[T_i]_{s_iw,w}$ and 
the definition of $\tau_i$ in \eqref{intertwinerdefs} gives that
\begin{equation}
T_iv_w = \frac{t^{\frac12}-t^{-\frac12}}{1-\gamma_{w^{-1}(i)}\gamma^{-1}_{w^{-1}(i+1)}} v_w + [T_i]_{s_iw,w} v_{s_iw} 
\qquad\hbox{for $i=1, \ldots, k$,}
\label{Tpartway}
\end{equation}
and
\begin{equation}
T_0v_\gamma 
= \frac{ (t_0^{\frac12}-t_0^{-\frac12}) + (t_k^{\frac12} - t_k^{-\frac12})\gamma^{-1}_{w^{-1}(1)}} 
{1-\gamma^{-2}_{w^{-1}(1)}} v_w + [T_0]_{s_0w,w} v_{s_0w}.
\label{Xpartway}
\end{equation}
Thus $T_0$ is an operator on the subspace $\mathrm{span}_\CC\{ v_w, v_{s_0w}\}$
satisfying $(T_0-t_0^{\frac12})(T_0+t_0^{-\frac12}) = 0$  by \eqref{quadraticT0andTi}. Restricting to the action on $\mathrm{span}_\CC\{ v_w, v_{s_0w}\}$,   
the formulas in \eqref{snormActionX} now follow from the following argument about general $2\times 2$ matrices.

If a $2\times 2$ matrix $[T_0]$ has eigenvalues $\alpha_1$ and $\alpha_2$,
$$[T_0] = \begin{pmatrix} [T_0]_{ww} & [T_0]_{w,s_0w} \\
[T_0]_{s_0w,w} &[T_0]_{s_0w,s_0w} \end{pmatrix},
\qquad\hbox{then}\qquad ([T_0]-\alpha_1)([T_0]-\alpha_2) = 0$$
is the characteristic polynomial for $[T_0]$, and it follows that
\begin{align*}
\Tr([T_0]) &= [T_0]_{ww} + [T_0]_{s_0w,s_0w} = \alpha_1+\alpha_2,
\qquad\hbox{and} \\
\det([T_0]) &= [T_0]_{ww}[T_0]_{s_0w,s_0w}- [T_0]_{w,s_0w}[T_0]_{s_0w,w} = \alpha_1\alpha_2.
\end{align*}
Thus
\begin{align*}
-[T_0]_{w,s_0w}[T_0]_{s_0w,w}
&=\alpha_1\alpha_2 - [T_0]_{ww}[T_0]_{s_0w,s_0w}
=\alpha_1\alpha_2 - [T_0]_{ww}((\alpha_1+\alpha_2)-[T_0]_{ww}) \\
&=\alpha_1\alpha_2 - (\alpha_1+\alpha_2)[T_0]_{ww}+([T_0]_{ww})^2 
=([T_0]_{ww}-\alpha_1)([T_0]_{ww}-\alpha_2).
\end{align*}
Choosing a normalization of $v_{s_0 w}$ so that the matrix of $[T_0]$ is symmetric, we have $[T_0]_{w,s_0w} = [T_0]_{s_0w,w}$ and
\begin{align*}
[T_0]_{s_0w,w}
=\sqrt{([T_0]_{s_0w,w})^2}
=\sqrt{[T_0]_{w,s_0w}[T_0]_{s_0w,w}}
=\sqrt{-([T_0]_{ww}-\alpha_1)([T_0]_{ww}-\alpha_2)}.
\end{align*}

\end{proof}

\section{Classification of irreducible representations of $H_2$}\label{sec:H2Classification}

In this section we do a complete classification of the irreducible representations of 
the algebra $H_2^{\mathrm{ext}}$.  
An important reason for doing this classification of $H_2^\ext$ representations is to provide a sound basis for the definition of a skew local region (see the remarks immediately
after the definition of skew local region in \eqref{skewlocalregiondefn}).  
The classification and construction of calibrated representations of $H_k^{\mathrm{ext}}$
in terms of skew local regions in Theorem \ref{thm:calibconst} is important for determining the irreducible representations of 
$H_k^{\mathrm{ext}}$ that arise in the Schur-Weyl duality framework (see Theorem \ref{thm:partitions-to-SLRs}).  
We will do the clasification of irreducible $H_2^{\mathrm{ext}}$ representations 
under \textbf{genericity assumptions on the parameters}:
$t^{\frac12}$ is not a root of unity and
\begin{equation}
t_0^{\frac12}t_k^{\frac12},-t_0^{-\frac12}t_k^{\frac12}\not\in \{1, -1, t^{\pm\frac12}, -t^{\pm\frac12}, t^{\pm1}, -t^{\pm1}\}
\quad \hbox{and}\quad
t_0^{\frac12}t_k^{\frac12}\ne (-t_0^{-\frac12}t_k^{\frac12})^{\pm1}.
\label{eq:genericcond}
\end{equation}
More specifically, this condition is used for the (rank 2) computation in equation \eqref{nonKatononreg}.
Similar methods apply to the nongeneric cases but the final classification needs to be stated differently 
and we will not treat the nongeneric cases here.
The nongeneric case $t_0^{\frac12} = t_k^{\frac12}=t^{\frac12}$ is done in  \cite{Ra1, Ra2}
and \cite{Re}; the case where $t_0^{\frac12} = t_k^{\frac12} \neq t^{\frac12}$ appears in \cite{En} (see also \cite{KR}).

The algebra $H_2$ is generated by $W_1^{\pm1}, W_2^{\pm1}, T_0$, and $T_1$, and 
the Weyl group $\cW_0$ is generated by $s_0$ and $s_1$ with relations
$s_i^2 = 1$ and $s_0s_1s_0s_1 = s_1s_0s_1s_0$.
By \eqref{Hextasacentralextension},
$$H_2^{\mathrm{ext}} = \CC[W_0^{\pm}]\otimes H_2
\quad\hbox{as algebras,}
$$
and therefore it is sufficient to do the classification of irreducible representations of $H_2$. This is because all irreducible representations of  $\CC[W_0^{\pm1}]$ are one dimensional and determined by the image of $W_0$; 
and all irreducible representations of $H_2^{\mathrm{ext}}$ are the tensor product of an irreducible representation
of $\CC[W_0^{\pm1}]$ and an irreducible representation of $H_2$.

The group $\cW_0$ acts on $(\CC^\times)^2$ by
\begin{equation}
s_0 (\gamma_1,\gamma_2) = (\gamma_1^{-1},\gamma_2)
\qquad\hbox{and}\qquad s_1 (\gamma_1, \gamma_2) = (\gamma_2, \gamma_1).
\label{H2W0action}
\end{equation}
By \eqref{intertwinerdefs}, the intertwiners are 
$$
\tau_0 
= T_0 - \frac{ (t_0^{\frac12} - t_0^{-\frac12}) + (t_k^{\frac12} - t_k^{-\frac12})W_1^{-1} }
{1-W_1^{-2}} 
\qquad\hbox{and}\qquad
\tau_1 = T_1 - \frac{t^{\frac12} - t^{-\frac12} }{1-W_1W_2^{-1} }.
$$


\subsection{Classification of central characters}\label{H2centcharclass}

Following  \cite[\S 5]{Ra1},
the classification of irreducible $H_k^\ext$-modules begins with a classification of possible
pairs $(Z(\cc), P(\cc))=(Z(\gamma), P(\gamma))$ (where $\gamma$ and $\cc$ are related as in \eqref{eq:gamma-to-c}).  It is straightforward (though slightly tedious) to enumerate all the possibilities
by taking note of the following:
\begin{enumerate}
\item[(0)] Since $(Z(w\gamma), P(w\gamma)) = (wZ(\gamma), wP(\gamma))$, 
it is sufficient to do the analysis for a single representative $\gamma$ of each $\cW_0$-orbit on $(\CC^\times)^k$.
\item[(1)] The $\cW_0$-orbits of roots are $\{\pm \vep_1, \pm \vep_2\}$ and $\{ \pm(\vep_2\pm\vep_1)\}$,
and our preferred representative of the $\cW_0$-orbit will have $\vep_1$ or $\vep_2-\vep_1$ 
in $Z(\gamma)$ if $Z(\gamma)\ne\emptyset$.
\item[(2)] If $Z(\gamma)= \emptyset$ and $P(\gamma)\ne \emptyset$ then our preferred representative of the $\cW_0$-orbit will have $\vep_1$ or $\vep_2-\vep_1$ in $Z(\gamma)$. 
\end{enumerate}
With these preferences, the classification of $(Z(\gamma), P(\gamma))$ is accomplished by noting that
\begin{enumerate}[(a)]
\item if $\gamma\in \{(1,1), (-1,-1)\}$ then $(Z(\gamma), P(\gamma)) = (\{\vep_1, \vep_2, \vep_2\pm \vep_1\}, \emptyset)$;
\item if $\gamma\in \{(1,-1), (-1,1)\}$ then $(Z(\gamma), P(\gamma)) = (\{\vep_1, \vep_2\}, \emptyset)$;
\item $\vep_2-\vep_1\in Z(\gamma)$ if and only if $\gamma = (\gamma_1, \gamma_1)$;
\item $\vep_2+\vep_1\in Z(\gamma)$ if and only if $\gamma = (\gamma_1, \gamma_1^{-1})$;
\item $\vep_1\in Z(\gamma)$ if and only if $\gamma = (1, \gamma_2)$ or $\gamma=(-1, \gamma_2)$;
\item $\vep_2\in Z(\gamma)$ if and only if $\gamma = (\gamma_1, 1)$ or $\gamma=(\gamma_1,-1)$;
\item $\vep_1\in P(\gamma)$ if and only if $\gamma = (\gamma_1, \gamma_2)$ with 
$\gamma_1\in \{ t_0^{\frac12}t_k^{\frac12}, -t_0^{-\frac12}t_k^{\frac12}, -t_0^{\frac12}t_k^{-\frac12}, t_0^{-\frac12}t_k^{-\frac12}\}$;
\item $\vep_2-\vep_1\in P(\gamma)$ if and only if $\gamma = (\gamma_1, \gamma_2)$ with 
$\gamma_2 = \gamma_1t^{\pm1}$;
\item $\vep_2+\vep_1\in P(\gamma)$ if and only if $\gamma = (\gamma_1, \gamma_2)$ with 
$\gamma_1\gamma_2 = t^{\pm1}$.
\end{enumerate}
We shall freely use the conversion between $\gamma=(\gamma_1, \gamma_2)$ and $\cc = (c_1,c_2)$ given by
\eqref{eq:gamma-to-c},
$$\gamma_1 = -t^{c_1}, \quad \gamma_2=-t^{c_2},
\qquad\hbox{and write $(Z(\cc), P(\cc)) = (Z(\gamma), P(\gamma))$.}
$$
Representatives of the 12 possible $(Z(\cc), P(\cc))$ with $Z(\cc)=\emptyset$ are displayed in Figure \ref{fig:rank2regulars}.
Representatives of the 9 possible $(Z(\cc), P(\cc))$ with $Z(\cc)\ne\emptyset$ are displayed in Figure \ref{fig:rank2nonregulars}.
It works out that, in each case, the pair $(Z(\cc), P(\cc))$ is attained by an element $\cc$ that has real coordinates
(the one complex character in the equal parameter case that behaves differently from the real characters, namely the point $t_b$ in \cite[Figure 5.1]{Ra1}, 
does not appear in the generic unequal parameter case assumed in \eqref{eq:genericcond}).

With notation as at the beginning of Section 3, in Figures \ref{fig:rank2regulars} and \ref{fig:rank2nonregulars}, 
the fundamental region $C$ is the shaded area, the solid lines are the hyperplanes
$\fh^{\alpha}$ for $\alpha\in R^+$, and the dotted hyperplanes are labeled by the
equation that defines them.
If $\cc = (c_1, c_2)\in C$, so that $0\le c_1\le c_2$, then
\begin{align*}
Z(\cc) = \{ \hbox{solid hyperplanes through $\cc$}\}
\quad\hbox{and}\quad
P(\cc) = \{ \hbox{dotted hyperplanes through $\cc$}\}.
\end{align*}
The bijection
\begin{equation}
\begin{matrix}
\cW_0 &\leftrightarrow &\{ \hbox{chambers}\} \\
w &\mapsto &w^{-1}C
\end{matrix}
\qquad\hbox{identifies each $\cF^{(\cc, J)}$ with a set of chambers,}
\label{FcJaschambers}
\end{equation}
a \emph{local region} in $\fh_\RR^*$.  As illustrated by the example at the bottom right
of Figures \ref{fig:rank2regulars} and \ref{fig:rank2nonregulars}, $\cF^{(\cc, J)}$ is identified with the
set of chambers that are on the negative side of the hyperplanes in $J$ and
on the positive side of the hyperplanes in $P(\cc)-J$ .
For each $(\cc, J)$ the corresponding configuration of boxes $\kappa$ is displayed
in the local region of chambers corresponding to the elements of $\cF^{(\cc,J)}$ by \eqref{FcJaschambers}.
In Figure \ref{fig:rank2regulars}, only the boxes on positive diagonals are shown, since they determine the entire doubled configuration when $Z(\cc) = \emptyset$. 
The diagram at the bottom right of each figure gives an example of the correspondence between chambers corresponding to 
$\cF{(\cc,J)}$, the elements of  $\cF^{(\cc,J)}$, and the standard fillings of the corresponding configuration of 
boxes $\kappa$: the point $\cc = (r_1-1, r_1)$ in the bottom right of Figure \ref{fig:rank2regulars}, and
the point $\cc = (0,1)$ in the bottom right of Figure \ref{fig:rank2nonregulars}.

In  Figure  \ref{fig:rank2nonregulars}, the small graphs nearby each marked $\cc=(c_1,c_2)$ indicate the structure (generalized
weight spaces and intertwiner maps) of the irreducible modules $M$ of central character $\cc$.   This structure is determined 
below in Section \ref{H2nonregirredconst}.
There is a vertex in the chamber $w^{-1}C$ for each element of a basis of $M_{w\cc}^{\mathrm{gen}}$
and there is an edge if the matrix of $\tau_i$ (or $T_i$ if $\tau_i$ is not defined on $M_{w\cc}^{\mathrm{gen}}$)
is nonzero in the entry corresponding to the two vertices that are connected.

{\def\RTwo{2.6}
\def\ROne{1.2}
\def\ONE{.5}
\def\TOP{3.75}
\begin{figure}\caption{Regular central characters in rank 2. See the description in Section \ref{H2centcharclass}}\label{fig:rank2regulars}
\centerline{\begin{tikzpicture}[scale=3.25, every node/.style={inner sep=3pt}]
	\filldraw[gray!15] (0,\TOP)--(0,0)--(\TOP,\TOP)--(0,\TOP);
	\node[WhiteCircle, inner sep=17.5pt](r1r2) at (\ROne,\RTwo){};		
	\node[B] at (r1r2){};
		\node[above right] at (r1r2) {
			\begin{tikzpicture}[BoxArr]
			\BOX{0,0}{}\BOX{2,0}{}\bDOT{0,0}\bDOT{2,0}
			\end{tikzpicture}
			};		
		\node[above left] at (r1r2) {
			\begin{tikzpicture}[BoxArr]
				\BOX{0,0}{}\BOX{2,0}{}\bDOT{1,1}\bDOT{2,0}
			\end{tikzpicture}
			};
		\node[below right] at (r1r2) {
			\begin{tikzpicture}[BoxArr]
			\BOX{0,0}{}\BOX{2,0}{}\bDOT{0,0}\bDOT{3,1}
			\end{tikzpicture}
			};
		\node[below left] at (r1r2) {
			\begin{tikzpicture}[BoxArr]
				\BOX{0,0}{}\BOX{2,0}{}\bDOT{1,1}\bDOT{3,1}
			\end{tikzpicture}
			};
	\node[WhiteCircle, inner sep=17.5pt](r1-1r1) at (\ROne-\ONE,\ROne){};		
	\node[B] at (r1-1r1){};
		\node[above] at (r1-1r1) {
			\begin{tikzpicture}[BoxArr]
			\BOX{0,0}{}\BOX{1,0}{}\bDOT{1,0}
			\node at (3,1){};
			\end{tikzpicture}
			};	
		\node[above right, inner sep=0pt] at (r1-1r1) {
			\begin{tikzpicture}[BoxArr]
			\BOX{0,0}{}\BOX{0,1}{}\bDOT{0,0}
			\node at (-2.5,2){};
			\end{tikzpicture}
			};	
		\node[below right] at (r1-1r1) {
			\begin{tikzpicture}[BoxArr]
			\BOX{0,0}{}\BOX{0,1}{}\bDOT{1,1}
			\end{tikzpicture}
			};			
		\node[below left, inner sep=0pt] at (r1-1r1) {
			\begin{tikzpicture}[BoxArr]
			\BOX{0,0}{}\BOX{1,0}{}\bDOT{2,1}
			\node at (3.75,-.2){};
			\end{tikzpicture}
			};	
	\node[WhiteCircle, inner sep=17.5pt](r2-1r2) at (\RTwo-\ONE,\RTwo){};		
	\node[B] at (r2-1r2){};
		\node[above] at (r2-1r2) {
			\begin{tikzpicture}[BoxArr]
			\BOX{0,0}{}\BOX{1,0}{}\bDOT{1,0}
			\node at (3,1){};
			\end{tikzpicture}
			};	
		\node[above right, inner sep=0pt] at (r2-1r2) {
			\begin{tikzpicture}[BoxArr]
			\BOX{0,0}{}\BOX{0,1}{}\bDOT{0,0}
			\node at (-2.5,2){};
			\end{tikzpicture}
			};	
		\node[below right] at (r2-1r2) {
			\begin{tikzpicture}[BoxArr]
			\BOX{0,0}{}\BOX{0,1}{}\bDOT{1,1}
			\end{tikzpicture}
			};			
		\node[below left, inner sep=0pt] at (r2-1r2) {
			\begin{tikzpicture}[BoxArr]
			\BOX{0,0}{}\BOX{1,0}{}\bDOT{2,1}
			\node at (3.75,-.2){};
			\end{tikzpicture}
			};
	\node[WhiteCircle, inner sep=17.5pt](r1r1+1)  at (\ROne,\ROne+\ONE)  {};		
	\node[B] at (r1r1+1) {};
		\node[left] at (r1r1+1) {
			\begin{tikzpicture}[BoxArr]
			\BOX{0,0}{}\BOX{1,0}{}\bDOT{1,1}
			\node at (2.5,1){};
			\end{tikzpicture}
			};	
		\node[below left, inner sep=0pt] at (r1r1+1) {
			\begin{tikzpicture}[BoxArr]
			\BOX{0,0}{}\BOX{0,1}{}\bDOT{1,2}
			\node at (1.3,-1.75){};
			\end{tikzpicture}
			};	
		\node[below right] at (r1r1+1) {
			\begin{tikzpicture}[BoxArr]
			\BOX{0,0}{}\BOX{0,1}{}\bDOT{0,1}
			\end{tikzpicture}
			};			
		\node[above right, inner sep=0pt] at (r1r1+1) {
			\begin{tikzpicture}[BoxArr]
			\BOX{0,0}{}\BOX{1,0}{}\bDOT{0,0}
			\node at (-.25,3.5){};
			\end{tikzpicture}
			};	
	\node[WhiteCircle, inner sep=17.5pt](r2r2+1) at (\RTwo,\RTwo+\ONE) {};		
	\node[B] at  (r2r2+1){};	
		\node[left] at (r2r2+1) {
			\begin{tikzpicture}[BoxArr]
			\BOX{0,0}{}\BOX{1,0}{}\bDOT{1,1}
			\node at (2.5,1){};
			\end{tikzpicture}
			};	
		\node[below left, inner sep=0pt] at (r2r2+1) {
			\begin{tikzpicture}[BoxArr]
			\BOX{0,0}{}\BOX{0,1}{}\bDOT{1,2}
			\node at (1.3,-1.75){};
			\end{tikzpicture}
			};	
		\node[below right] at (r2r2+1) {
			\begin{tikzpicture}[BoxArr]
			\BOX{0,0}{}\BOX{0,1}{}\bDOT{0,1}
			\end{tikzpicture}
			};			
		\node[above right, inner sep=0pt] at (r2r2+1) {
			\begin{tikzpicture}[BoxArr]
			\BOX{0,0}{}\BOX{1,0}{}\bDOT{0,0}
			\node at (-.25,3.5){};
			\end{tikzpicture}
			};	
	\node[WhiteCircle, inner sep=14pt](c1c1+1) at (.5*\ROne+.5*\RTwo - .5*\ONE,(.5*\ROne+.5*\RTwo + .5*\ONE) {};		
	\node[B] at (c1c1+1){};
		\node[below right] at (c1c1+1) {
			\begin{tikzpicture}[BoxArr]
			\BOX{0,0}{}\BOX{0,1}{}
			\end{tikzpicture}
			};
		\node[above left] at (c1c1+1) {
			\begin{tikzpicture}[BoxArr]
			\BOX{0,0}{}\BOX{1,0}{}
			\end{tikzpicture}
			};
	\node[WhiteCircle, inner sep=11pt](c1r2) at (.5,\RTwo) {};		
	\node[B] at (c1r2){};
		\node[below] at (c1r2) {
		\begin{tikzpicture}[BoxArr]
			\BOX{0,0}{}\BOX{2,0}{}\bDOT{3,1}
		\end{tikzpicture}
		};
		\node[above] at (c1r2) {
		\begin{tikzpicture}[BoxArr]
			\BOX{0,0}{}\BOX{2,0}{}\bDOT{2,0}
		\end{tikzpicture}
		};
	\node[WhiteCircle, inner sep=11pt](c1r1) at (.2,\ROne){};		
	\node[B] at (c1r1){};
		\node[below] at (c1r1) {
			\begin{tikzpicture}[BoxArr]
				\BOX{0,0}{}\BOX{2,0}{}\bDOT{3,1}
			\end{tikzpicture}
			};
		\node[above] at (c1r1) {
			\begin{tikzpicture}[BoxArr]
			\BOX{0,0}{}\BOX{2,0}{}\bDOT{2,0}
			\end{tikzpicture}
			};
	\node[WhiteCircle, inner sep=11pt](r1c2) at (\ROne,.5*\RTwo+.5*\TOP){};		
	\node[B] at (r1c2) {};
		\node[left] at (r1c2) {
			\begin{tikzpicture}[BoxArr]
				\BOX{0,0}{}\BOX{0,-2}{}\bDOT{1,1}
			\end{tikzpicture}
			};
		\node[right] at (r1c2) {
			\begin{tikzpicture}[BoxArr]
				\BOX{0,0}{}\BOX{0,-2}{}\bDOT{0,0}
			\end{tikzpicture}
			};
	\node[WhiteCircle, inner sep=11pt](r2c2) at (\RTwo,\TOP-.15){};		
	\node[B] at (r2c2){};
		\node[left] at (r2c2) {
			\begin{tikzpicture}[BoxArr]
				\BOX{0,0}{}\BOX{0,-2}{}\bDOT{1,1}
			\end{tikzpicture}
			};
		\node[right] at (r2c2) {
			\begin{tikzpicture}[BoxArr]
				\BOX{0,0}{}\BOX{0,-2}{}\bDOT{0,0}
			\end{tikzpicture}
			};
	\node[WhiteCircle, inner sep=11pt](c1c2) at (.5*\ROne + .5*\RTwo,.5*\TOP+.5*\RTwo){};		
	\node[B] at (c1c2){};
			\node[below, outer sep=2pt] at (c1c2) {
			\begin{tikzpicture}[BoxArr]
				\BOX{0,0}{}\BOX{2,0}{}
			\end{tikzpicture}
			};
	\node[WhiteCircle, inner sep=14pt](c1-c1+1) at (.25*\ONE,.75*\ONE) {};		
	\node[B] at (c1-c1+1){};
		\node[below left] at (c1-c1+1) {
			\begin{tikzpicture}[BoxArr]
			\BOX{0,0}{}\BOX{0,1}{}
			\end{tikzpicture}
			};
		\node[above right] at (c1-c1+1) {
			\begin{tikzpicture}[BoxArr]
			\BOX{0,0}{}\BOX{1,0}{}
			\end{tikzpicture}
			};
	\draw[thick] (0,-.2)--(0,\TOP) node[above] {$c_1=0$};
	\draw[thick] (-.2,-.2)--(\TOP,\TOP) node[above right] {$c_1=c_2$};
	\draw[thick] (-.2,0)--(2,0) node[right] {$c_2=0$};
	\draw[dotted, thick] (\TOP-\ONE,\TOP)--(-.2,-.2+\ONE) node[below left]{\small$c_2=c_1+1$};
	\draw[dotted, thick] (-.2,\ONE+.2)--(\ONE+.2,-.2) node[below right]{\small$c_2=-c_1+1$};
	\draw[dotted, thick] (\ROne+.5,\ROne)--(-.2,\ROne) node[left] {\small$c_2=r_1$};
	\draw[dotted, thick] (\RTwo+.5,\RTwo)--(-.2,\RTwo) node[left] {\small$c_2=r_2$};
	\draw[dotted, thick] (\ROne,\ROne-.5)--(\ROne,\TOP) node[above] {\small$c_1=r_1$};
	\draw[dotted, thick] (\RTwo,\RTwo-.5)--(\RTwo,\TOP) node[above] {\small$c_1=r_2$};
	\node at (3,1){
		\begin{tikzpicture}[scale=1.5]
			\draw[black!25] (-45:-1.2)--(-45:1.2) ;
			\draw[dotted, thick] (45:-2)--(45:2);
			\draw[dotted, thick] (-2,0)--(2,0);
			\draw[black!25]  (0,-1.2)--(0,1.2);
			\node at  (112.5 - 45:.9) {\scriptsize$1$};
			\node at  (112.5 - 45:1.5) {
				\begin{tikzpicture}[xscale = .3, yscale=-.3, every node/.style={ inner sep=1pt}]
					\BOX{0,-2}{1}\BOX{1,-2}{2}  
					\BOX{0,0}{-1} \BOX{-1,0}{-2} 
					\DOT{1,-2} \DOT{0,1}
				\end{tikzpicture}};
			\node at  (112.5 - 2*45:.9) {\scriptsize$s_1$};
			\node at  (112.5 - 2*45:1.5) {
				\begin{tikzpicture}[xscale = .3, yscale=-.3, every node/.style={ inner sep=1pt}]
					\BOX{0,-2}{2}\BOX{0,-3}{1}  
					\BOX{-2,-2}{-2} \BOX{-2,-1}{-1} 
					\DOT{0,-3} \DOT{-1,0}
				\end{tikzpicture}};
			\node at  (112.5 - 3*45:.9) {\scriptsize$s_0 s_1$};
			\node at  (112.5 - 3*45:1.5) {
				\begin{tikzpicture}[xscale = .3, yscale=-.3, every node/.style={ inner sep=1pt}]
				\BOX{0,-2}{2}\BOX{0,-3}{-1}  
				\BOX{-2,-2}{-2} \BOX{-2,-1}{1} 
				\DOT{1,-2} \DOT{-2,-1}
			\end{tikzpicture}};
			\node at  (112.5 - 4*45:.9) {\scriptsize$~s_1s_0s_1$};
			\node at  (112.5 - 4*45:1.5) {
				\begin{tikzpicture}[xscale = .3, yscale=-.3, every node/.style={ inner sep=1pt}]
					\BOX{0,-2}{1}\BOX{0,-3}{-2}  
					\BOX{-2,-2}{-1} \BOX{-2,-1}{2} 
					\DOT{1,-2} \DOT{-2,-1}
				\end{tikzpicture}};
			\node at  (67.5 + 1*45:.9) {\scriptsize$s_0$}; 
			\node at  (67.5 + 1*45:1.5) {
				\begin{tikzpicture}[xscale = .3, yscale=-.3, every node/.style={ inner sep=1pt}]
					\BOX{0,-2}{-1}\BOX{1,-2}{2}  
					\BOX{0,0}{1} \BOX{-1,0}{-2} 
					\DOT{1,-2} \DOT{0,1}
				\end{tikzpicture}}; 
			\node at  (67.5 + 2*45:.9) {\scriptsize$s_1s_0$}; 
			\node at  (67.5 + 2*45:1.5) {
					\begin{tikzpicture}[xscale = .3, yscale=-.3, every node/.style={ inner sep=1pt}]
					\BOX{0,-2}{-2}\BOX{1,-2}{1}  
					\BOX{0,0}{2} \BOX{-1,0}{-1} 
					\DOT{1,-2} \DOT{0,1}
				\end{tikzpicture}
			}; 
			\node at  (67.5 + 3*45:.9) {\scriptsize$s_0s_1s_0$}; 
			\node at  (67.5 + 3*45:1.65) {
				\begin{tikzpicture}[xscale = .3, yscale=-.3, every node/.style={ inner sep=1pt}]
					\BOX{0,-2}{-2}\BOX{1,-2}{-1}  
					\BOX{0,0}{2} \BOX{-1,0}{1} 
					\DOT{2,-1} \DOT{-1,0}
				\end{tikzpicture} }; 
			\node at  (67.5 + 4*45:.9) {\scriptsize$s_1s_0s_1s_0~$}; 
			\node at  (67.5 + 4*45:1.5) {
				\begin{tikzpicture}[xscale = .3, yscale=-.3, every node/.style={ inner sep=1pt}]
					\BOX{0,-2}{-1}\BOX{0,-3}{-2}  \BOX{-2,-2}{1} \BOX{-2,-1}{2} 
					\DOT{1,-2} \DOT{-2,-1}
				\end{tikzpicture}}; 
			\node[B] at (0,0){};
			\foreach \x [count=\c from 1] in {1, ..., 9}{\node[V] (a\x) at (112.5 - \c*45:.5){}; }
			\foreach \x [count=\c from 4] in { 3, 4}{ \draw (a\x)--(a\c);}
			\foreach \x [count=\c from 8] in {7,8}{\draw (a\x)--(a\c);}
			\node at  (90+40:2) {\small$J = \emptyset$};
			\node[above] at  (0:2.5) {\small$J = \{\vep_2-\vep_1\}$};
			\node[below] at  (185:2.2) {\small$J = \{\vep_2\}$};
			\node at  (-90+22.5:2.2) {\small$J = \{\vep_2, \vep_2-\vep_1\}$};
		\end{tikzpicture}};
\end{tikzpicture}}
\end{figure}}

{\def\RTwo{3.25}
\def\ROne{1.6}
\def\ONE{.85}
\def\TOP{3.75}
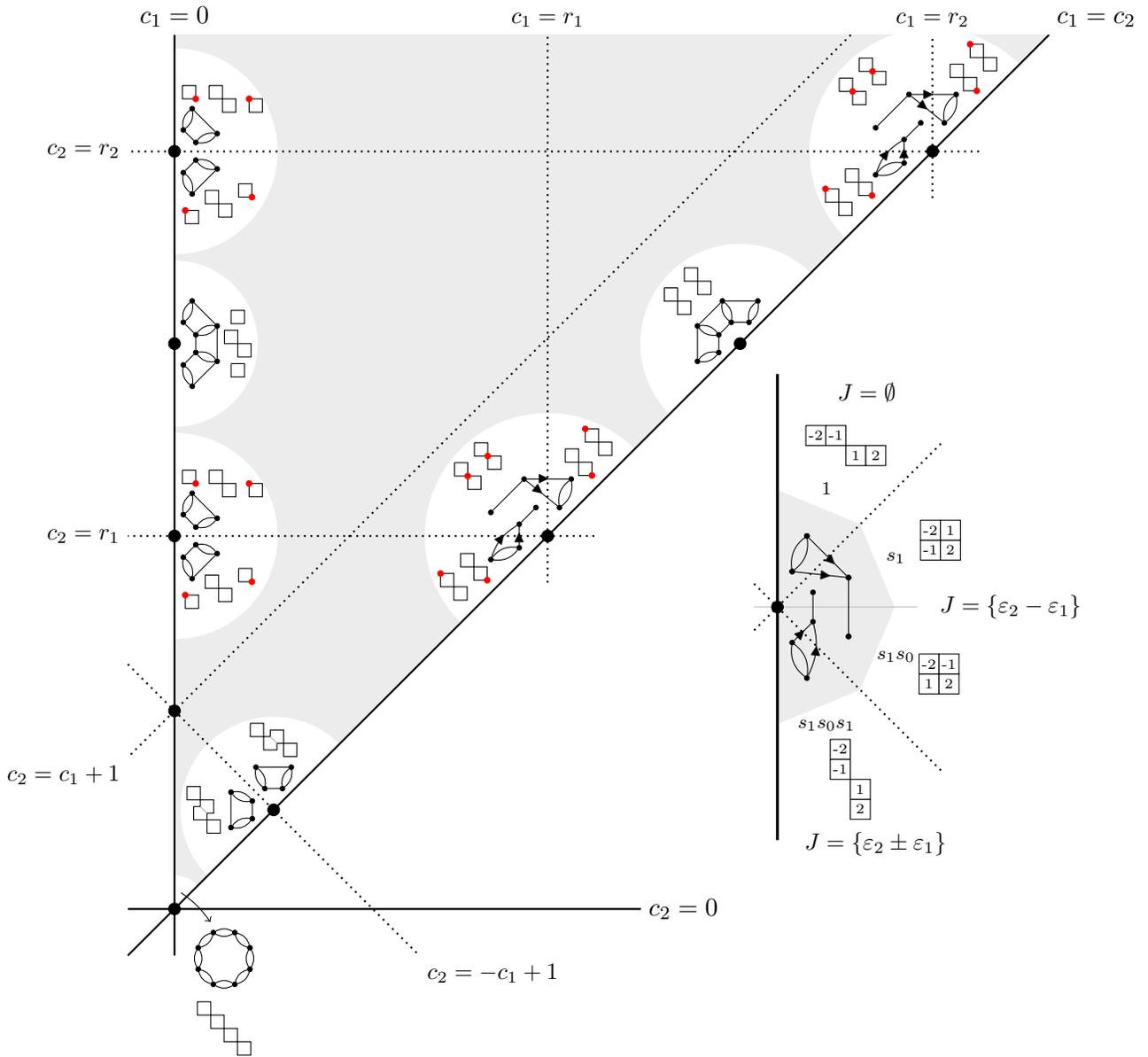
\begin{figure}\caption{Non-regular points}\label{fig:rank2nonregulars}
\centerline{\begin{tikzpicture}[scale=3.5, every node/.style={inner sep=3pt}]
	\filldraw[gray!15] (0,\TOP)--(0,0)--(\TOP,\TOP)--(0,\TOP);
	\coordinate (0r2) at (0,\RTwo);
	\begin{scope}[shift=(0r2), scale=.3]
		\node[WhiteCircle, inner sep=31pt] at (0,0){};		
		\node[B] at (0,0){};
		\foreach \x [count=\c from 1] in {1, ..., 4}{
			\node[V] (a\x) at (112.5 - \c*45:.66){}; 
			\node[V] (b\x) at (112.5 - \c*45:.33){}; }
		\foreach \x in {1, ..., 4}{\draw (b\x) to [bend left=30] (a\x);\draw (b\x)to [bend right=30] (a\x);}
		\draw (b1)--(b2) (a1)--(a2);
		\draw (b3)--(b4) (a3)--(a4);
		\node[right] at (0,.75){
			\begin{tikzpicture}[xscale = .2, yscale=-.2]
				\BOX{-3,-1}{} \BOX{-1,-1}{} \BOX{0,0}{}\BOX{2,0}{}  
				\bDOT{-2,0} \bDOT{2,0}
			\end{tikzpicture}};
		\node[right] at (0,-.75){
			\begin{tikzpicture}[xscale = .2, yscale=-.2]
				\BOX{-2.5,.5}{} \BOX{-1,-1}{} \BOX{0,0}{}\BOX{1.5,-1.5}{}  
				\bDOT{-2.5,.5} \bDOT{2.5,-.5}
			\end{tikzpicture}};
	\end{scope}
	\coordinate (0c2) at (0,.5*\ROne + .5*\RTwo);
	\begin{scope}[shift=(0c2), scale=.3]
		\node[WhiteCircle, inner sep=25pt] at (0,0){};		
		\node[B] at (0,0){};
		\foreach \x [count=\c from 1] in {1, ..., 4}{
			\node[V] (a\x) at (112.5 - \c*45:.66){}; 
			\node[V] (b\x) at (112.5 - \c*45:.33){}; }
		\foreach \x in {1, ..., 4}{\draw (b\x) to [bend left=30] (a\x);\draw (b\x)to [bend right=30] (a\x);}
		\draw (b1)--(b2) (a1)--(a2);
		\draw (b3)--(b2) (a3)--(a2);
		\draw (b3)--(b4) (a3)--(a4);
		\node[right] at (.6,0){
			\begin{tikzpicture}[xscale = .2, yscale=-.2]
				\BOX{-.5,1.5}{} \BOX{-1,-1}{} \BOX{0,0}{}\BOX{-.5,-2.5}{}  
			\end{tikzpicture}
		};
	\end{scope}
	\coordinate (0r1) at (0,\ROne);
	\begin{scope}[shift=(0r1), scale=.3]
		\node[WhiteCircle, inner sep=31pt] at (0,0){};		
		\node[B] at (0,0){};
		\foreach \x [count=\c from 1] in {1, ..., 4}{
			\node[V] (a\x) at (112.5 - \c*45:.66){}; 
			\node[V] (b\x) at (112.5 - \c*45:.33){}; }
		\foreach \x in {1, ..., 4}{\draw (b\x) to [bend left=30] (a\x);\draw (b\x)to [bend right=30] (a\x);}
		\draw (b1)--(b2) (a1)--(a2);
		\draw (b3)--(b4) (a3)--(a4);
		\node[right] at (0,.75){
			\begin{tikzpicture}[xscale = .2, yscale=-.2]
				\BOX{-3,-1}{} \BOX{-1,-1}{} \BOX{0,0}{}\BOX{2,0}{}  
				\bDOT{-2,0} \bDOT{2,0}
			\end{tikzpicture}};
		\node[right] at (0,-.75){
			\begin{tikzpicture}[xscale = .2, yscale=-.2]
				\BOX{-2.5,.5}{} \BOX{-1,-1}{} \BOX{0,0}{}\BOX{1.5,-1.5}{}  
				\bDOT{-2.5,.5} \bDOT{2.5,-.5}
			\end{tikzpicture}};
	\end{scope}
	\coordinate (halfhalf) at (.5*\ONE,.5*\ONE);
	\begin{scope}[shift=(halfhalf), scale=.3]
		\node[WhiteCircle, inner sep=28pt] at (0,0){};		
		\node[B] at (0,0){};
		\foreach \x [count=\c from 1] in {1, ..., 4}{
			\node[V] (a\x) at (22.5 + \c*45:.66){}; 
			\node[V] (b\x) at (22.5 + \c*45:.33){}; }
		\foreach \x in {1,..., 4}{
			\draw (b\x) to [bend left = 30] (a\x);
			\draw (b\x) to [bend right = 30]   (a\x);}
		\draw(a1) to (a2) (b1) to  (b2);
		\draw(a3) to (a4) (b3) to  (b4);
		\node[] at (0,1){
			\begin{tikzpicture}[xscale = .2, yscale=-.2]
				\draw[fill=white]	 (0,0)--(.5,0)--(.5,-.5)--(1.5,-.5)--(1.5,.5)--(1,.5)--(1,1)--(0,1)--(0,0);
				\draw[thin, black!30] (.5,0)--(1,.5);
				 \BOX{-1,-1}{}
				 \BOX{1.5,.5}{}
			\end{tikzpicture}};
		\node[] at (-1,0){
			\begin{tikzpicture}[xscale = .2, yscale=-.2]
				\draw[fill=white]	 (0,0)--(.5,0)--(.5,-.5)--(1.5,-.5)--(1.5,.5)--(1,.5)--(1,1)--(0,1)--(0,0);
				\draw[thin, black!30] (.5,0)--(1,.5);
				 \BOX{-.5,-1.5}{}
				 \BOX{1,1}{}
			\end{tikzpicture}};
	\end{scope}	
	\coordinate (r1r1) at (\ROne,\ROne);
	\begin{scope}[shift=(r1r1), scale=.4]
		\node[WhiteCircle, inner sep=37pt] at (0,0){};		
		\node[B] at (0,0){};
		\foreach \x [count=\c from 1] in {1, ..., 4}{
			\node[V] (a\x) at (22.5 + \c*45:.66){}; 
			\node[V] (b\x) at (22.5 + \c*45:.33){}; }
		\foreach \x in {1, 4}{
			\draw (b\x) to [bend left = 30]  (a\x);
			\draw (b\x) to [bend right = 30]  (a\x);}
		\draw(a1) to node[pos=.6, ]{\tiny$\blacktriangleright$} (a2)
			(b1) to node[pos=.6, sloped]{\tiny$\blacktriangleright$} (a2);
		\draw (a2) to (a3);
		\draw (b2) to  (b3);
		\draw (a4)  to [bend left=20] 
			node[pos=.5, sloped]{\tiny$\blacktriangleright$} (b3)
		 (b4) to 	node[pos=.4, sloped]{\tiny$\blacktriangleright$} (b3);
		\node at  (90-26:1){
			\begin{tikzpicture}[xscale = .2, yscale=-.2]
				\BOX{-.5,-1.5}{} \BOX{.5,-.5}{}
				\BOX{-1,1}{} \BOX{-2,0}{}
				\bDOT{-.5,-1.5} \bDOT{0,2}
			\end{tikzpicture}};
		\node[] at (180+26:1){
			\begin{tikzpicture}[xscale = .2, yscale=-.2]
				\BOX{-.5,-1.5}{} \BOX{.5,-.5}{}
				\BOX{-1,1}{} \BOX{-2,0}{}
				\bDOT{1.5,.5} \bDOT{-2,0}
			\end{tikzpicture}};
		\node[] at (-.75,.75){
			\begin{tikzpicture}[xscale = .2, yscale=-.2]
				\BOX{-.5,-1.5}{} \BOX{.5,-.5}{}
				\BOX{-1,1}{} \BOX{-2,0}{}
				\bDOT{.5,-.5} \bDOT{-1,1}
			\end{tikzpicture}};
		\end{scope}	
	\coordinate (c1c1) at (.5*\ROne + .5*\RTwo,.5*\ROne + .5*\RTwo);
	\begin{scope}[shift=(c1c1), scale=.3]
		\node[WhiteCircle, inner sep=30pt] at (0,0){};		
		\node[B] at (0,0){};
		\foreach \x [count=\c from 1] in {1, ..., 4}{
			\node[V] (a\x) at (22.5 + \c*45:.66){}; 
			\node[V] (b\x) at (22.5 + \c*45:.33){}; }
		\foreach \x in {1, ..., 4}{\draw (b\x) to [bend left=30] (a\x);\draw (b\x)to [bend right=30] (a\x);}
		\draw (b1)--(b2) (a1)--(a2);
		\draw (b3)--(b2) (a3)--(a2);
		\draw (b3)--(b4) (a3)--(a4);
		\node at (-.75,.75){
			\begin{tikzpicture}[xscale = .2, yscale=-.2]
				\BOX{-1.5,1.5}{}\BOX{-2.5,.5}{} \BOX{-1,-1}{} \BOX{0,0}{}
			\end{tikzpicture}
		};
	\end{scope}
	\coordinate (r2r2) at (\RTwo,\RTwo);
	\begin{scope}[shift=(r2r2), scale=.4]
		\node[WhiteCircle, inner sep=37pt] at (0,0){};		
		\node[B] at (0,0){};
		\foreach \x [count=\c from 1] in {1, ..., 4}{
			\node[V] (a\x) at (22.5 + \c*45:.66){}; 
			\node[V] (b\x) at (22.5 + \c*45:.33){}; }
		\foreach \x in {1, 4}{
			\draw (b\x) to [bend left = 30]  (a\x);
			\draw (b\x) to [bend right = 30]  (a\x);}
		\draw(a1) to node[pos=.6, ]{\tiny$\blacktriangleright$} (a2)
			(b1) to node[pos=.6, sloped]{\tiny$\blacktriangleright$} (a2);
		\draw (a2) to (a3);
		\draw (b2) to  (b3);
		\draw (a4)  to [bend left=20] 
			node[pos=.5, sloped]{\tiny$\blacktriangleright$} (b3)
		 (b4) to 	node[pos=.4, sloped]{\tiny$\blacktriangleright$} (b3);
		\node at  (90-26:1){
			\begin{tikzpicture}[xscale = .2, yscale=-.2]
				\BOX{-.5,-1.5}{} \BOX{.5,-.5}{}
				\BOX{-1,1}{} \BOX{-2,0}{}
				\bDOT{-.5,-1.5} \bDOT{0,2}
			\end{tikzpicture}};
		\node[] at (180+26:1){
			\begin{tikzpicture}[xscale = .2, yscale=-.2]
				\BOX{-.5,-1.5}{} \BOX{.5,-.5}{}
				\BOX{-1,1}{} \BOX{-2,0}{}
				\bDOT{1.5,.5} \bDOT{-2,0}
			\end{tikzpicture}};
		\node[] at (-.75,.75){
			\begin{tikzpicture}[xscale = .2, yscale=-.2]
				\BOX{-.5,-1.5}{} \BOX{.5,-.5}{}
				\BOX{-1,1}{} \BOX{-2,0}{}
				\bDOT{.5,-.5} \bDOT{-1,1}
			\end{tikzpicture}};
		\end{scope}	
	\node[WhiteCircle, inner sep=10pt] at (0,0){};
	\node[B] at (0,0){};	
	\draw[->, bend left=10] (90*.75:.075) to +(.125,-.125);
	\coordinate (00a) at (.25*\ONE,-.25*\ONE);
	\begin{scope}[shift=(00a), scale=.4]
			\foreach \x in {1, ..., 8}{
			\node[V] (a\x) at (22.5 + \x*45:.3){};  }
			\foreach \x [count=\c from 2] in {1, ..., 7}{
			\draw[bend left] (a\x)to(a\c);\draw[bend right] (a\x)to(a\c);}
			\draw[bend left] (a8)to(a1);\draw[bend right] (a8)to(a1);
		\node[] at (0,-.75){
			\begin{tikzpicture}[xscale = .2, yscale=-.2]
			\BOX{-2,-2}{} \BOX{-1,-1}{}
				\BOX{0,0}{} \BOX{1,1}{}
			\end{tikzpicture}};
	\end{scope}
	\node[B] at (0,\ONE){}	;	
	\draw[thick] (0,-.2)--(0,\TOP) node[above] {$c_1=0$};
	\draw[thick] (-.2,-.2)--(\TOP,\TOP) node[above right] {$c_1=c_2$};
	\draw[thick] (-.2,0)--(2,0) node[right] {$c_2=0$};
	\draw[dotted, thick] (\TOP-\ONE,\TOP)--(-.2,-.2+\ONE) node[below left]{\small$c_2=c_1+1$};
	\draw[dotted, thick] (-.2,\ONE+.2)--(\ONE+.2,-.2) node[below right]{\small$c_2=-c_1+1$};
	\draw[dotted, thick] (\ROne+.2,\ROne)--(-.2,\ROne) node[left] {\small$c_2=r_1$};
	\draw[dotted, thick] (\RTwo+.2,\RTwo)--(-.2,\RTwo) node[left] {\small$c_2=r_2$};
	\draw[dotted, thick] (\ROne,\ROne-.2)--(\ROne,\TOP) node[above] {\small$c_1=r_1$};
	\draw[dotted, thick] (\RTwo,\RTwo-.2)--(\RTwo,\TOP) node[above] {\small$c_1=r_2$};
	\node at (3.2,1.25){
		\begin{tikzpicture}[scale=1.75]
			\filldraw[gray!15] (0,1)--(45:1)--(1,0)--(-45:1)--(0,-1);
			\draw[dotted, thick] (-.2,.2)--(-45:2) (-.2,-.2)--(45:2);
			\draw[very thick] (0,-2)--(0,2);
			\draw[black!25] (-.2,0)--(1.2,0);
			\node[B] at (0,0){};
			\node at  (112.5 - 45:1.1) {\scriptsize$1$};
			\node at  (112.5 - 2*45:1.1) {\scriptsize$s_1$};
			\node at  (112.5 - 3*45:1.1) {\scriptsize$s_1 s_0$};
			\node at  (112.5 - 4*45:1.1) {\scriptsize$s_1s_0s_1$}; 
			\foreach \x [count=\c from 1] in {1, ..., 4}{
				\node[V] (a\x) at (112.5 - \c*45:.66){}; 
				\node[V] (b\x) at (112.5 - \c*45:.33){}; }
			\node[V] at (b1) {};
			\node[V] at (a1) {};
			\node[V] at (a2) {};
			\node[V] at (a3) {};
			\foreach \x in {1, 4}{
				\draw (b\x) to [bend left = 30]  (a\x);
				\draw (b\x) to [bend right = 30]  (a\x);}
			\draw(a1) to node[pos=.6, sloped]{\tiny$\blacktriangleright$} (a2)
				(b1) to node[pos=.6, sloped]{\tiny$\blacktriangleright$} (a2);
			\draw (a2) to  (a3);
			\draw (b2) to  (b3);
			\draw (a4)  to [bend right=20] 
				node[pos=.5, sloped]{\tiny$\blacktriangleright$} (b3)
	 		(b4) to 	node[pos=.4, sloped]{\tiny$\blacktriangleright$} (b3);
			\node at  (112.5 - 45:1.5) {
			\begin{tikzpicture}[xscale = .3, yscale=-.3]
				\BOX{-2,-1}{-2} \BOX{-1,-1}{-1} \BOX{0,0}{1}\BOX{1,0}{2}
			\end{tikzpicture}};
			\node at  (112.5 - 2*45:1.5) {
			\begin{tikzpicture}[xscale = .3, yscale=-.3]  
				\BOX{-1,0}{-1} \BOX{-1,-1}{-2} \BOX{0,0}{2}\BOX{0,-1}{1} 
			\end{tikzpicture}};
			\node at  (112.5 - 3*45:1.5) {
			\begin{tikzpicture}[xscale = .3, yscale=-.3]  
				\BOX{-1,0}{1} \BOX{-1,-1}{-2} \BOX{0,0}{2}\BOX{0,-1}{-1}  
			\end{tikzpicture}};
			\node at  (112.5 - 4*45:1.6) {
			\begin{tikzpicture}[xscale = .3, yscale=-.3]
				\BOX{0,1}{2} \BOX{-1,-1}{-1} \BOX{0,0}{1}\BOX{-1,-2}{-2}   
			\end{tikzpicture}
			}; 
			\node at  (112.5 - 45:2) {\small$J = \emptyset$};
			\node at  (0:2) {\small$J = \{\vep_2 -\vep_1\}$};
			\node[] at  (112.5 - 4*45:2.2) {\small$J = \{\vep_2 \pm \vep_1\}$};
		\end{tikzpicture}};
\end{tikzpicture}}
\end{figure}
}

\subsection{Construction of the irreducible $H_2$-modules}\label{H2nonregirredconst}

The group $\cW_0$ acts on $(\CC^\times)^2$ as in \eqref{H2W0action} and
the \emph{central characters} are the $\cW_0$-orbits on $(\CC^\times)^2$.
The \emph{regular central characters} are the $\cW_0$-orbits of $\gamma = (\gamma_1, \gamma_2)\in (\CC^\times)^2$ 
that have $Z(\gamma) = \emptyset$, i.e.\ where the intertwining operators in \eqref{intertwinermaps} are defined.
Let $\CC[W]=\CC[W_1^{\pm1},W_2^{\pm1}] \subseteq H_2$.
By Kato's criterion (see \cite[Proposition 2.11b]{Ra2}),
for central characters $\gamma = (\gamma_1, \gamma_2)$ with $P(\gamma)=\emptyset$ there is 
a single irreducible module of dimension eight given by
$$L_{(\gamma_1,\gamma_2)} = \Ind_{\CC[W]}^H(\CC_{\gamma_1,\gamma_2}),
\qquad\hbox{where $\CC_{\gamma_1,\gamma_2} = \CC v$ with $W_1v = \gamma_1 v$ and $W_2v = \gamma_2 v$.}
$$ 
All irreducible modules with $Z(\gamma)=\emptyset$ 
are calibrated and can be constructed as in Theorem \ref{thm:calibconst}.

Representatives of the $\cW_0$-orbits of $\gamma = (\gamma_1, \gamma_2)\in (\CC^\times)^2$ that have 
$Z(\gamma)\ne \emptyset$ and $P(\gamma)\ne \emptyset$ are as follows:
\begin{equation}
\begin{array}{c|c|c}
\gamma=(\gamma_1,\gamma_2) &Z(\gamma) &P(\gamma) \\\hline
(t^{\frac12}, t^{\frac12}), (-t^{\frac12}, -t^{\frac12}) &\{\vep_2-\vep_1\} &\{\vep_2+\vep_1\} \\
(t_0^{\frac12}t_k^{\frac12},t_0^{\frac12}t_k^{\frac12}),
(-t_0^{-\frac12}t_k^{\frac12},-t_0^{-\frac12}t_k^{\frac12})
&\{\vep_2-\vep_1\} &\{\vep_1,\vep_2\} \\
(1, t), (-1, -t) &\{\vep_1\} &\{\vep_2-\vep_1, \vep_2+\vep_1\} \\
(\pm1, t_0^{\frac12}t_k^{\frac12}),
(\pm1, -t_0^{-\frac12}t_k^{\frac12}),
&\{\vep_1\} &\{\vep_2\}
\end{array}
\label{nonKatononreg}
\end{equation}
This classification is {valid under the genericity assumption on the parameters \eqref{eq:genericcond}}, 
which guarantees that none of these representatives are in the $\cW_0$-orbit of another.

The following analysis of modules of central character $\gamma=(\gamma_1,\gamma_2)$ in \eqref{nonKatononreg}
shows that no irreducible calibrated $H_2$-modules appear at these central characters.
As in \eqref{eq:r1andr2}, the values $r_1$ and $r_2$ are defined by
\begin{equation*}
-t^{r_1} = -t_k^{\frac12}t_0^{-\frac12} \quad\hbox{ and }\quad {-t}^{r_2}=t_k^{\frac12}t_0^{\frac12}.
\end{equation*}

\smallskip\noindent
\textbf{ Case $(\gamma_1, \gamma_2)=(-1,-t^{r_i})$ for $i=1$ or $2$:}  Let $H_{\{0\}}$ be the subalgebra of $H_2$ generated by 
$T_0, W_1^{\pm1}, W_2^{\pm1}$.  For each of $i=1$ and $i=2$, there are two irreducible modules 
of central character $\cc=(0,r_i)$:
$$L_{(0,r_i)}^{+} = \Ind_{H_{\{0\}}}^{H_2}(\CC_{(r_i,0)}),
\qquad\hbox{where $\CC_{(r_i,0)}=\CC v$ with}\quad
\begin{array}{l}
W_1 v = -t^{r_i}v, \\
W_2 v = -v, \\
T_0 v = t_0^{\frac12} v,
\end{array}
$$
and
$$L_{(0,r_i)}^{-} = \Ind_{H_{\{0\}}}^{H_2}(\CC_{(-r_i,0)}),
\qquad\hbox{where $\CC_{(-r_i,0)}=\CC v$ with}\quad
\begin{array}{l}
W_1 v = -t^{-r_i}v, \\
W_2 v = -v, \\
T_0 v = -t_0^{-\frac12} v.
\end{array}
$$
With $M=L_{(0,r_i)}^{+}$, the generalized weight space decomposition is
\begin{equation}
M = M_{(r_i, 0)}^{\mathrm{gen}} \oplus M_{(0,r_i)}^{\mathrm{gen}},
\qquad\hbox{with}\quad
\dim(M_{(r_i, 0)}^{\mathrm{gen}})=\dim(M_{(0,r_i)}^{\mathrm{gen}})=2.
\label{eq:weight-space0r2}
\end{equation}
The element $W_1W_2^{-1}$ acts on $M_{(r_i,0)}^{\mathrm{gen}}$ with eigenvalues $t^{r_i}$.
Since the parameters are generic (see \eqref{eq:genericcond}), $t^{r_i}\ne t^{\pm1}$ and thus, 
by \eqref{tauisq}, $\tau_1^2$ has no kernel.  Thus the intertwiner
$\tau_1\colon M_{(r_i, 0)}^{\mathrm{gen}}\to M_{(0,r_i)}^{\mathrm{gen}}$ is
invertible and $M=L_{(0,r_i)}^{+}$ is irreducible.
Replacing $r_i$ with $-r_i$ in \eqref{eq:weight-space0r2} yields the decomposition of $M=L_{(0,r_i)}^-$ analogously.

\smallskip\noindent
\textbf{ Case $(\gamma_1,\gamma_2)=(-t^{\frac12}, -t^{\frac12})$:}  Let $H_{\{1\}}$ be the subalgebra of $H_2$ generated by 
$T_1, W_1^{\pm1}, W_2^{\pm1}$.  There are two irreducible modules 
of central character $\cc=(\frac12,\frac12)$:
$$L_{(\frac12,\frac12)}^{+} = \Ind_{H_{\{1\}}}^{H_2}(\CC_{(-\frac12,\frac12)}),
\qquad\hbox{where $\CC_{(-\frac12,\frac12)}=\CC v$ with}\quad
\begin{array}{l}
W_1 v = -t^{-\frac12}v, \\
W_2 v = -t^{\frac12}v, \\
T_1 v = t^{\frac12} v,
\end{array}
$$
and
$$L_{(\frac12,\frac12)}^{-} = \Ind_{H_{\{1\}}}^{H_2}(\CC_{(\frac12,-\frac12)}),
\qquad\hbox{where $\CC_{(\frac12,-\frac12)}=\CC v$ with}\quad
\begin{array}{l}
W_1 v = -t^{\frac12}v, \\
W_2 v = -t^{-\frac12}v, \\
T_1 v = -t^{-\frac12} v.
\end{array}
$$
With $M=L_{(\frac12,\frac12)}^{+}$, the generalized weight space decomposition is
\begin{equation}
M = M_{(\frac12,\frac12)}^{\mathrm{gen}} \oplus M_{(-\frac12,\frac12)}^{\mathrm{gen}},
\qquad\hbox{with}\quad
\dim(M_{(\frac12,\frac12)}^{\mathrm{gen}})=\dim(M_{(-\frac12,\frac12)}^{\mathrm{gen}})=2.
\label{eq:weight-spacehalfhalf}
\end{equation}
The element $W_1^{-1}$ acts on $M_{(\frac12,\frac12)}^{\mathrm{gen}}$ with eigenvalues $-t^{\frac12}$.
Since the parameters are generic (see \eqref{eq:genericcond}), $-t^{\frac12}\not\in \{ -t^{\pm r_1}, -t^{\pm r_2} \}$ 
and thus, by \eqref{tau0sq}, $\tau_0^2$ has no kernel.  Thus the intertwiner
$\tau_0\colon M_{(\frac12,-\frac12)}^{\mathrm{gen}}\to M_{(-\frac12,-\frac12)}^{\mathrm{gen}}$ is invertible
and $M=L_{(\frac12,\frac12)}^{+}$ is irreducible.
Similarly, the structure of $M=L_{(\frac12,\frac12)}^-$ is given by 
swapping $\frac12$ and $-\frac12$ in \eqref{eq:weight-spacehalfhalf}.

\smallskip\noindent
\textbf{ Case $(\gamma_1,\gamma_2)=(-t^{r_i},-t^{r_i})$ for $i=1$ or $2$:}  Let $H_{\{0\}}$ be the subalgebra of $H_2$ generated by 
$T_0, W_1^{\pm1}, W_2^{\pm1}$.  For each of $i=1$ and $i=2$, there are two irreducible modules 
of central character $\cc=(r_i,r_i)$:
$$L_{(r_i,r_i)}^{+} = \Ind_{H_{\{0\}}}^{H_2}(\CC_{(r_i,-r_i)}),
	\qquad
	\hbox{where $\CC_{(r_i,-r_i)}=\CC v$ with}\quad
	\begin{array}{l}
		W_1 v = -t^{r_i}v, \\
		W_2 v = -t^{-r_i}v, \\
		T_0 v = t_0^{\frac12} v,
	\end{array}
$$
and
$$L_{(r_i,r_i)}^{-} = \Ind_{H_{\{0\}}}^{H_2}(\CC_{(-r_i,r_i)}),
	\qquad
	\hbox{where $\CC_{(-r_i,r_i)}=\CC v$ with}\quad
	\begin{array}{l}
		W_1 v = -t^{-r_i}v, \\
		W_2 v = -t^{r_i}v, \\
		T_0 v = -t_0^{-\frac12} v.
	\end{array}
$$
The irreducibility of $L_{(r_i,r_i)}^{+}$ and $L_{(r_i,r_i)}^{-}$ is not immediate. We will show that 
$M = L_{(r_i,r_i)}^{+}$ is irreducible; the irreducibility of  $L_{(r_i,r_i)}^{-}$ is proved analogously. 

The generalized weight space decomposition of $M = L_{(r_i,r_i)}^{+}$ is 
$$M = M_{(r_i,-r_i)}^{\mathrm{gen}}
\oplus
M_{(-r_i,r_i)}^{\mathrm{gen}}
\oplus M_{(r_i,r_i)}^{\mathrm{gen}}
\qquad\hbox{with}\quad
\begin{array}{l}
\dim(M_{(r_i,-r_i)}^{\mathrm{gen}})=\dim(M_{(-r_i,r_i)}^{\mathrm{gen}})=1, \\
\dim(M_{(r_i,r_i)}^{\mathrm{gen}})=2.
\end{array}$$
The element $W_1W_2^{-1}$ acts on $M_{(r_i,-r_i)}^{\mathrm{gen}}$ with eigenvalue $t^{r_i-(-r_i)}$.
Since the parameters are generic (see \eqref{eq:genericcond}), $t^{2r_i}\ne t^{\pm1}$ and thus, 
by \eqref{tauisq}, $\tau_1^2$ has no kernel.  Thus the intertwiner
$\tau_1\colon M_{(r_i,-r_i)}^{\mathrm{gen}}\to M_{(-r_i,r_i)}^{\mathrm{gen}}$
is invertible.
As a $H_{\{0\}}$-module, $M_{(r_i,r_i)}^{\mathrm{gen}}$ is irreducible (2-dimensional). 
So either $N=M_{(r_i,r_i)}^{\mathrm{gen}}$ is an $H_2$-submodule or $M$ is irreducible.  

For the purpose of deriving a contradiction, assume that $N=M_{(r_i,r_i)}^{\mathrm{gen}}$ is an $H_2$-submodule of $M$.
The space $N$ has a basis
$$\{n_\gamma, T_1n_\gamma\}
\qquad\hbox{with}\qquad
W_1 n_\gamma = -t^{r_i}n_\gamma, \quad\hbox{and}\quad W_2n_\gamma = -t^{r_i} n_\gamma.$$
By \eqref{eq:TiW},
$W_1^{-1}T_1n_\gamma = T_1W_2^{-1}n_\gamma+(t^{\frac12}-t^{-\frac12})W_1^{-1}n_\gamma 
= T_1(-t^{-r_i})n_\gamma+(t^{\frac12}-t^{-\frac12})(-t^{-r_i})n_\gamma$ and
the action of $W_1^{-1}$ and 
$W_1^{-2}$ on the basis $\{n_\gamma, T_1n_\gamma\}$ are given by the matrices
$$
\rho(W_1^{-1}) = (-t^{-r_i})\begin{pmatrix}
1 &(t^{\frac12}-t^{-\frac12}) \\
0 &1
\end{pmatrix}
\qquad\hbox{and}\qquad
\rho(W_1^{-2}) = \rho(W_1^{-1})^2 = t^{-2r_i}\begin{pmatrix}
1 &2(t^{\frac12}-t^{-\frac12}) \\
0 &1
\end{pmatrix}.
$$
Thus
\begin{align*}
\rho(1-W_1^{-2}) 
&= (1-t^{-2r_i})\begin{pmatrix} 1 &\frac{-2(t^{\frac12}-t^{-\frac12})t^{-2r_i}}{1-t^{-2r_i}} \\
0 & 1
\end{pmatrix}
\qquad\hbox{and} \\
\rho(1-W_1^{-2})^{-1} 
&= \frac{1}{(1-t^{-2r_i})}\begin{pmatrix} 1 &\frac{ 2(t^{\frac12}-t^{-\frac12})t^{-2r_i} }{1-t^{-2r_i} } \\
0 &1
\end{pmatrix}.
\end{align*}
Since $N$ is a submodule of $M$, we have
$\displaystyle{ 0=\tau_0 = T_0 - \frac{(t_0^{1/2}-t_0^{-1/2})+(t_k^{1/2}-t_k^{-1/2})W_1^{-1}}{1-W_1^{-2}} }$
(see \eqref{intertwinerdefs} for the formula for $\tau_0$),
and so
\begin{align*}
\rho(T_0) 
	&=((t_0^{1/2}-t_0^{-1/2})+(t_k^{1/2}-t_k^{-1/2})W_1^{-1})
		(1-W_1^{-2})^{-1}\\
&=\frac{(t_0^{1/2}-t_0^{-1/2}) + (t_k^{1/2}-t_k^{-1/2})(-t^{-r_i}) }{1-t^{-2r_i}} 
	\begin{pmatrix}
		1 
			&\frac{(t_k^{1/2}-t_k^{-1/2})(t^{\frac12}-t^{-\frac12})(-t^{-r_i})}
					{(t_0^{1/2}-t_0^{-1/2}) + (t_k^{1/2}-t_k^{-1/2})(-t^{-r_i}) } \\
		0 &1
	\end{pmatrix}
 	\begin{pmatrix} 
		1&\frac{2(t^{\frac12}-t^{-\frac12})t^{-2r_i}}{1-t^{-2r_i}} \\
		0 &1
	\end{pmatrix}\\
&=\frac{(t_0^{1/2}-t_0^{-1/2}) + (t_k^{1/2}-t_k^{-1/2})(-t^{-r_i}) }{1-t^{-2r_i}} 
	\begin{pmatrix}
		1 &
			\frac{t^{\frac12}-t^{-\frac12}}{-t^{r_i}}\left(\frac{2(-t^{-r_i})}{1-t^{-2r_i}}
			+ \frac{(t_k^{1/2}-t_k^{-1/2})}
					{(t_0^{1/2}-t_0^{-1/2}) + (t_k^{1/2}-t_k^{-1/2})(-t^{-r_i}})\right)\\
0 &1
\end{pmatrix}.
\end{align*}
Recall, from \eqref{eq:r1andr2}, that $-t^{r_i} = \pm t_k^{\pm\frac12} t_0^{\frac12}$, so that
\begin{align*}
\frac{(t_0^{1/2}-t_0^{-1/2}) + (t_k^{1/2}-t_k^{-1/2})(-t^{-r_i}) }{1-t^{-2r_i}}
&=\frac{(t_0^{1/2}-t_0^{-1/2}) + (t_k^{1/2}-t_k^{-1/2})(\pm t_k^{\mp\frac12}t_0^{-\frac12}) }{1-t_k^{\mp1}t_0^{-1}} 
= t_0^{\frac12}.
\end{align*}
The eigenvalues of $\rho(T_0)$ are $t_0^{\frac12}$
and, since $(T_0-t_0^{\frac12})(T_0+t_0^{-\frac12})=0$,
the Jordan blocks of $\rho(T_0)$ are of size 1,  forcing
 \begin{align*}
0
&= \frac{2(-t^{-r_i})}{1-t^{-2r_i}}
			+ \frac{(t_k^{1/2}-t_k^{-1/2})}
					{(t_0^{1/2}-t_0^{-1/2})+(t_k^{1/2}-t_k^{-1/2})(-t^{-r_i})} 
= \frac{2(-t^{-r_i})}{1-t^{-2r_i}}
			+ \frac{(t_k^{1/2}-t_k^{-1/2})}
					{(1-t^{-2r_i})t_0^{1/2}} \\
&= \frac{2(-t^{-r_i}) t_0^{\frac12} + (t_k^{1/2}-t_k^{-1/2})}{(1-t^{-2r_i})t_0^{1/2}} 
= \frac{2(\pm t_k^{\mp\frac12}t_0^{-\frac12}) t_0^{\frac12} + (t_k^{1/2}-t_k^{-1/2})}{(1-t^{-2r_i})t_0^{1/2}}
= \frac{\pm(t_k^{1/2}+t_k^{-1/2})}{(1-t^{-2r_i})t_0^{1/2}}.
\end{align*}
This is a contradiction since, by the generic condition on parameters in \eqref{eq:genericcond},  
$1\ne (-t^{r_1})(-t^{r_2}) = (-t_k^{\frac12}t_0^{-\frac12})(t_k^{\frac12}t_0^{\frac12}) = -(t_k^{\frac12})^2$.
Thus $N$ is not a submodule of $M$, and so $M$ is irreducible.

\smallskip\noindent
\textbf{ Case $(\gamma_1, \gamma_2)=(-1,-t)$:}  Let $H_{\{1\}}$ be the subalgebra of $H_2$ generated by 
$T_1, W_1^{\pm1}, W_2^{\pm1}$.  There are two irreducible modules 
of central character $\cc=(0,1)$:
$$L_{(0,1)}^{+} = \Ind_{H_{\{1\}}}^{H_2}(\CC_{(-1,0)}),
	\qquad
	\hbox{where $\CC_{(-1,0)}=\CC v$ with}\quad
	\begin{array}{l}
		W_1 v = -t^{-1}v, \\
		W_2 v = -v, \\
		T_1 v = t^{\frac12} v,
	\end{array}
$$
and
$$L_{(0,1)}^{-} = \Ind_{H_{\{1\}}}^{H_2}(\CC_{(1,0)}),
	\qquad
	\hbox{where $\CC_{(1,0)}=\CC v$ with}\quad
	\begin{array}{l}
		W_1 v = -tv, \\
		W_2 v = -v, \\
		T_1 v = -t^{-\frac12} v.
	\end{array}
$$
The irreducibility of $L_{(0,1)}^{+}$ and $L_{(0,1)}^-$ is not immediate. We will show that 
$M = L_{(0,1)}^{+}$ is irreducible; the irreducibility of  $L_{(0,1)}^{-}$ is proved analogously. 

The generalized weight space decomposition of $M = L_{(0,1)}^{+}$ is 
$$M = M_{(-1,0)}^{\mathrm{gen}}
	\oplus
	M_{(1,0)}^{\mathrm{gen}}
	\oplus M_{(0,1)}^{\mathrm{gen}}
		\qquad\hbox{with}\quad
	\begin{array}{l}
	\dim(M_{(-1,0)}^{\mathrm{gen}})=\dim(M_{(1,0)}^{\mathrm{gen}})=1,\\
	\dim(M_{(0,1)}^{\mathrm{gen}})=2. \end{array}
$$
The element $W_1^{-1}$ acts on $M_{(-1,0)}^{\mathrm{gen}}$ with eigenvalue $-t$.
Since the parameters are generic (see \eqref{eq:genericcond}), $-t\not\in \{ -t^{\pm r_1}, -t^{\pm r_2}\}$ and thus, 
by \eqref{tau0sq}, $\tau_0^2$ has no kernel.  Thus the intertwiner
$\tau_0\colon M_{(-1,0)}^{\mathrm{gen}}\to M_{(1,0)}^{\mathrm{gen}}$
is invertible.
Since 
$M_{(0,1)}^{\mathrm{gen}}$ is irreducible as a $H_{\{0\}}$-module, 
we have either $N=M_{(0,1)}^{\mathrm{gen}}$ is an $H_2$-submodule or $M$ is irreducible.

For the purpose of deriving a contradiction, assume that $N=M_{(0,1)}^{\mathrm{gen}}$ is an $H_2$-submodule of $M$.
The space $N$ has a basis
$$\{n_\gamma, T_0n_\gamma\}
	\qquad\hbox{with}\qquad
	W_1 n_\gamma = -n_\gamma, \quad\hbox{and}\quad W_2n_\gamma = -t n_\gamma.$$
By \eqref{T0pastW1} and \eqref{T0WcommuteB3},
\begin{align*}
	W_1W_2^{-1}T_0n_\gamma
	&= T_0W_1^{-1}W_2^{-1}n_\gamma
		+ ((t_0^{1/2}-t_0^{-1/2})+(t_k^{1/2}-t_k^{-1/2})W_1^{-1}) 
			\frac{W_1-W_1^{-1}}{1-W_1^{-2}} W_2^{-1} n_\gamma \\
	&= T_0 t^{-1} n_\gamma+ ((t_0^{1/2}-t_0^{-1/2})+(t_k^{1/2}-t_k^{-1/2})(-1))t^{-1}n_\gamma,
\end{align*}
and the action of $W_1W_2^{-1}$ on the basis $\{n_\gamma, T_0n_\gamma\}$ is given by the matrix
$$
\rho(W_1W_2^{-1}) = 
	\begin{pmatrix}
		t^{-1} &((t_0^{1/2}-t_0^{-1/2}) + (t_k^{1/2}-t_k^{-1/2})(-1)) t^{-1} \\
		0 &t^{-1}
	\end{pmatrix}.
$$
Thus
\begin{align*}
	\rho(1-W_1W_2^{-1}) 
		&= \begin{pmatrix} 
			1-t^{-1} &-((t_0^{1/2}-t_0^{-1/2}) + (t_k^{1/2}-t_k^{-1/2})(-1)) t^{-1} \\
			0 & 1-t^{-1}
			\end{pmatrix}\\
		&=(1-t^{-1})
			\begin{pmatrix} 
			1 &\displaystyle{-\frac{((t_0^{1/2}-t_0^{-1/2}) + (t_k^{1/2}-t_k^{-1/2})(-1)) t^{-1}}{1-t^{-1}} } \\
			0 & 1
			\end{pmatrix},
\end{align*}
and
\begin{align*}
	\rho(1-W_1W_2^{-1})^{-1} 
		&= \frac{1}{(1-t^{-1})}
			\begin{pmatrix} 
			1 &\displaystyle{ \frac{((t_0^{1/2}-t_0^{-1/2}) + (t_k^{1/2}-t_k^{-1/2})(-1)) t^{-1}}{1-t^{-1}} } \\
			0 &1
		\end{pmatrix}.
\end{align*}
If $N$ is a submodule of $M$ then $0=\tau_1 = T_1 - \frac{t^{\frac12} + t^{-\frac12}}{1-W_1W_2^{-1}}$
(see \eqref{intertwinerdefs} for the formula for $\tau_1$).
Thus
\begin{align*}
	\rho(T_1) 
		&=t^{\frac12}
			\begin{pmatrix} 
			1 &\displaystyle{ \frac{((t_0^{1/2}-t_0^{-1/2})+(t_k^{1/2}-t_k^{-1/2})(-1)) t^{-1}}{1-t^{-1}} } \\
			0 &1
			\end{pmatrix}.
\end{align*}
Since $(T_1-t^{\frac12})(T_1+t^{-\frac12})=0$ the Jordan blocks of $\rho(T_1)$
are of size one, forcing
$$0= (t_0^{1/2}-t_0^{-1/2})-(t_k^{1/2}-t_k^{-1/2}) = t_0^{-\frac12}(t_0^{1/2}+t_k^{-1/2})(t_0^{1/2}-t_k^{1/2}).$$
This is a quadratic equation in $t_0^{\frac12}$ with two solutions, 
$t_0^{\frac12} = t_k^{\frac12}$ and $t_0^{\frac12} = -t_k^{-\frac12}$.
This is a contradiction since, 
by the generic condition on parameters in \eqref{eq:genericcond}, $-t^{-r_1} = -t_0^{\frac12}t_k^{-\frac12} \ne -1$ 
and $-t^{r_2} = t_0^{\frac12}t_k^{\frac12}\ne -1$.  Thus
$N$ is not a submodule of $M$, and so $M$ is irreducible.

\section{Representations of $\cB_k^\ext$ in tensor space}
\label{sec:braids-on-tensor-space}

In this section we give a Schur-Weyl duality approach to the representations of the two boundary
Hecke algebras $H_k^{\mathrm{ext}}$.
More generally, in Theorem \ref{thm:BraidRep} we show that, for a quantum group or quasitriangular Hopf algebra
$U_q\fg$ and three $U_q\fg$-modules $M$, $N$ and $V$,
there is an action of the two boundary braid group $\cB_k^{\mathrm{ext}}$ on tensor space
$M\otimes N\otimes V^{\otimes k}$ that commutes with the $U_q\fg$-action.  This means that
there is a weak Schur-Weyl duality pairing between $U_q\fg$-modules and $\cB_k^{\mathrm{ext}}$-modules,
so that if $M\otimes N\otimes V^{\otimes k}$ is completely reducible as a $U_q\fg$-module then
$$M\otimes N\otimes V^{\otimes k} \cong \bigoplus_{\lambda} L(\lambda)\otimes B_k^\lambda
\qquad\hbox{as $(U_q\fg, \cB_k^{\mathrm{ext}})$-modules,}$$
where $L(\lambda)$ are irreducible $U_q\fg$-modules and $B_k^\lambda$ are $\cB_k^{\mathrm{ext}}$-modules.
In Section \ref{mainthm} we will explain that when $\fg=\fgl_n$ and $M$ and $N$ and $V$ are appropriately
chosen, the $\cB^{\mathrm{ext}}$-action provides an action of the two boundary Hecke algebra
$H_k^{\mathrm{ext}}$ (where the parameters depend on the choice of $M$ and $N$).  Our main theorem,
Theorem  \ref{thm:partitions-to-SLRs}, 
proves that the $H_k^{\mathrm{ext}}$-modules $B_k^\lambda$ that appear in tensor
space $M\otimes N\otimes V^{\otimes k}$ are irreducible, and identifies them in terms of the classification
of irreducible calibrated $H_k^{\mathrm{ext}}$-modules which is given in Theorem \ref{thm:calibconst}.

\subsection{Quantum groups and $R$-matrices}

Let $\fg$ be a finite-dimensional complex Lie algebra with a symmetric nondegenerate
$\ad$-invariant bilinear form, and let
$U_q\fg$ be the Drinfel'd-Jimbo quantum group corresponding to $\fg$. The quantum group $U_q\fg$ is a ribbon Hopf algebra with invertible $R$-matrix 
$$R=\sum_{R} R_1\otimes R_2
\qquad\hbox{in}\quad U_q\fg\otimes U_q\fg, \qquad \text{ and ribbon element } v = q^{-2\rho}u,$$ 
where $u = \sum_{R} S(R_2)R_1$ and $\rho$ is the staircase weight (see \cite[Corollary (2.15)]{LR}).
For $U_q\fg$-modules $M$ and $N$, the map
\begin{equation}\label{preRMNdefn}
\begin{tikzpicture}[xscale=.8, yscale=.6]
	\node at (-6,1.4) 
		{$\begin{matrix}
		\check R_{MN}\colon &N\otimes M &\longrightarrow &M\otimes N\\
		&n\otimes m &\longmapsto &\displaystyle{
		\sum_{R} R_2 m\otimes R_1 n }
		\end{matrix}
		$};
	\Under[1,1][2,2] 	\Over[1,2][2,1]
	\Nodes[1][2]	\Nodes[2][2]
	\node at (1.5,2.5) {$M \otimes N$};
	\node at (1.5,.5) {$N \otimes M$};
\end{tikzpicture}
\end{equation}
is a $U_q\fg$-module isomorphism.  The quasitriangularity of a ribbon Hopf algebra provides the  relations (see, for example, \cite[(2.9), (2.10), and (2.12)]{OR}),
$$\begin{array}{rl}
\begin{tikzpicture}[xscale=.8, yscale=.6]
	\Under[1,1][2,2] 	\Over[1,2][2,1]  \draw (1,2)--(1,4) (2,2) -- (2,4);
	\node[shape=rectangle,draw, fill=white] at (1,3) {$\varphi$};
	\Nodes[1][2]	\Nodes[4][2]
	\node at (1.5,4.5) {$M \otimes N$};
	\node at (1.5,.5) {$N \otimes M $};
\end{tikzpicture}
~~&
\begin{tikzpicture}[xscale=.8, yscale=.6]
	\node at (0,2.5) {$=$};
	\Under[1,3][2,4] 	\Over[1,4][2,3]  \draw (1,1)--(1,3) (2,1) -- (2,3);
	\node[shape=rectangle,draw, fill=white] at (2,2) {$\varphi$};
	\Nodes[1][2]	\Nodes[4][2]
	\node at (1.5,4.5) {$M \otimes N$};
	\node at (1.5,.5) {$N \otimes M $};
\end{tikzpicture} \\
( \varphi \otimes \id_N)
\check R_{MN}
&=
\check R_{MN}( \id_N \otimes \varphi),
\end{array} \qquad \text{ for any isomorphism $\varphi: M \to M$,}$$

\begin{align*}
\begin{tikzpicture}[xscale=.8, yscale=.6]
	\Under[1,1][2,2] 	\Over[1,2][2,1]	\draw (3,1)--(3,2);
	\Under[2,2][3,3] 	\Over[2,3][3,2]	\draw (1,2)--(1,3);
	\Under[1,3][2,4] 	\Over[1,4][2,3]	\draw (3,3)--(3,4);
	\Nodes[1][3]	\Nodes[4][3]
	\node at (2,4.5) {$M \otimes N \otimes V$};
	\node at (2,.5) {$V \otimes N \otimes M $};
\end{tikzpicture}
~~&
\begin{tikzpicture}[xscale=.8, yscale=.6]
\node at (0,2.5) {$=$};
	\Under[2,1][3,2] 	\Over[2,2][3,1]	\draw (1,1)--(1,2);
	\Under[1,2][2,3] 	\Over[1,3][2,2]	\draw (3,2)--(3,3);
	\Under[2,3][3,4] 	\Over[2,4][3,3]	\draw (1,3)--(1,4); 
	\Nodes[1][3]	\Nodes[4][3]
	\node at (2,4.5) {$M \otimes N \otimes V$};
	\node at (2,.5) {$V \otimes N \otimes M $};
\end{tikzpicture}
\\
(\check R_{MN}\otimes \id_V)
(\id_N\otimes \check R_{MV})
(\check R_{NV}\otimes \id_M)
&=
(\id_M\otimes \check R_{NV})
(\check R_{MV}\otimes \id_N)
(\id_V\otimes \check R_{MN}),
\end{align*}
\begin{align}
\begin{tikzpicture}[xscale=.8, yscale=.6]
	\Under[1,2][2,1]		\Over[1,1][2,2] 
	\Nodes[1][2]	\Nodes[2][2]
	\node at (2,2.5) {$M \otimes (N \otimes V)$};
	\node at (1,.5) {$(N \otimes V) \otimes M$};
\end{tikzpicture}
\begin{tikzpicture}[xscale=.8, yscale=.6]
\node at (0,2.5) {$=$};
	\Under[2,1][3,2] 	\Over[3,1][2,2]	\draw (1,1)--(1,2);
	\Under[1,2][2,3] 	\Over[2,2][1,3]	\draw (3,2)--(3,3); 
	\Nodes[1][3]	\Nodes[3][3]
	\node at (2,3.5) {$M \otimes N \otimes V$};
	\node at (2,.5) {$N \otimes V \otimes M $};
\end{tikzpicture}
\quad&\quad
\begin{tikzpicture}[xscale=.8, yscale=.6]
	\Under[1,1][2,2] 	\Over[1,2][2,1]	
	\Nodes[1][2]	\Nodes[2][2]
	\node at (1,2.5) {$(M \otimes N) \otimes V$};
	\node at (2,.5) {$V \otimes (M \otimes N) $};
\end{tikzpicture}
\begin{tikzpicture}[xscale=.8, yscale=.6]
\node at (0,2.5) {$=$};
	\Under[1,1][2,2] 	\Over[2,1][1,2]	\draw (3,1)--(3,2);
	\Under[2,2][3,3] 	\Over[3,2][2,3]	\draw (1,2)--(1,3); 
	\Nodes[1][3]	\Nodes[3][3]
	\node at (2,3.5) {$M \otimes N \otimes V$};
	\node at (2,.5) {$V \otimes M \otimes N $};
\end{tikzpicture}
\nonumber\\
(\check R_{M\otimes N, V})
=(\id_M\otimes \check R_{NV})
(\check R_{MV}\otimes \id_N)
\quad&\quad
(\check R_{M\otimes N, V})
=(\id_M\otimes \check R_{NV})
(\check R_{MV}\otimes \id_N).\label{Rcabling}
\end{align}

For a $U_q\fg$-module $M$
define
\begin{equation}\label{qcasimir}
\begin{matrix}
C_M\colon &M &\longrightarrow &M \\
&m &\longmapsto &vm 
\end{matrix}
\qquad\hbox{so that}\qquad
C_{M\otimes N} =
(\check R_{MN}\check R_{NM})^{-1}
(C_M\otimes C_N)
\end{equation}
(see \cite[Prop. 3.2]{Dr}).
Let $L(\lambda)$ denote the simple $U_q\fg$-module generated by a highest weight vector 
$v^+_\lambda$ of weight $\lambda$. Then 
\begin{equation}\label{qCasvalue}
C_{L(\lambda)} = q^{-\< \lambda,\lambda+2\rho\>} \id_{L(\lambda)}
\end{equation}
(see \cite[Prop. 2.14]{LR} or \cite[Prop. 5.1]{Dr}).
From \eqref{qCasvalue} and \eqref{qcasimir}, it follows that if $M=L(\mu)$ and 
$N=L(\nu)$ are finite-dimensional irreducible $U_q\fg$-modules of highest weights
$\mu$ and $\nu$ respectively, 
then $\check R_{MN}\check R_{NM}$ acts on the
$L(\lambda)$-isotypic component 
$L(\lambda)^{\oplus c_{\mu\nu}^\lambda}$
of the decomposition
\begin{equation}\label{eq:fulltwist} 
L(\mu)\otimes L(\nu) = \bigoplus_\lambda L(\lambda)^{\oplus c_{\mu\nu}^\lambda}
\qquad\hbox{by the scalar}\qquad
q^{\<\lambda,\lambda+2\rho\> 
-\<\mu,\mu+2\rho\> 
-\<\nu,\nu+2\rho\>}.
\end{equation}

\begin{prop} \label{thm:BraidRep}  Let $\fg$ be a finite-dimensional complex Lie algebra  with a symmetric nondegenerate $\ad$-invariant bilinear form, 
let $U_q\fg$ be the corresponding Drinfeld-Jimbo quantum group, and let
$\cZ=Z(U_\fg)$ be the center of $U$.  Let $M$, $N$, and $V$ be $U_\fg$-modules.
Then $M\otimes N \otimes V^{\otimes k}$ is a $\cZ \cB^\ext_k$-module with
action given by
\begin{equation}\label{eq:braidaction}
\begin{matrix}
\Phi\colon &\cZ \cB^\ext_k &\longrightarrow
&\End_{U_q\fg}(M\otimes N\otimes V^{\otimes k}) \\
&T_i &\longmapsto &  \check{R}_i, &\qquad &\hbox{for $i=1,\ldots, k-1$,}\\
&X_1 &\longmapsto &\check{R}_M^2, \\
&Y_1 &\longmapsto &  \check{R}_N^2, \\
&Z_1 &\longmapsto &  \check{R}_0^2, \\
&P &\longmapsto &  (\check R_{MN}\check R_{NM})\otimes \id_V^{\otimes (k)},
\end{matrix}
\end{equation}
where 
$$\check R_0^2 = (\check R_{(M\otimes N)V}\check R_{V(M\otimes N)})\otimes \id_V^{\otimes (k-1)},
\qquad
\check R_i = \id_M\otimes \id_V^{\otimes (i-1)}
\otimes \check R_{VV}\otimes \id_V^{\otimes (k-i-1)}$$
for $i=1,\ldots, k-1$,
$$\check R_M^2 = ((\id_M \otimes \check R_{NV})((\check R_{MV} \check R_{VM})\otimes \id_N)(\id_M \otimes \check R_{NV}^{-1}))\otimes \id_V^{\otimes k-1}, \quad\hbox{and} $$
$$\check R_N^2 = \id_M \otimes (\check R_{NV}\check R_{VN})\otimes \id_V^{\otimes (k-1)},
$$
with $\check R_{MV}$ as in \eqref{preRMNdefn}.
Moreover, this $\cZ\cB^\ext_k$ action commutes with the $U_q\fg$-action on $M\otimes N\otimes V^{\otimes k}$.
\end{prop}

\begin{proof}
This proof follows the proof of \cite[Prop.\ 3.1]{OR}, checking that the images of the generators $T_i$, $X_1$, $Y_1$, and $Z_1$ 
under the map $\Phi$ satisfy the
relations of presentation (a) of the two boundary braid group in Theorem \ref{thm:BraidPres}, as well as relations \eqref{Pcomm1} and \eqref{Pcomm2} for the extended two boundary braid group.
%
For $i\in \{1, \ldots, k-2\}$,
$$
{\def\XUNIT{.4} 	\def\YUNIT{.4}
\Phi(T_i)\Phi(T_{i+1})\Phi(T_i) = \check R_i\check R_{i+1}\check R_i =
\begin{matrix}\begin{tikzpicture}[xscale=\XUNIT, yscale=\XUNIT]
	\Under[1,1][2,2] 	\Over[1,2][2,1]	\draw (3,1)--(3,2);
	\Under[2,2][3,3] 	\Over[2,3][3,2]	\draw (1,2)--(1,3);
	\Under[1,3][2,4] 	\Over[1,4][2,3]	\draw (3,3)--(3,4);
	\Nodes[1][3]	\Nodes[4][3]
\end{tikzpicture}\end{matrix}
=
\begin{matrix}
\begin{tikzpicture}[xscale=\XUNIT, yscale=\XUNIT]
	\Under[2,1][3,2] 	\Over[2,2][3,1]	\draw (1,1)--(1,2);
	\Under[1,2][2,3] 	\Over[1,3][2,2]	\draw (3,2)--(3,3);
	\Under[2,3][3,4] 	\Over[2,4][3,3]	\draw (1,3)--(1,4); 
	\Nodes[1][3]	\Nodes[4][3]
\end{tikzpicture}\end{matrix}
= \check R_{i+1}\check R_i \check R_{i+1}=\Phi(T_{i+1})\Phi(T_i)\Phi(T_{i+1}).
}$$
Using the notation $\check R_{M\otimes N}$ for the endomorphism $\check R_0$, we have that, for $L = M, N,$ or $M \otimes N$,
$$
{\def\XUNIT{.3} 	\def\YUNIT{.35}
\check R_L^2\check R_1\check R_L^2\check R_1=
\begin{matrix}\begin{tikzpicture}[xscale=\XUNIT, yscale=\YUNIT]
	\Under[1,1][2,2] 	\Over[1,2][2,1]  \draw (3,1)--(3,2);
	\pgftransformyshift{-3cm};
		\Under[1,1][2,2] 	\Over[1,2][2,1]	\draw (3,1)--(3,2);
		\Under[2,2][3,3] 	\Over[2,3][3,2]	\draw (1,2)--(1,3);
		\Under[1,3][2,4] 	\Over[1,4][2,3]	\draw (3,3)--(3,4);
	\pgftransformyshift{-3cm};
		\Under[2,2][3,3] 	\Over[2,3][3,2]	\draw (1,2)--(1,3);
		\Under[1,3][2,4] 	\Over[1,4][2,3]	\draw (3,3)--(3,4);
	\Nodes[2][3]	\Nodes[8][3]
\end{tikzpicture}\end{matrix}
=
\begin{matrix}\begin{tikzpicture}[xscale=\XUNIT, yscale=\YUNIT]
	\Under[1,1][2,2] 	\Over[1,2][2,1]  \draw (3,1)--(3,2);
	\pgftransformyshift{-3cm};
	\Under[2,1][3,2] 	\Over[2,2][3,1]	\draw (1,1)--(1,2);
	\Under[1,2][2,3] 	\Over[1,3][2,2]	\draw (3,2)--(3,3);
	\Under[2,3][3,4] 	\Over[2,4][3,3]	\draw (1,3)--(1,4); 
	\pgftransformyshift{-3cm};
	\Under[2,2][3,3] 	\Over[2,3][3,2]	\draw (1,2)--(1,3);
	\Under[1,3][2,4] 	\Over[1,4][2,3]	\draw (3,3)--(3,4);
\draw[gray] (.7,4)--(3.3,4) (.7,7)--(3.3,7); 
\Nodes[2][3]	\Nodes[8][3]
\end{tikzpicture}\end{matrix}
=
\begin{matrix}\begin{tikzpicture}[xscale=\XUNIT, yscale=\YUNIT]
	\Under[2,1][3,2] 	\Over[2,2][3,1]	\draw (1,1)--(1,2);
	\Under[1,2][2,3] 	\Over[1,3][2,2]	\draw (3,2)--(3,3);
	\Under[2,3][3,4] 	\Over[2,4][3,3]	\draw (1,3)--(1,4); 
	\pgftransformyshift{-3cm};
	\Under[1,1][2,2] 	\Over[1,2][2,1]	\draw (3,1)--(3,2);
	\Under[2,2][3,3] 	\Over[2,3][3,2]	\draw (1,2)--(1,3);
	\Under[1,3][2,4] 	\Over[1,4][2,3]	\draw (3,3)--(3,4);
	\draw[gray] (.7,4)--(3.3,4);
	\Nodes[1][3]	\Nodes[7][3]
\end{tikzpicture}\end{matrix}
=
\begin{matrix}\begin{tikzpicture}[xscale=\XUNIT, yscale=\YUNIT]
	\Under[1,2][2,3] 	\Over[1,3][2,2]	\draw (3,2)--(3,3);
	\Under[2,3][3,4] 	\Over[2,4][3,3]	\draw (1,3)--(1,4); 
	\pgftransformyshift{-2cm};
	\Under[1,1][2,2] 	\Over[1,2][2,1]	\draw (3,1)--(3,2);
	\Under[2,2][3,3] 	\Over[2,3][3,2]	\draw (1,2)--(1,3);
	\Under[1,3][2,4] 	\Over[1,4][2,3]	\draw (3,3)--(3,4);
	\pgftransformyshift{-1cm};
	\Under[1,1][2,2] 	\Over[1,2][2,1]	\draw (3,1)--(3,2);
	\draw[gray] (.7,2)--(3.3,2) (.7,5)--(3.3,5);
	\Nodes[1][3]	\Nodes[7][3]
\end{tikzpicture}\end{matrix}
= \check R_1\check R_L^2\check R_1\check R_L^2,
}$$
which establishes 
$$\Phi(A)\Phi(T_1)\Phi(A)\Phi(T_1)=\Phi(T_1)\Phi(A)\Phi(T_1)\Phi(A)
\qquad\hbox{for $A= X_1, Y_1$, and $Z_1$, respectively.}$$
The formula
$$\Phi(Z_1)=\check R_0^2 = \check R_{M}^2\check R_{N}^2=\Phi(X_1)\Phi(Y_1)$$
is a consequence of the third set of relations (cabling relations) in \eqref{Rcabling}.
Finally, the relations
$$\Phi(P)\Phi(Y_1)\Phi(P)=\Phi(Z_1^{-1})\Phi(Y_1)\Phi(Z_1)
\qquad\hbox{and}\qquad
\Phi(P)\Phi(X_1)\Phi(P)=\Phi(Z_1^{-1})\Phi(X_1)\Phi(Z_1)$$
follow from the first and second sets of relations for $\check R$-matrices in \eqref{Rcabling} by the same braid computation by
which the identities \eqref{PYP} were derived.  The remainder of the relations (commuting generators) follow directly from 
the definitions of $\Phi(T_i)$, $\Phi(X_1)$, $\Phi(Y_1)$, $\Phi(Z_1)$, and $\Phi(P)$.
\end{proof}

\subsection{The $\cB_k^{\mathrm{ext}}$-modules $B_k^\lambda$} 
Assume that $M$, $N$, and $V$ are finite-dimensional $U_q\fg$-modules and that $\omega$ is the highest weight of $V$ so that
$$V=L(\omega)
\qquad\hbox{is irreducible of highest weight $\omega$.}
$$
Let $\cP^{(j)}$ be an index set for the irreducible $U_q\fg$-modules that appear in $M\otimes N\otimes V^{\otimes j}$ 
and let $\cP^{(-1)}$ be an index set for the irreducible $U_q\fg$-modules in $M$.
The \emph{Bratteli diagram} for the sequence of $U_q\fg$-modules
\begin{equation}
M, \quad M\otimes N, \quad M\otimes N\otimes V,\quad
M\otimes N\otimes V\otimes V ,\quad \cdots
\label{Brattelidefn}
\end{equation}
is the graph with
\begin{enumerate}[\quad]
\item vertices on level $j$ labeled by $\mu\in \cP^{(j)}$, for $j\in \ZZ_{\ge -1}$,
\item $m_{\mu\lambda}$ edges $\mu\to \lambda$ for $\mu\in \cP^{(j)}$ and $\lambda\in \cP^{(j+1)}$, and
where $L(\mu)\otimes V \cong \bigoplus_{\lambda\in \cP^{(j+1)}} L(\lambda)^{\oplus m_{\mu\lambda}}$,
\item  each edge $\mu\to \lambda$ labeled with 
$\frac12(\langle \lambda, \lambda+2\rho \rangle - \langle \omega, \omega+2\rho \rangle - \langle \mu, \mu+2\rho \rangle)$.
\end{enumerate}
A specific example in the case where $\fg=\fgl_n$ is given in Figure \ref{fig:Brat_diagram_with_contents}.

If $M$ and $N$ are finite-dimensional then $M\otimes N \otimes V^{\otimes k}$ is completely decomposable as a $U_q\fg$-module. If 
$B_k^\lambda$ is the space of highest weight vectors of weight $\lambda$ in $M\otimes N\otimes V^{\otimes k}$, then
\begin{equation}
M\otimes N\otimes V^{\otimes k} \cong \bigoplus_{\lambda\in \cP^{(k)}} L(\lambda)\otimes B_k^\lambda,
\qquad\hbox{as $(U_q\fg, \cB_k^{\mathrm{ext}})$-bimodules.}
\label{Bklambdadefn}
\end{equation}
The $\cB_k^{\mathrm{ext}}$-modules $B_k^\lambda$ are not necessarily irreducible and not necessarily nonisomorphic, 
though they will be in the (mostly rare but very important) settings where 
$\Phi(\CC \cB_k^{\mathrm{ext}}) = \End_{U_q\fg}(M\otimes N\otimes V^{\otimes k})$.

Recall from \eqref{BraidMurphy} that
$$Z_i = T_{i-1}\cdots T_1Z_1T_1\cdots T_{i-1}
\qquad\hbox{for $i=1,\ldots, k$.}$$
The following proposition shows that, as operators on $B_k^\lambda$, the $Z_i$ are simultaneously diagonalizable and
have eigenvalues determined by the edges on the Bratteli diagram.  The proof follows the same
schematic that is used, for example, in the proof of \cite[Proposition 3.2]{OR}.

\begin{prop}\label{prop:calibrated} Assume $M$, $N$, and $V$ are finite-dimensional $U_q\fg$ modules
with $V$ irreducible.  For $\lambda \in \cP^{(k)}$, let
$B_k^\lambda$ be the $\cB^{\mathrm{ext}}_k$-module in \eqref{Bklambdadefn}
and let
$$\cT_k^\lambda 
= \{ \hbox{paths $S = (S^{(-1)}\stackrel{e_0}{\to} S^{(0)} \stackrel{e_1}{\to} \ldots \stackrel{e_k}{\to} S^{(k)}=\lambda)$
in the Bratteli diagram} \}.
$$
Then
$$B_k^\lambda\quad\hbox{has a basis}\qquad \{ v_S\ |\ S\in \cT_k\}$$
of simultaneous eigenvectors for the action of $P, Z_1, \ldots, Z_k$,
with
$$Pv_S = q^{2e_0}v_S
\qquad\hbox{and}\qquad
Z_i v_S= q^{2e_i}v_S,\quad\hbox{for $i=1,\ldots, k$,}
$$
so that the eigenvalues of $P$ and $Z_1, \ldots, Z_k$ on $v_S$ are determined by the labels on the edges of the path $S$.
\end{prop}
\begin{proof}  
The basis $\{ v_S\ |\ S\in \cT_k^\lambda\}$ is constructed inductively.
For the initial case, choose any basis $\hat B_{-1}$ of the highest weight vectors in $M$,
and let $\hat B_{-1}^\nu$ be the set of basis elements in $\hat B_{-1}$ of weight $\nu$.
For the inductive step, assume that $\hat B_{k-1}^\mu=\{ v_T \ |\ T\in \cT_{k-1}^\mu\}$  has been constructed so that
$$M\otimes N \otimes V^{\otimes(k-1)}
= \bigoplus_{\mu\in P^{(k-1)}} L(\mu)\otimes B_{k-1}^\mu
= \bigoplus_{\mu\in P^{(k-1)}} L(\mu)\otimes \left(\sum_{T \in \cT_{k-1}^\mu} \CC v_T\right),
$$
The set $\hat B_{k-1}^\mu = \{ v_T \ |\ T\in \cT_{k-1}^\mu\}$ is a basis of the vector space of highest weight vectors
of weight $\mu$ in $M\otimes N\otimes V^{\otimes(k-1)}$ that is indexed by the
paths $T=(T^{(-1)}\to \cdots \to T^{(k-1)}=\mu)$ of length $k$ in the Bratteli diagram that end at $\mu$.  In this form
$L(\mu)\otimes \CC v_T$ denotes the irreducible $U_q\fg$-submodule
of $M\otimes N\otimes V^{\otimes(k-1)}$ with highest weight vector $v_T$ of weight $\mu$.
 
Then, for each $T=(T^{(-1)}\to \cdots \to T^{(k-1)}=\mu)$ in $\cT_{k-1}^\mu$, choose a basis 
$$\hat B_k^{T\to \lambda} = \{v_S\ |\ \hbox{$S = (T^{(-1)}\to \cdots \to T^{(k-1)}=\mu \to \lambda)$}\}$$
of highest weight vectors in the submodule of $M\otimes N\otimes V^{\otimes k}$ given by
$$(L(\mu)\otimes \CC v_T)\otimes V = L(\mu)\otimes V \otimes \CC v_T
= \sum_{\mu\to \lambda} L(\lambda)\otimes \CC v_S.$$
The basis $\hat B_k^{T\to \lambda}$ is indexed by the edges in the Bratteli diagram from $\mu$ to a partition $\lambda$ on level $k$.
Then 
$$\hat B_k^\lambda = \bigsqcup_{\mu}\  \bigsqcup_{T\in \cT_{k-1}^\mu} \cT_k^{T\to\lambda}
\qquad\hbox{is a basis of $B_k^\lambda$.}$$
The central element $q^{-2\rho}u$
in $U_q\fg$ acts on the submodule $L(\mu)\otimes \CC v_T$ of
$M\otimes N\otimes V^{\otimes (k-1)}$
by the constant $q^{-\langle \mu,\mu+2\rho\rangle}$.
From \eqref{Rcabling}, \eqref{qcasimir}, and \eqref{qCasvalue} it follows that $Z_i$ acts on
$M\otimes N\otimes V^{\otimes k}$ by
\begin{align}
	\Phi(Z_i) = \check R_{i-1}\cdots \check R_1
				&\check R_0^2\check R_1\cdots \check R_{i-1}
	=\check R_{M\otimes N\otimes V^{\otimes (i-1)},V}
		\check R_{V,M\otimes N\otimes V^{\otimes (i-1)}}
		\otimes \id_V^{\otimes (k-i)} \nonumber \\
	&=(C_{M\otimes N\otimes V^{\otimes (i-1)}}\otimes C_V)
			C^{-1}_{M\otimes N\otimes V^{\otimes i}}
			\otimes \id_V^{\otimes (k-i)} \nonumber \\
	&= \sum_{\lambda,\mu,\nu}
			q^{\langle \lambda,\lambda+2\rho\rangle
			-\langle \mu,\mu+2\rho\rangle-\langle \omega,\omega+2\rho\rangle}
			\pi_{\mu\omega}^\lambda\otimes \id_V^{\otimes (k-i)},
\label{Zieigenval} 
\end{align}
where $\pi_{\mu\nu}^\lambda\colon M\otimes N\otimes \id_V^{\otimes i}\to
M\otimes N\otimes \id_V^{\otimes i}$ is the projection onto the $L(\lambda)$ isotypic component
of $(L(\mu)\otimes B_{i-1}^{\mu})\otimes V$.
Thus $Z_i$ acts diagonally on the basis $\hat B_k^\lambda$ and, by the definition of the
labels of edges in the Bratteli diagram in \eqref{Brattelidefn},
the eigenvalues of $Z_iv_S = q^{2e_i}v_S$ where $e_i$ is the label on the edge $S^{(i)}\to S^{(i+1)}$ in 
the Bratteli diagram.

\end{proof}

\subsection{Some tensor products for $\fg = \fgl_n$}

The finite-dimensional irreducible 
polynomial representations $L(\lambda)$ of $U_q\fgl_n$ are indexed by elements of 
 $$P_{poly}^+ = \left\{ \lambda = \lambda_1 \vep_1 + \cdots + \lambda_{n} \vep_{n}, 
	~|~  \lambda_i \in \ZZ, ~ \lambda_1 \geq \cdots \geq \lambda_{n}\ge 0\right\}.$$
Use
	 \begin{equation}\label{eq:delta}
	 \rho = (n-1)\vep_1 + (n-2)\vep_2 + \cdots +\vep_{n-1} = \sum_{i=1}^{n} (n-i)\vep_i,
	 \end{equation}
as in \cite[I\  (1.13)]{Mac1}. 
Identify each element $\lambda= \lambda_1 \vep_1 + \cdots + \lambda_{n} \vep_{n}$ in
$P_{poly}^+$ with the corresponding partition having $\lambda_i$ boxes in row $i$ so that,
for example,
$$\lambda = 3 \vep_1 + 2 \vep_2 + 2 \vep_3 =
	\begin{matrix}\begin{tikzpicture}[xscale=.3, yscale=-.3]
		\Part{3,2,2}
	\end{tikzpicture}\end{matrix}.
$$
The \emph{content} of the box in row $i$ and column $j$ of a partition $\lambda$ is
\begin{equation}\label{content}
c(\mathrm{box}) = j-i = (\hbox{diagonal number of box}),
\end{equation}
where the diagonals are numbered by the elements of $\ZZ$ from southwest to northeast, with the northwest corner box of a partition being in diagonal 0. 

The representation $L(\vep_1) = L(\square)$ is the standard $n$-dimensional representation of $U_q\fgl_n$.
When $\nu = \vep_1$, the decompsition in \eqref{eq:fulltwist} is given by
\begin{equation}
L(\mu)\otimes L(\square) \cong \bigoplus_{\lambda\in \mu^+} L(\lambda),
\label{tensorbyV}
\end{equation}
where $\mu^+$ is the set of partitions obtained by adding a box to $\mu$. If $\lambda \in \mu^+$ and $\lambda/\mu$ is the box added to $\mu$ to obtain $\lambda$,
then the action in \eqref{eq:fulltwist} is given by
\begin{align}
\langle \lambda,&\lambda+2\rho \rangle - \langle \mu,\mu+2\rho \rangle - \langle \vep_1,\vep_1+2\rho \rangle \nonumber \\
&=\langle \mu+\vep_i,\mu+\vep_i+2\rho \rangle - \langle \mu,\mu+2\rho \rangle - \langle \vep_1,\vep_1+2\rho \rangle 
=2\mu_i +1+ 2\rho_i  - 1-2\rho_1 \nonumber \\
& =2\mu_i + 2(n-i)-2(n-1) = 2\mu_i - 2i +2 
= 2c(\lambda/\mu) \label{eq:contentAdded}
\end{align}
(see \cite[I\ (5.16) and (8.4)]{Mac1}). 
Since $\< \vep_1, \vep_1+2\rho\> = 2(n-1)+1=2n-1$,
it follows by induction on the number of boxes in a partition $\lambda$ that
\begin{equation}
\< \lambda, \lambda+2\rho\> = (2n-1)|\lambda|+\sum_{\mathrm{box}\in \lambda} 2c(\mathrm{box}).
\label{contentCasval}
\end{equation}

For $\mu,\nu\in P_{poly}^+$, the decomposition of the tensor product 
$L(\mu)\otimes L(\nu)$ can be calculated using the Littlewood-Richardson rule (see \cite[Ch.\ I\ (9.2)]{Mac1}). 
When $\mu$ and $\nu$ are rectangles the decomposition is multiplicity free by the following theorem.
In equation \eqref{eq:P0drawings}, 
$\cA$ consists of the boxes that are in the union of the rectangles $(a^c)$ and $(b^d)$ 
(placed with northwest corner at $(1,1)$), and
the dashed rectangular regions are the $\min(a,b) \times d$ rectangle $\cB$
with northwest corner box at $(\max(a,b)+1,1)$, and
the $d\times \min(a,b)$ rectangle $\cB'$ with northwest corner at $(1,c+1)$.


\begin{prop} \label{prop:rectangles}
\emph{(See \cite[Lem. 3.3]{St}, \cite[Thm 2.4]{Ok})}
Let $a,b,c,d\in \ZZ_{\ge 0}$ such that $c \geq d$. 
For $\mu\subseteq (\min(a,b)^d)$ let
\begin{equation}\label{eq:P0drawings}
\begin{matrix}
\begin{tikzpicture}[xscale=.3, yscale=-.3]
\filldraw[black!10] (0,0)--(11,0)--(11,2)--(9,2)--(9,3)--(8,3)--(8,4)--(7,4)--(7,6)--(4,6)--(4,7)--(3,7)--(3,8)--(1,8)--(1,10)--(0,10)--(0,0);
\draw (0,0)--(7,0)--(7,6)--(0,6)--(0,0);
\draw[dashed] (7,0)--(12,0)--(12,4)--(7,4) (5,6)--(5,10)--(0,10)--(0,6);
\node at (3.5,3) {$\cA$};
\draw[->](9.5,5) node[below] {$\cB$} to  (9.5,3.5);
\draw[->] (6,8) node[right] {$\cB'$} to (4.5,8);
\node[above] at (3.5,0) {\small $a$};
\node[above] at (9.5,0) {\small$b$};
\node[below] at (2.5,10) {\small$b$};
\node[right] at (12,2) {\small$d$};
\node[left] at (0,3) {\small$c$};
\node[left] at (0,8) {\small$d$};
\node at (9,1) {\small$\mu$};
\node at (2,7) {\small$\mu^c$};
\draw (7,0) -- (11,0)--(11,2)--(9,2)--(9,3)--(8,3)--(8,4)--(7,4);
\draw (4,6)--(4,7)--(3,7)--(3,8)--(1,8)--(1,10)--(0,10)--(0,6);
\node[above] at (6, -2) {if $a \geq b$:};
\node at (-3,5) {$\mathring{\mu}=$};
\end{tikzpicture} 
\qquad 
\begin{tikzpicture}[xscale=.3, yscale=-.3]
\filldraw[black!10] (0,0)--(12,0)--(12,2)--(10,2)--(10,3)--(9,3)--(9,4)--(5,4)--(5,6)--(4,6)--(4,7)--(3,7)--(3,8)--(1,8)--(1,10)--(0,10)--(0,0);
\draw (0,0)--(8,0)--(8,4)--(5,4)--(5,6)--(0,6)--(0,0);
\draw[dashed] (8,0)--(13,0)--(13,4)--(8,4) (5,6)--(5,10)--(0,10)--(0,6);
\node at (4,2) {$\cA$};
\draw[->](10.5,5) node[below] {$\cB$} to  (10.5,3.5);
\draw[->] (6,8) node[right] {$\cB'$} to (4.5,8);
\node[above] at (4,0) {\small $b$};
\node[above] at (10,0) {\small$a$};
\node[below] at (2.5,10) {\small$a$};
\node[right] at (13,2) {\small$d$};
\node[left] at (0,3) {\small$c$};
\node[left] at (0,8) {\small$d$};
\node at (10,1) {\small$\mu$};
\node at (2,7) {\small$\mu^c$};
\draw (8,0) -- (12,0)--(12,2)--(10,2)--(10,3)--(9,3)--(9,4)--(8,4);
\draw (4,6)--(4,7)--(3,7)--(3,8)--(1,8)--(1,10)--(0,10)--(0,6);
\node[above] at (6, -2) {if $a \leq b$:};
\node at (-3,5) {$\mathring{\mu}=$};
\end{tikzpicture}
\end{matrix}
\end{equation}
so that $\mu^c$ is the $180^\circ$ rotation of the complement of $\mu$ in a $\min(a,b)\times d$ rectangle.
Denote the rectangular partition with $c$ rows of length $a$ by $(a^c)$. 
Then 
\begin{equation}
L((a^c)) \otimes L((b^d))
\cong \bigoplus_{\mu\subseteq (\min(a,b)^d) } L(\mathring{\mu})
\cong \bigoplus_{\nu\in \cP^{(0)} } L(\nu),
\label{rectprod}
\end{equation}
where $\cP^{(0)} = \{ \mathring{\mu}\ |\ \mu\subseteq ((\min(a,b)^d)\}$.
\end{prop}

\noindent
For an example of the decomposition in \eqref{rectprod}, see Figure \ref{fig:Brat_diagram_with_contents}, where
the decomposition of $L(a^c)\otimes L(2^2)$ for $a,c \geq 2$ is indicated in level $0$ of the Bratteli diagram 
(see the description following \eqref{BklambdaHeckedefn} for explanation of the Bratteli diagram).

The value in \eqref{eq:fulltwist} for the product in \eqref{rectprod}
is given by using \eqref{contentCasval} to compute
\begin{align}
\langle \mathring{\mu},\, &\mathring{\mu}+2\rho\rangle-\langle (a^c), (a^c)+2\rho\rangle - \langle (b^d), (b^d)+2\rho\rangle \nonumber \\
&=(2n-1)\big( \vert \mathring{\mu} \vert - \vert (a^c)\vert - \vert (b^d) \vert\big)
+\left( \sum_{\mathrm{box}\in \mathring{\mu}} 2c(\mathrm{box})  \right)
- \sum_{\mathrm{box}\in (a^c)} 2c(\mathrm{box})  
- \sum_{\mathrm{box}\in (b^d)} 2c(\mathrm{box})  \nonumber \\
&=0 + \sum_{\mathrm{box}\in \mathring{\mu}} 2c(\mathrm{box})  
-ac(a-c)-bd(b-d).
\label{eq:0edgelabel}
\end{align}

\subsection{Irreducible $H_k^\ext$-modules in $M\otimes N\otimes V^{\otimes k}$}\label{mainthm}

In this subsection we provide, for $\fg = \fgl_n$, specific highest weight modules $M$, $N$, and $V$ such that the 
$\cB_k^\ext$-action factors through the extended two boundary Hecke algebra $H_k^{\mathrm{ext}}$.  
In these cases the $\cB_k^\ext$-modules $B_k^\lambda$ in \eqref{Bklambdadefn} are calibrated $H_k^\ext$-modules.  
Theorem \ref{thm:partitions-to-SLRs} identifies
the $B_k^\lambda$ for these cases explicitly in terms of the indexings of calibrated $H_k^\ext$-modules
given in Theorem \ref{thm:calibconst} and Proposition \ref{prop:stdtabbijection}.

Recall that, as defined in Section \ref{Heckedefnsubsection},
the \emph{extended two boundary Hecke algebra} $H_k^\ext$ is the quotient of the group algebra of the extended two boundary 
braid group $\CC \cB_k^\ext$ by the relations 
\begin{equation}\label{HeckedefnREPEAT}
(X_1 - a_1)(X_1-a_2) = 0,
\qquad
(Y_1 - b_1)(Y_1-b_2) = 0,
\quad\hbox{and}\quad 
(T_i - t^{\half})(T_i+ t^{-\half}) = 0,
\end{equation}
$i = 1, \dots, k-1$, for fixed $a_1, a_2, b_1, b_2, t^{\frac12} \in \CC^\times$.

\begin{thm} \label{thm:HeckeActionOnTensorSpace}
If $\fg = \fgl_n$, $M = L((a^c))$, $N=L((b^d))$, and $V = L(\square)$, 
\begin{equation}
a_1=q^{2a}, \qquad
a_2 = q^{-2c}, \qquad
b_1=q^{2b}, \qquad
b_2=q^{-2d},\qquad \text{ and } \qquad t^{\frac12} = q,
\label{eq:rectparams}
\end{equation}
then the map $\Phi$ from Proposition \ref{thm:BraidRep} gives an action of  $H_k^\ext$ on $M \otimes N \otimes V^{\otimes k}$ commuting with that of $U_q\fgl_n$.
\end{thm}
\begin{proof}
The module $M \otimes V$ decomposes as 
\begin{equation}\label{MVdecomp}M \otimes V  = L
\begin{pmatrix}\begin{tikzpicture}[xscale=.15, yscale=-.15, every node/.style={inner sep=2pt}]
\draw(0,0)--(5,0)--(5,4)--(0,4)--(0,0) (5,0)--(6,0)--(6,1)--(5,1);
\node[above] at (2.5,0) {\scriptsize $a$};
\node[right] at (0,2) {\scriptsize $c$};
\end{tikzpicture}\end{pmatrix}
\oplus L
\begin{pmatrix}\begin{tikzpicture}[xscale=.15, yscale=-.15, every node/.style={inner sep=2pt}]
\draw(0,0)--(5,0)--(5,4)--(0,4)--(0,0) (0,4)--(0,5)--(1,5)--(1,4);
\node[above] at (2.5,0) {\scriptsize $a$};
\node[right] at (0,2) {\scriptsize $c$};
\end{tikzpicture}\end{pmatrix}.
\end{equation}By \eqref{eq:fulltwist} and \eqref{eq:contentAdded}, $\check R_{MV} \check R_{VM}$ acts on
the first summand by the constant $q^{2a}$ and on the second summand by the constant $q^{-2c}$.
So 
$$(\Phi(X_1) - q^{2a})(\Phi(X_1) - q^{-2c}) = 0;
\qquad\hbox{similarly}\qquad
(\Phi(Y_1) - q^{2b})(\Phi(Y_1) - q^{-2d}) = 0
$$
by replacing $(a^c)$ with $(b^d)$.  The relation 
$$(\Phi(T_i) - q)(\Phi(T_i)+q^{-1})=0$$
follows similarly by considering the tensor product $V\otimes V = L(\square)\otimes L(\square)$.
\end{proof}

From \eqref{eq:ab-to-t0tk}, \eqref{eq:rectparams}, and \eqref{eq:r1andr2},
\begin{equation}
\begin{array}{c}
	a_1=q^{2a}, \qquad a_2 = q^{-2c}, \qquad
	b_1=q^{2b}, \qquad b_2=q^{-2d},	\qquad  t^{\frac12} = q, \\
	t_k^{\frac12} =  a_1^{\frac12}(-a_2)^{-\frac12} 
				= -i q^{a+c}  \quad \text{ and }\quad 
	t_0^{\frac12} = b_1^{\frac12}(-b_2)^{-\frac12} 
				= -i q^{b+d}, \\
-t^{r_1} = -t_k^{\frac12}t_0^{-\frac12} 
=-q^{(a+c)-(b+d)},
\quad\hbox{ and } \qquad
{-t}^{r_2} = t_k^{\frac12}t_0^{\frac12} 
=-q^{a+c+b+d}.
\end{array}
\label{eq:ttoqconversions}
\end{equation}
Using these conversions, the genericity conditions in \eqref{eq:genericcond} become requirements that
$q$ is not a root of unity and
\begin{equation*}
-q^{(a+c)-(b+d)}, -q^{a+c+b+d}\not\in \{ 1, -1, q^{\pm1}, -q^{\pm1}, q^{\pm2}, -q^{\pm2}\}
\quad\hbox{and}\quad -q^{(a+c)-(b+d)}\ne -q^{\pm(a+c+b+d)}.
\end{equation*}
In the context of Theorem \ref{thm:HeckeActionOnTensorSpace}, these genericity conditions are  
\begin{equation}
\hbox{$q$ is not a root of unity,}\quad
\hbox{$a,b,c,d\in \ZZ_{>0}$}
\quad\hbox{and}\quad
(a+c)-(b+d)\not\in \{0,\pm1, \pm2\}.
\label{newgenconds}
\end{equation}

In the setting of Theorem \ref{thm:HeckeActionOnTensorSpace}, equation \eqref{Bklambdadefn} provides
$H_k^{\mathrm{ext}}$-modules $B_k^\lambda$ with
\begin{equation}
M\otimes N\otimes V^{\otimes k} \cong \bigoplus_{\lambda\in \cP^{(k)}} L(\lambda)\otimes B_k^\lambda,
\qquad\hbox{as $(U_q\fg, \cH_k^{\mathrm{ext}})$-bimodules.}
\label{BklambdaHeckedefn}
\end{equation}
Theorem \ref{thm:partitions-to-SLRs} below will accomplish our primary goal for this paper by identifying
the module $B_k^\lambda$ explicitly as a calibrated $H_k^{\ext}$-module $H_k^{(z,\cc,J)}$ as constructed in 
Theorem \ref{thm:calibconst}.
The results of \eqref{tensorbyV}, \eqref{eq:contentAdded}, and Proposition \ref{prop:rectangles}
show that the Bratteli diagram of \eqref{Brattelidefn} has
$\cP^{(-1)} = \{ (a^c)\}$,  $\cP^{(0)} = \{\ring{\mu} ~|~ \mu \subseteq  ((\mathrm{min}(a,b))^d)\}$ as in Proposition \ref{prop:rectangles} and, for $j\in \ZZ_{\ge 0}$,
$$\hbox{$\cP^{(j)} = \{$ partitions obtained by adding $j$ boxes to a partition in $\cP^{(0)}\ \}$.}$$
By \eqref{eq:0edgelabel}, if $\mathring{\mu}\in\cP^{(0)}$ then there is an edge
\begin{equation}
\hbox{$(a^c)\xrightarrow{e_0(\mathring{\mu})} \mathring{\mu}$ with label}\qquad
\displaystyle{ e_0(\mathring{\mu}) =  -\frac{ac}{2}(a-c)-\frac{bd}{2}(b-d)+\sum_{\mathrm{box}\in \mathring{\mu}} c(\mathrm{box}) }.
\label{c0edgelabel}
\end{equation}
For $j\ge0$, the edges $\mu\to \lambda$ from level $j$ to level $j+1$ correspond to adding a single box to $\mu$ to get $\lambda$, and are labeled by 
$c(\lambda/\mu)$, the content of the box $\lambda/\mu$:
\begin{equation}
\mu\xrightarrow{c(\lambda/\mu)}\lambda
\qquad\hbox{for edges from level $j$ to level $j+1$.}
\label{cjedgelabel}
\end{equation}
The case when $M = L(a^c)$ and $N = L(2^2)$ with $a,c>2$ is illustrated in Figure \ref{fig:Brat_diagram_with_contents}.

\newcommand{\BoxPart}[2]{
\draw (0,0)--(5,0)--(5,4)--(0,4)--(0,0);
\coordinate (up) at (5,0); \coordinate (down) at (0,4);
\begin{scope}[shift=(up)]\Part{#1}\end{scope}
\begin{scope}[shift=(down)]\Part{#2}\end{scope}
}
\newcommand{\BoxPART}[2]{
\begin{tikzpicture}[xscale=\PartUNIT, yscale=-\PartUNIT]
	\node at (8,0) {}; \node at (0,7){};
	\BoxPart{#1}{#2}
\end{tikzpicture}
}
\newcommand{\BoxPARTs}[2]{
\begin{tikzpicture}[xscale=\PartUNIT, yscale=-\PartUNIT]
	\node at (8,0) {}; \node at (0,6){};
	\BoxPart{#1}{#2}
\end{tikzpicture}
}
\tikzstyle{EdgeLabel}=[blue!70!black, inner sep=2pt, fill=white]

\begin{figure}
$$\begin{tikzpicture}[xscale=14*\PartUNIT, yscale=10*\PartUNIT]
\node[right] at (6.5,1) {\small level $-1$};
\node[right] at (6.5,-1) {\small level 0};
\draw[|-|] (6.5,-2.5)--(6.5,-4.5);
\node[right] at (6.5,-3.5) {\small level 1};
\coordinate (01) at (3.5,1);
\foreach \x in {1, ..., 6}{ \coordinate (1\x) at (\x,-1);}
\foreach \x in {1, ..., 6}{ \coordinate (2\x1) at (\x,-2.5);}
\foreach \x in {1, ..., 6}{ \coordinate (2\x2) at (\x,-3.5);}
\foreach \x in {1, ..., 6}{ \coordinate (2\x3) at (\x,-4.5);}
\draw (01) to [bend right=25]  node[sloped, EdgeLabel]{\scriptsize$4a$} (11);
\draw (01) to [bend right=10]  node[sloped, EdgeLabel]{\scriptsize$3a$-$c$} (12);
\draw (01) to [bend right=0]  node[sloped, EdgeLabel]{\scriptsize$2(a$-$c$+$1)$} (13);
\draw (01) to [bend left=0] node[sloped, EdgeLabel]{\scriptsize$2(a$-$c$-$1)$} (14);
\draw (01) to [bend left=10] node[sloped, EdgeLabel]{\scriptsize$a$ -$3c$} (15);
\draw (01) to [bend left=25] node[sloped, EdgeLabel]{\scriptsize-$4c$} (16);
\draw (11)--(211) node[midway,EdgeLabel] {\tiny -$c$}
	(12)--(211) node[pos=.6, EdgeLabel] {\tiny $a$}
	(12)--(221) node[midway, EdgeLabel] {\tiny -$c$+$1$}
	(12)--(231) node[pos=.35, EdgeLabel] {\tiny -$c$-$1$}
	(13)--(221) node[pos=.6, EdgeLabel] {\tiny $a$-$1$}
	(13)--(241) node[pos=.35, EdgeLabel] {\tiny -$c$-$1$}
	(14)--(231) node[pos=.6, EdgeLabel] {\tiny $a$+$1$}
	(14)--(251) node[pos=.35, EdgeLabel] {\tiny -$c$+$1$}
	(15)--(241) node[pos=.6, EdgeLabel] {\tiny $a$+$1$}
	(15)--(251) node[midway, EdgeLabel] {\tiny $a$-$1$}
	(15)--(261) node[pos=.4, EdgeLabel] {\tiny -$c$}
	(16)--(261) node[midway, EdgeLabel] {\tiny $a$};
\begin{scope}[every node/.style={fill=white, inner sep=1pt}]
	\node at (01) {\begin{tikzpicture}[xscale=\PartUNIT, yscale=-\PartUNIT]
		\draw (0,0)--(5,0)--(5,4)--(0,4)--(0,0);
		\node[right] at (0,2) {\tiny $c$};
		\node[above] at (2.5,0) {\tiny $a$};
		\end{tikzpicture}
};
	\node at (211) {\BoxPART{2,2}{1}};
	\node at (221) {\BoxPART{2,1}{2}};
	\node at (231) {\BoxPART{2,1}{1,1}};
	\node at (241) {\BoxPART{1,1}{2,1}};
	\node at (251) {\BoxPART{1,1}{2,1}};
	\node at (261) {\BoxPART{1}{2,2}};
\draw (11) .. controls (-.55+1,-2) and (-.55+1,-2.5) .. (212) node[pos=.8,EdgeLabel]{\tiny $a$+$2$};
\draw (12) .. controls (-.55+2,-2) and (-.55+2,-2.5) .. (222) node[pos=.8,EdgeLabel]{\tiny $a$+$2$};
\draw (13) .. controls (-.55+3,-2) and (-.55+3,-2.5) .. (232) node[pos=.8,EdgeLabel]{\tiny $a$+$2$};
\draw (14) .. controls (-.55+4,-2) and (-.55+4,-2.5) .. (242) node[pos=.8,EdgeLabel]{\tiny $a$-$2$};
\draw (15) .. controls (-.55+5,-2) and (-.55+5,-2.5) .. (252) node[pos=.8,EdgeLabel]{\tiny -$c$+$2$};
\draw (16) .. controls (-.55+6,-2) and (-.55+6,-2.5) .. (262) node[pos=.8,EdgeLabel]{\tiny -$c$+$2$};
	\node at (212) {\BoxPART{3,2}{}};
	\node at (222) {\BoxPART{3,1}{1}};
	\node at (232) {\BoxPART{3}{2}};
	\node at (242) {\BoxPART{1,1,1}{1,1}};
	\node at (252) {\BoxPART{1}{3,1}};
	\node at (262) {\BoxPART{}{3,2}};
\draw (11) .. controls (-.8+1,-2) and (-.8+1,-3.5) .. (213) node[pos=.8,EdgeLabel]{\tiny $a$-$2$};
\draw (12) .. controls (-.8+2,-2) and (-.8+2,-3.5) .. (223) node[pos=.8,EdgeLabel]{\tiny $a$-$2$};
\draw (13) .. controls (-.8+3,-2) and (-.8+3,-3.5) .. (233) node[pos=.8,EdgeLabel]{\tiny -$c$+$2$};
\draw (14) .. controls (-.8+4,-2) and (-.8+4,-3.5) .. (243) node[pos=.8,EdgeLabel]{\tiny -$c$-$2$};
\draw (15) .. controls (-.8+5,-2) and (-.8+5,-3.5) .. (253) node[pos=.8,EdgeLabel]{\tiny -$c$-$2$};
\draw (16) .. controls (-.8+6,-2) and (-.8+6,-3.5) .. (263) node[pos=.8,EdgeLabel]{\tiny -$c$-$2$};
	\node at (213) {\BoxPART{2,2,1}{}};
	\node at (223) {\BoxPART{2,1,1}{1}};
	\node at (233) {\BoxPART{2}{3}};
	\node at (243) {\BoxPART{1,1}{1,1,1}};
	\node at (253) {\BoxPART{1}{2,1,1}};
	\node at (263) {\BoxPART{}{2,2,1}};
	\node at (11) {\BoxPARTs{2,2}{}};
	\node at (12) {\BoxPARTs{2,1}{1}};
	\node at (13) {\BoxPARTs{2}{2}};
	\node at (14) {\BoxPARTs{1,1}{1,1}};
	\node at (15) {\BoxPARTs{1}{2,1}};
	\node at (16) {\BoxPARTs{}{2,2}};
\end{scope}
\end{tikzpicture}
$$
\caption{Levels $-1$, $0$, and $1$ of a Bratteli diagram encoding isotypic components of $M \otimes N \otimes V$ where $a,c >2$ and $b=d=2$. The edges from level $-1$ to level $0$ are labeled by $e_0(T^{(0)})$ as in \eqref{eq:0edgelabel}; the edges from level $0$ to $1$ are labeled by the content of the box added. }
\label{fig:Brat_diagram_with_contents}
\end{figure}
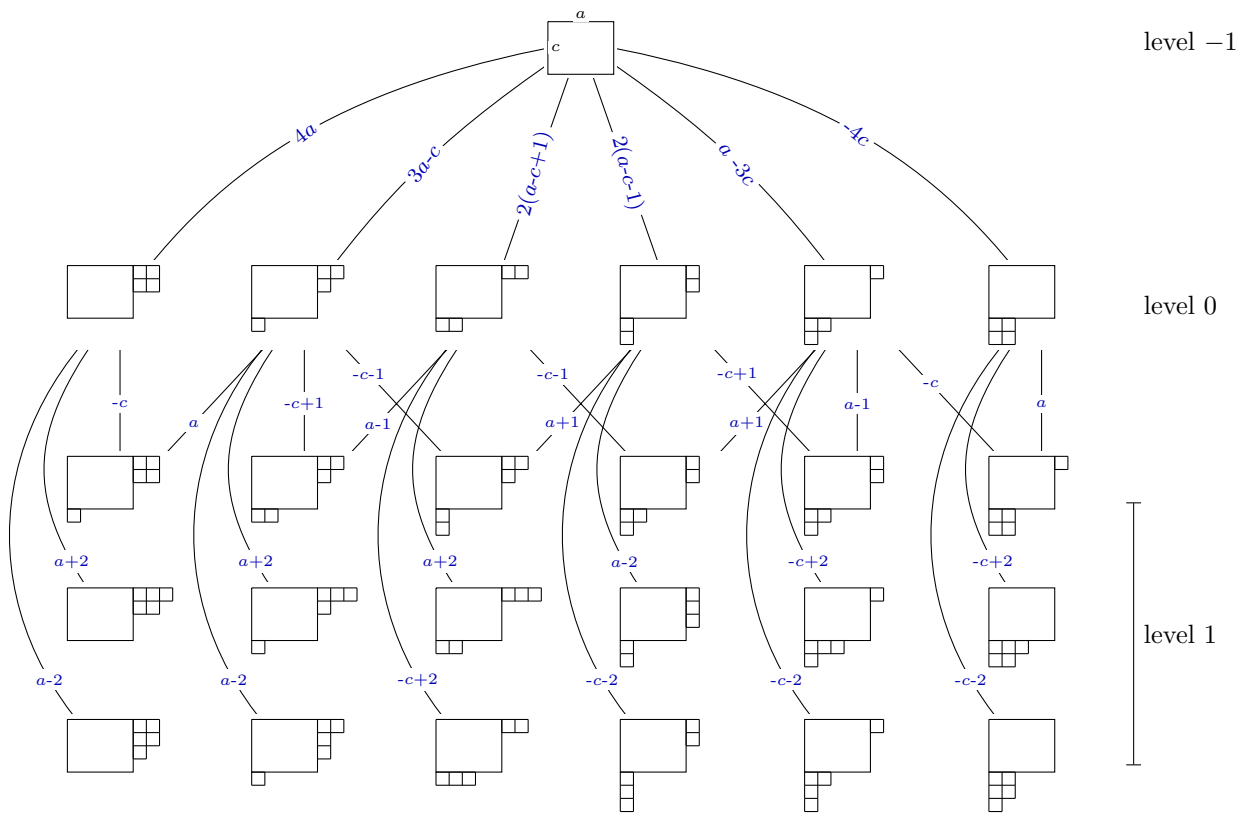

Let $\lambda\in \cP^{(k)}$.  
Define 
\begin{equation}
c_0 = -\half(k(a-c+b-d) 
+ac(a-c)+bd(b-d))+ \sum_{\mathrm{box}\in \lambda} c(\mathrm{box})
\quad\hbox{and}\quad
z = (-1)^k q^{2c_0}.
\label{eq:zetavalue}
\end{equation}
Using notation as in \eqref{eq:P0drawings}, let 
\begin{equation}
\mu^c = \lambda \cap \cB' \quad\hbox{and let 
$S_{\mathrm{max}}^{(0)}$ be the corresponding $\ring{\mu}$}.
\label{S0maxdefn}
\end{equation}
Define the \emph{shifted content} of a $\mathrm{box}$ by
\begin{equation}
\tilde{c}(\mathrm{box}) = c(\mathrm{box}) - \half(a-c+b-d),
\quad \hbox{and let}\ 
\cc= (c_1, \ldots, c_k)\hbox{ with } 0\le c_1\le c_2\le \cdots \le c_k
\label{ccdefn}
\end{equation}
be the sequence of absolute values of the shifted contents of the boxes in $\lambda/S_{\mathrm{max}}^{(0)}$
arranged in increasing order.
Index the boxes of $\lambda/S^{(0)}_{\mathrm{max}}$ with $1,2\ldots, k$ so that 
\begin{equation*}
\begin{array}{l}
\hbox{(a) if $i<j$ then $|\tilde c(\mathrm{box}_i)|\le |\tilde c(\mathrm{box}_j)|$,} \\
\hbox{(b) if $i<j$ and $\tilde c(\mathrm{box}_i) = \tilde c(\mathrm{box}_j)<0$ then $\mathrm{box}_i$ is SE of $\mathrm{box}_j$,} \\
\hbox{(c) if $i<j$ and $\tilde c(\mathrm{box}_i) = \tilde c(\mathrm{box}_j)\ge0$ then $\mathrm{box}_i$ is NW of $\mathrm{box}_j$,}\\
\hbox{(d) if $i<j$ and $\tilde c(\mathrm{box}_i) = -\tilde c(\mathrm{box}_j)$, then $\tilde c(\mathrm{box}_i) \leq 0 \leq \tilde c(\mathrm{box}_j)$,}
\end{array}
\label{boxindexing}
\end{equation*}
and define
\begin{align}
J = 
&\left\{ \vep_i ~|~ \tilde c(\mathrm{box}_i) \in \{-r_1, -r_2\} \right\}\nonumber\\
&\sqcup \left\{\vep_j - \vep_i ~\left|~ 
\begin{array}{l}
\hbox{$\tilde c(\mathrm{box}_j) = \tilde c(\mathrm{box}_i)+1>0$ and $\mathrm{box}_j$ is NW of $\mathrm{box}_i$, or} \\
\hbox{$\tilde c(\mathrm{box}_j) = \tilde c(\mathrm{box}_i)-1<0$ and $\mathrm{box}_j$ is SE of $\mathrm{box}_i$, or} \\
\hbox{$\tilde c(\mathrm{box}_j) = -\tilde c(\mathrm{box}_i)-1<0<\tilde c(\mathrm{box}_i)$}
\end{array}
\right.\right\} \label{eq:J-from-partition}\\
&\sqcup \left\{\vep_j+\vep_i ~\left|~ 
\begin{array}{l}
\hbox{$\tilde c(\mathrm{box}_j) = -1$ and $\tilde c(\mathrm{box}_i)=0$ and
$\mathrm{box}_j$ is SE of $\mathrm{box}_i$, or} \\
\hbox{$\tilde c(\mathrm{box}_j) = \frac12$ and $\tilde c(\mathrm{box}_i)=-\frac12$ and $\mathrm{box}_j$ is NW of $\mathrm{box}_i$, or} \\
\hbox{$\tilde c(\mathrm{box}_j) = -\frac12$ and $\tilde c(\mathrm{box}_i)=-\frac12$}
\end{array}
\right.\right\},\nonumber
\end{align}
so that $J$ is a subset of $P(\cc)$, where $P(\cc)$ is as defined in \eqref{P(c)origdefn}. See Examples \ref{ex:two-row-SW-modules} and \ref{ex:partitionstoSLRs} following the proof of Theorem \ref{thm:partitions-to-SLRs}.

\begin{thm}  \label{thm:partitions-to-SLRs}
Let $\fg = \fgl_n$ and let $M=L(a^c)$, $N= L(b^d)$ and $V= L(\square)$ so that
$H_k^{\mathrm{ext}}$ acts on $M\otimes N\otimes V^{\otimes k}$ as in 
Theorem \ref{thm:HeckeActionOnTensorSpace}. 
Assume that the genericity conditions of \eqref{newgenconds} hold so that
$q$ is not a root of unity, $a,b,c,d\in \ZZ_{>0}$ and $(a+c)-(b+d)\not\in \{0, \pm1, \pm2\}$.
For $\lambda\in \cP^{(k)}$, let $B_k^\lambda$ be the $H_k^{\mathrm{ext}}$-module of \eqref{BklambdaHeckedefn}
and define $z$, $\cc$ and $J$ as in \eqref{eq:zetavalue}, \eqref{ccdefn}, and \eqref{eq:J-from-partition}.  Then
\begin{equation}
B_k^\lambda \cong H_k^{(z,\cc,J)}\ \hbox{as $H_k^\ext$-modules}.
\end{equation}
\end{thm}
\begin{proof}
By Proposition \ref{prop:calibrated}, $B_k^\lambda$ is a calibrated $H_k^\ext$ module.
Therefore $B_k^\lambda$ has a composition series
with factors that are irreducible calibrated $H_k^\ext$-modules.
By  Theorem \ref{thm:calibconst}, each factor is isomorphic to some $H_k^{(z, \cc, J)}$ where $(\cc, J)$ is a skew local region, and $(z, \cc, J)$ is determined by the eigenvalues of the action of $W_0, W_1, \dots, W_k$.
By Proposition \ref{prop:calibrated}, the simultaneous eigenbasis $\{ v_S\ |\ S\in \cT_k^\lambda\}$
$B_k^\lambda$ is indexed by 
\begin{align}\label{eq:tableaux-paths}
\cT_k^\lambda &= \{ \hbox{paths $S=((a^c) \to S^{(0)}
\to S^{(1)}\to \cdots \to S^{(k)}=\lambda)$ in the Bratteli diagram}\}.
\end{align}
To determine which $H_k^{(z, \cc, J)}$ appear as composition factors of $B_k^\lambda$ it is necessary
to compute the eigenvalues of the action of the $W_i$'s on the basis vectors $v_S$ as follows.

By \eqref{c0edgelabel}, \eqref{cjedgelabel},
and the formulas in Proposition \ref{prop:calibrated},
$$\Phi(P)v_S = q^{2e_0(S^{(0)})}v_S
\qquad\hbox{and}\qquad
\Phi(Z_i)v_S = q^{2c(S^{(i)}/S^{(i-1)})} v_S\quad\hbox{for $i=1,\ldots, k$.}
$$
Using \eqref{eq:normalized_gens} and \eqref{eq:rectparams}, 
$W_i = -(a_1a_2b_1b_2)^{-\frac12}Z_i$ with
$a_1=q^{2a}$,
$a_2 = q^{-2c}$,
$b_1=q^{2b}$, and 
$b_2=q^{-2d}$, and thus
\begin{equation}
\Phi(W_i)v_S
=-(a_1a_2b_1b_2)^{-\frac12}\Phi(Z_i)v_S
= -q^{-(a-c+b-d)}q^{2c(S^{(i)}/S^{(i-1)}}v_S
= -q^{2\tilde{c}(S^{(i)}/S^{(i-1)})}v_S.
\label{Weigenvals}
\end{equation}
Then
$
\Phi(PW_1\cdots W_k)v_S
=(-1)^k q^{2\left(e_0(S^{(0)})+c(S^{(1)}/S^{(0)})+\cdots +c(S^{(k)}/S^{(k-1)})\right)-k(a-c+b-d)} v_S
$
so that, with $c_0$ and $z$ as in \eqref{eq:zetavalue},
\begin{equation}
\Phi(W_0) = \Phi(PW_1\cdots W_k)v_S = (-1)^k q^{2c_0} v_S= zv_S.
\label{W0eigenvals}
\end{equation}

Let $S=((a^c) \to S^{(0)} \to S^{(1)}\to \cdots \to S^{(k)}=\lambda)$ be a path to $\lambda$ in the Bratteli diagram.
In the context of the diagrams in \eqref{eq:P0drawings}, the partitions $S^{(0)}$ and $S_{\mathrm{max}}^{(0)}$ differ by moving 
some boxes from $\mu$ to $\mu^c$ (from the NW border of $\lambda/S_{\mathrm{max}}^{(0)}$ in $\cB$
to the NW border of $\lambda/S^{(0)}$ in $\cB'$).  
Thus the sequence $\cc=(c_1,\ldots,c_k)$, where 
$$\hbox{$c_1,\ldots, c_k$ are the values
$|\tilde c(S^{(1)}/S^{(0)})|, \ldots, |\tilde c(S^{(k)}/S^{(k-1)})|$ arranged in increasing order,}
$$
coincides with $\cc$ as defined in \eqref{ccdefn}.
Let $w_S\in \cW_0$ be the minimal length element such that 
\begin{equation}
w_S\cc = w_S(c_1, \ldots, c_k) = (c_{w_S^{-1}(1)}, \ldots, c_{w_S^{-1}(k)}) 
= (\tilde c(S^{(1)}/S^{(0)}), \ldots, \tilde c(S^{(k)}/S^{(k-1)})),
\label{findwS}
\end{equation}
where $c_{-i}=-c_i$ for $i\in \{1,\ldots, k\}$.
The signed permutation $w_S$ is the unique signed
permutation such that 
$$w_S\cc = (\tilde c(S^{(1)}/S^{(0)}), \ldots, \tilde c(S^{(k)}/S^{(k-1)}))
\qquad\hbox{and}\qquad
R(w_S)\cap Z(\cc) = \emptyset,$$
where $Z(\cc)$ is
as in \eqref{Z(c)origdefn}.
If the boxes of $\lambda/S^{(0)}$ are indexed according to the same conditions as 
just before \eqref{eq:J-from-partition},  
then $w_S$ is the signed permutation given by
$$w_S(i) = \sgn(\tilde{c}(\mathrm{box}_i))(\hbox{entry in $\mathrm{box}_i$ of $S$}),$$
where the path $S$ is identified with the standard tableau of shape $\lambda/S^{(0)}$ that has $S^{(j)}/S^{(j-1)}$ filled with $j$. 

The basis vector $v_S$ appears in a composition factor isomorphic to $H_k^{(z,\cc,J)}$ where 
$$J = R(w_S)\cap P(\cc),
\qquad\hbox{where}\quad
R(w_S) = R_1\sqcup R_2\sqcup R_3
\quad\hbox{and}\quad
P(\cc) = P_1\sqcup P_2\sqcup P_3,
$$
as defined in \eqref{R(w)origdefn} and \eqref{P(c)origdefn}, are given by
$$
\begin{array}{ll}
R_1 = \{ \vep_i\ |\ \hbox{$i>0$ and $w_S(i)<0$}\}, 
&P_1 = \{ \vep_i\ |\ c_i\in \{r_1,r_2\}\}, \\
R_2 =  \{ \vep_j - \vep_i\ |\ \hbox{$i<j$ and $w_S(i)>w_S(j)$}\}, 
&P_2 = \{\vep_j-\vep_i\ |\ \hbox{$0<i<j$, $c_j=c_i+1$}\}, \\
R_3 = \{ \vep_j+\vep_i\ |\ \hbox{$i<j$ and $-w_S(i)>w_S(j)$}\},
&P_3 =  \{\vep_j+\vep_i\ |\ \hbox{$0<i<j$, $c_j=-c_i+1$}\}.
\end{array}
$$
To describe $J = (R_1\cap P_1)\sqcup (R_2\cap P_2) \sqcup (R_3\cap P_3)$ in terms of the boxes in $\lambda$, first record that
$$R_1\cap P_1
= \{ \vep_i \ |\ \hbox{$i>0$ and $w_S(i)<0$}\}\cap \{ \vep_i\ |\ c_i\in \{r_1,r_2\}\}\\
= \{ \vep_i\ |\ \hbox{$\tilde c(\mathrm{box}_i)=\{-r_1,-r_2\}$}\}.
$$
Next analyze
$$R_2\cap P_2
=
\{ \vep_j - \vep_i\ |\ \hbox{$i<j$ and $w(i)>w(j)$}\}
\cap
\{\vep_j-\vep_i\ |\ \hbox{$0<i<j$, $c_j=c_i+1$}\}.
$$
Since $0\le c_i$ and $c_j = c_i+1$, we have $c_j\ge 1$.
$$\begin{array}{l}
\hbox{Case 1: $\tilde c(\mathrm{box}_i)\ge 0$, so that $\tilde c(\mathrm{box}_j)  = \pm (\tilde c(\mathrm{box}_i)+1)$.} \\
\qquad\hbox{Case 1a: $\tilde c(\mathrm{box}_j)  = \tilde c(\mathrm{box}_i)+1$.} \\
\qquad\qquad\hbox{If $\mathrm{box}_j$ is NW of $\mathrm{box}_i$ then
$w(j) < w(i)$ and $\vep_j-\vep_i\in J$.} \\
\qquad\qquad\hbox{If $\mathrm{box}_j$ is SE of $\mathrm{box}_i$ then
$w(j) > w(i)$ and $\vep_j-\vep_i\not\in J$.} \\
\qquad\hbox{Case 1b: $\tilde c(\mathrm{box}_j)  = -(\tilde c(\mathrm{box}_i)+1)$.} \\
\qquad\qquad\hbox{Then $w(j)<0< w(i)$ so that $w(j)<w(i)$ and $\vep_j-\vep_i\in J$.} \\
\hbox{Case 2: $\tilde c(\mathrm{box}_i)< 0$, so that $\tilde c(\mathrm{box}_j)  = \pm (-\tilde c(\mathrm{box}_i)+1)$.} \\
\qquad\hbox{Case 2a: $\tilde c(\mathrm{box}_j) = \tilde c(\mathrm{box}_i)-1<\tilde c(\mathrm{box}_i)<0$.} \\
\qquad\qquad\hbox{If $\mathrm{box}_j$ is NW of $\mathrm{box}_i$ then $-w(j)<-w(i)$ so that $w(i)<w(j)$ and $\vep_j-\vep_i\not\in J$.} \\
\qquad\qquad\hbox{If $\mathrm{box}_j$ is SE of $\mathrm{box}_i$ then $-w(j)>-w(i)$ so that $w(i)>w(j)$ and $\vep_j-\vep_i\in J$.} \\
\qquad\hbox{Case 2b: $\tilde c(\mathrm{box}_j)  = -\tilde c(\mathrm{box}_i)+1>0>\tilde c(\mathrm{box}_i)$.} \\
\qquad\qquad\hbox{Then $w(i)<0$ and $0< w(j)$ so that $\vep_j-\vep_i\not\in J$.} 
\end{array}
$$
Finally, analyze
$$R_3\cap P_3 = 
\{ \vep_j + \vep_i\ |\ \hbox{$i<j$ and $-w(i)>w(j)$}\}
\cap
\{\vep_j+\vep_i\ |\ \hbox{$0<i<j$, $c_j=-c_i+1$}\}.
$$
Since $0\le c_i$ and $c_j = -c_i+1\ge c_i$, we have $0\le c_i\le 1/2$.  Since the entries of $\cc$ are
in $\ZZ$ or in $\frac12 + \ZZ$, the possibilities for $(c_i,c_j)$ are $(0,1)$ and $(\frac12, \frac12)$,
and the possibilities for $(\tilde c(\mathrm{box}_i), \tilde c(\mathrm{box}_j))$ are 
$(0,1)$, $(0, -1)$, $(\frac12, \pm\frac12)$, or $(-\frac12, \pm\frac12)$.
$$\begin{array}{l}
\hbox{Case 1: $\tilde c(\mathrm{box}_j)  = 1$ and $\tilde c(\mathrm{box}_i)  = 0$.} \\
\qquad\hbox{If $\mathrm{box}_j$ is NW of $\mathrm{box}_i$ then
$0<w(j) < w(i)$ so that $-w(i)<0<w(j)$ and $\vep_j+\vep_i\not\in J$.} \\
\qquad\hbox{If $\mathrm{box}_j$ is SE of $\mathrm{box}_i$ then
$0<w(i)<w(j)$ so that $-w(i)<0<w(j)$ and $\vep_j+\vep_i\not\in J$. } \\ 
\hbox{Case 2: $\tilde c(\mathrm{box}_j)  = -1$ and $\tilde c(\mathrm{box}_i)  = 0$.} \\
\qquad\hbox{If $\mathrm{box}_j$ is NW of $\mathrm{box}_i$ then $-w(j)<w(i)$ so that $-w(i)<w(j)$ and 
$\vep_j+\vep_i\not\in J$.} \\
\qquad\hbox{If $\mathrm{box}_j$ is SE of $\mathrm{box}_i$ then $-w(j)>w(i)$ so that
$-w(i)>w(j)$ and $\vep_j+\vep_i\in J$.  } \\
\hbox{Case 3: $\tilde c(\mathrm{box}_j)  = \frac12$ and $\tilde c(\mathrm{box}_i)  = \frac12$.} \\
\qquad\hbox{Then $0<w(i)<w(j)$ so that $-w(i)<0<w(j)$ and $\vep_j+\vep_i\not\in J$.} \\
\hbox{Case 4: $\tilde c(\mathrm{box}_j)  = -\frac12$ and $\tilde c(\mathrm{box}_i)  = \frac12$.} \\
\qquad\hbox{
This case cannot occur since, when indexing the boxes of $\lambda/S^{(0)}$,} \\
\qquad\hbox{ the boxes of shifted content $-\frac12$ are numbered before the boxes of shifted content $\frac12$.} \\
\hbox{Case 5: $\tilde c(\mathrm{box}_j)  = \frac12$ and $\tilde c(\mathrm{box}_i)  = -\frac12$.} \\
\qquad\hbox{If $\mathrm{box}_j$ is NW of $\mathrm{box}_i$ then $w(i)<0$ and $w(j)<-w(i)$ so that 
$\vep_j+\vep_i\in J$.} \\
\qquad\hbox{If $\mathrm{box}_j$ is SE of $\mathrm{box}_i$ then $w(i)<0$ and $-w(i)<w(i)$ so that
$\vep_j+\vep_i\not\in J$.  } \\
\hbox{Case 6: $\tilde c(\mathrm{box}_j)  = -\frac12$ and $\tilde c(\mathrm{box}_i)  = -\frac12$.} \\
\qquad\hbox{Then $0<-w(j)<-w(i)$ and $w(j)<0<-w(i)$ so that $\vep_j+\vep_i\in J$.}
\end{array}
$$
This analysis shows that $J = R(w_S)\cap P(\cc)  = (R_1\cap P_1) \sqcup (R_2\cap P_2)\sqcup (R_3\cap P_3)$ is
as given in \eqref{eq:J-from-partition}.  

A consequence of the description of $J$ in \eqref{eq:J-from-partition} is that
$J=R(w_S)\cap P(\cc)$ is independent of the choice of $S\in \cT_k^\lambda$.  It follows that all composition factors
of $B_k^\lambda$ are isomorphic to $H_k^{(z,\cc,J)}$.

Let $S,T\in \cT_k^\lambda$ such that $v_S$ and $v_T$ have the same eigenvalues for $W_0, \ldots, W_k$.  
By definition of $\cT_k^\lambda$, $S^{(k)}=T^{(k)}=\lambda$. 
Since $W_kv_S = -q^{\tilde c(S^{(k)}/S^{(k-1)})}v_S=-q^{\tilde c(\lambda/S^{(k-1)})}v_S$ 
and $W_kv_T = -q^{\tilde c(T^{(k)}/T^{(k-1)})}v_T=-q^{\tilde c(\lambda/T^{(k-1)})}v_T$, we have
$\tilde c(\lambda/T^{(k-1)}) = \tilde c(\lambda/S^{(k-1)})$ which implies that $T^{(k-1)}=S^{(k-1)}$.  Using this and
the fact that the eigenvalues of $W_{k-1}$ on $v_S$ and $v_T$ are the same,  implies similarly that $T^{(k-2)}=S^{(k-2)}$.
Induction gives that 
$$S^{(0)}=T^{(0)},\quad \ldots,\quad S^{(k)}=T^{(k)} \qquad\hbox{so that}\qquad
S=T.$$  
Thus $\dim((B_k^\lambda)_{\gamma})\le 1$ (in the notation of \eqref{calibrateddefn}) 
and $B_k^\lambda\cong H_k^{(z,\cc,J)}$ as $H_k^{\mathrm{ext}}$-modules.
\end{proof}

\noindent
In the course of the proof of Theorem  \ref{thm:partitions-to-SLRs} we have also established the following result, which deserves mention.

\begin{cor}  \label{onStdTableaux} Keeping the notations of Theorem \ref{thm:partitions-to-SLRs},
let $\lambda\in \cP^{(k)}$ and $S \in \cT_k^\lambda$, and let $w_S$ be the signed permutation defined in \eqref{findwS}.
Then
\begin{equation*}
\begin{matrix}
\cT_k^\lambda &\longrightarrow &\cF^{(\cc,J)} \\
S &\longmapsto &w_S \\
\end{matrix}
\quad\hbox{is a bijection.}
\label{oldpartitionstolocalregions}
\end{equation*}
\end{cor}

\begin{example}\label{ex:two-row-SW-modules} Let $M=L(a^c) = L(6)$ and $N= L(b^d) = L(3)$ so that
$$a = 6, \quad c=1, \quad b = 3, \quad d=1, \quad r_1 = \hbox{$\frac32$}, \quad\hbox{and}\quad
r_2 = \hbox{$\frac{11}{2}$}.
$$
The partition $\lambda = (10,8)$ is in $\cP^{(k)}$ with $k=9$.  Then we draw $\lambda$ as the (marked) partition
$$
\lambda = (10,8) = \begin{matrix}
\begin{tikzpicture}[xscale=.4, yscale=-.4] 
	\filldraw[black!40] (0,0) to (6,0) to (6,1) to (3,1) to (3,2) to (0,2) to (0,0);
	\Part{10,8}
	 \node[bV,red] at (9,0){}; \node[bV,red!50!blue] at (6,1){}; \node[bV,red!50!blue] at (3,1){}; \node[bV,red] at (0,2){};
\end{tikzpicture}
\end{matrix}.
\qquad\hbox{Here, \quad $S^{(0)}_{\mathrm{max}} = (6,3)$}
$$
is indicated by the shaded boxes. 
The boxes of $\lambda/S^{(0)}_{\mathrm{max}}$ have
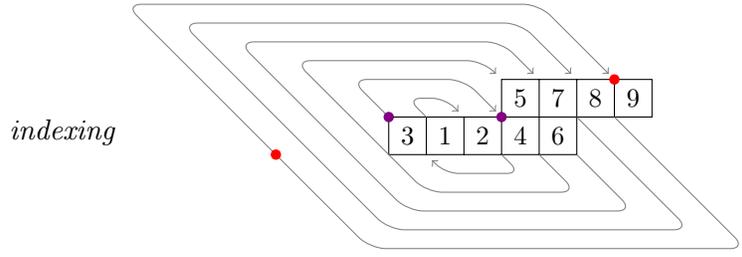
$$\text{indexing } 
\begin{matrix}\begin{tikzpicture}[xscale=.5,yscale=-.5]
	\draw (6,2) to (6,0) to (10,0) to (10,1) to (3,1) to (3,2) to (8,2) to (8,0) (7,2) to (7,0) (4,2) to (4,1) (5,1) to (5,2) (9,0) to (9,1);
	\Cont{3.5,1.5}{$3$}
	\Cont{4.5,1.5}{$1$}
	\Cont{5.5,1.5}{$2$}
	\Cont{6.5,1.5}{$4$}
	\Cont{6.5,.5}{$5$}
	\Cont{7.5,1.5}{$6$}
	\Cont{7.5,.5}{$7$}
	\Cont{8.5,.5}{$8$}
	\Cont{9.5,.5}{$9$}
	\begin{scope}[->, rounded corners=5pt, black!50!white, shorten >= 3pt]
	\draw (4, 1) to ++(-.5,-.5) to ++(1,0) to ++(.5,.5);
	\draw (6, 2)  to ++(.5,.5) to ++(-2,0) to ++(-.5,-.5);
	\end{scope}
	\begin{scope}[->, rounded corners=6pt, black!50!white, shorten >= 3pt]
	\draw (3, 1) to ++(-1,-1) to ++(3,0) to ++(1,1);
	\draw (7, 2)  to ++(1,1) to ++(-4,0) to ++(-3.5,-3.5) to ++(5,0) to ++(.5,.5);
	\draw (8, 2)  to ++(1.5,1.5) to ++(-6,0) to ++(-4.5,-4.5) to ++(7,0) to ++(1,1);
	\draw (8, 1)  to ++(3,3) to ++(-8,0) to ++(-5.5,-5.5) to ++(9,0) to ++(1.5,1.5);
	\draw (9, 1)  to ++(3.5,3.5) to ++(-10,0) to ++(-6.5,-6.5) to ++(11,0) to ++(2,2);
	\end{scope}
	  \node[bV,red] at (9,0){}; \node[bV,red!50!blue] at (6,1){}; \node[bV,red!50!blue] at (3,1){};  \node[bV,red] at (0,2){};
\end{tikzpicture}\end{matrix},$$
$$\text{and shifted contents } 
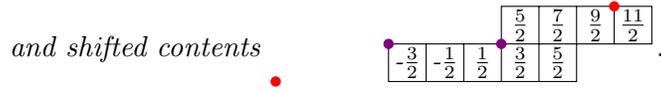
\begin{matrix}\begin{tikzpicture}[xscale=.5,yscale=-.5]
	\draw (6,2) to (6,0) to (10,0) to (10,1) to (3,1) to (3,2) to (8,2) to (8,0) (7,2) to (7,0) (4,2) to (4,1) (5,1) to (5,2) (9,0) to (9,1);
	\Cont{3.5,1.5}{-$\frac32$}
	\Cont{4.5,1.5}{-$\frac12$}
	\Cont{5.5,1.5}{$\frac12$}
	\Cont{6.5,1.5}{$\frac32$}
	\Cont{6.5,.5}{$\frac52$}
	\Cont{7.5,1.5}{$\frac52$}
	\Cont{7.5,.5}{$\frac72$}
	\Cont{8.5,.5}{$\frac92$}
	\Cont{9.5,.5}{$\frac{11}2$}
	 \node[bV,red] at (9,0){}; \node[bV,red!50!blue] at (6,1){}; \node[bV,red!50!blue] at (3,1){}; \node[bV,red] at (0,2){};
\end{tikzpicture}\end{matrix}.
$$
\end{example}

\begin{example}
\label{ex:partitionstoSLRs}
\tikzstyle{C}=[draw, solid, circle, inner sep=0pt, minimum size=4pt]
\def\ExA{
\Part{9,9,6,6,6,2,1,1,1}
\draw[blue, very thick](0,0)--(0,4)--(5,4)--(5,0)--(0,0);
\draw[dashed, blue!50, thick] (5,0)--(8,0)--(8,3)--(5,3);
\draw[dashed, blue!50, thick] (0,4)--(0,7)--(3,7)--(3,4);
}
Let $M=L(a^c)=L(5^4)$ and $N = L(b^d)=L(3^3)$ so that
$$
\hbox{$a=5$,\quad $c=4$,\quad $b=3$\quad $d=3$,\quad $r_1 = \frac{3}{2}$,\quad and\quad $r_2 = \frac{15}{2}$.}
$$
The partition  $\lambda = (9,9,6,6,6,2,1,1,1)$ is in $\cP^{(k)}$ with $k = 12$.  For this partition $S_{\mathrm{max}}^{(0)} = (7,6,5,5,3,2,1)$; and one tableau $S \in \cT_k^\lambda$ with $S^{(0)} = S^{(0)}_{\mathrm{max}}$ (where the shaded portion of $\lambda$ corresponds to $S^{(0)}$)
is 
$$
S=
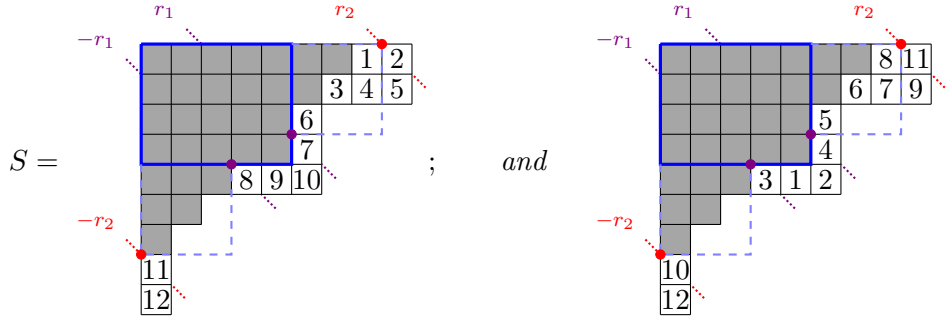
\begin{matrix} 
\begin{tikzpicture}[xscale=.4,yscale=-.4]
\draw[densely dotted, thick, red!50!blue] (1.5, -.5) node[above left] {\scriptsize$r_1$} to (6.5, 4.5);
\draw[densely dotted, thick, red!50!blue] (-.5, .5) node[above left] {\scriptsize$-r_1$} to (4.5, 5.5);
\draw[densely dotted, thick,, red] (-.5, 6.5) node[above left] {\scriptsize$-r_2$} to (1.5, 8.5);
\draw[densely dotted, thick, red] (7.5, -.5) node[above left] {\scriptsize$r_2$} to (9.5, 1.5);
\filldraw[white] (0,0)--(9,0)--(9,2)--(6,2)--(6,5)--(2,5)--(2,6)--(1,6)--(1,9)--(0,9);
\filldraw[gray!70] (0,0)--(7,0)--(7,1)--(6,1)--(6,2)--(5,2)--(5,4)--(3,4)--(3,5)--(2,5)--(2,6)--(1,6)--(1,7)--(0,7);
\node at (7.5,.5) {$1$};
\node at (8.5,.5) {$2$};
\node at (6.5,1.5) {$3$};
\node at (7.5,1.5) {$4$};
\node at (8.5,1.5) {$5$};
\node at (5.5,2.5) {$6$};
\node at (5.5,3.5) {$7$};
\node at (3.5,4.5) {$8$};
\node at (4.5,4.5) {$9$};
\node at (5.5,4.5) {$10$};
\node at (.5,7.5) {$11$};
\node at (.5,8.5) {$12$};
\ExA
\node[bV,red] at (8,0){};
\node[bV,red] at (0,7){};
\node[bV,red!50!blue] at (5,3){};
\node[bV,red!50!blue] at (3,4){};
\end{tikzpicture}
\end{matrix};
\qquad\hbox{and \quad }
\begin{matrix} 
\begin{tikzpicture}[xscale=.4,yscale=-.4]
\draw[densely dotted, thick, red!50!blue] (1.5, -.5) node[above left] {\scriptsize$r_1$} to (6.5, 4.5);
\draw[densely dotted, thick, red!50!blue] (-.5, .5) node[above left] {\scriptsize$-r_1$} to (4.5, 5.5);
\draw[densely dotted, thick, red] (-.5, 6.5) node[above left] {\scriptsize$-r_2$} to (1.5, 8.5);
\draw[densely dotted, thick, red] (7.5, -.5) node[above left] {\scriptsize$r_2$} to (9.5, 1.5);
\filldraw[white] (0,0)--(9,0)--(9,2)--(6,2)--(6,5)--(2,5)--(2,6)--(1,6)--(1,9)--(0,9);
\filldraw[gray!70] (0,0)--(7,0)--(7,1)--(6,1)--(6,2)--(5,2)--(5,4)--(3,4)--(3,5)--(2,5)--(2,6)--(1,6)--(1,7)--(0,7);
\node at (7.5,.5) {$8$};
\node at (8.5,.5) {$11$};
\node at (6.5,1.5) {$6$};
\node at (7.5,1.5) {$7$};
\node at (8.5,1.5) {$9$};
\node at (5.5,2.5) {$5$};
\node at (5.5,3.5) {$4$};
\node at (3.5,4.5) {$3$};
\node at (4.5,4.5) {$1$};
\node at (5.5,4.5) {$2$};
\node at (.5,7.5) {$10$};
\node at (.5,8.5) {$12$};
\ExA
\node[bV,red] at (8,0){};
\node[bV,red] at (0,7){};
\node[bV,red!50!blue] at (5,3){};
\node[bV,red!50!blue] at (3,4){};
\end{tikzpicture}
\end{matrix}
$$
indicates the indexing of the boxes in $\lambda/S_{\mathrm{max}}^{(0)}$.  The contents of the boxes $S^{(i)}/S^{(i-1)}$ for $i = 1, \dots, k$ are $7,8,5,6,7,3,2,-1,0,1,-7,-8$; and since $-\hbox{$\frac12$}(a-c+b-d)=-\hbox{$\frac12$}$, the shifted contents $\tilde{c}(S^{(i)}/S^{(i-1)})$ for $i=1, \dots, k$ are
$$\frac{13}{2},\ \frac{15}{2},\ \frac{9}{2},\ \frac{11}{2},\ \frac{13}{2},\ \frac{5}{2},\
\frac{3}{2},\ -\frac{3}{2},\ -\frac{1}{2},\ \frac{1}{2},\ -\frac{15}{2},\ -\frac{17}{2},$$
respectively.
The sum of the contents of the boxes in $S_{\mathrm{max}}^{(0)}$ is 1, the sum of the 
contents of the boxes in $\lambda$ is 23, 
$c_0 = - \half(12(5-4+3-3) + 5 \cdot 4(5-4) + 3 \cdot 3 (3-3)) + 24
		= 8$,
$$z=q^{16},
\qquad\hbox{and}\qquad
\cc = (\hbox{$\frac{1}{2}, \frac{1}{2}, \frac{3}{2}, \frac{3}{2}, \frac{5}{2}, \frac{9}{2}, \frac{11}{2}, \frac{13}{2}, \frac{13}{2}, \frac{15}{2}, \frac{15}{2}, \frac{17}{2}$})
$$
is the sequence of absolute values of the shifted contents, arranged in increasing order.
Using \eqref{findwS},
\begin{align*}
w_S &= \left( \begin{array}{cccccccccccc} 1 &2 &3 &4 &5 &6 &7 &8 &9 &10 &11 &12 \\ 
-9 &10 &-8 &7 &6 &3 &4 &1 &5 &-11 &2 &-12
\end{array}\right), \\
P(\cc) &= \left\{
\begin{array}{l}
\fbox{$\vep_3$},
\vep_4,
\fbox{$\vep_{10}$},
\vep_{11}, 
\vep_2-\vep_{-1}, \\
\vep_3-\vep_1, \vep_4-\vep_1,
\fbox{$\vep_3-\vep_2$}, \fbox{$\vep_4-\vep_2$},
\vep_5-\vep_3, \fbox{$\vep_5-\vep_4$}, \\
 \vep_7-\vep_6,
\fbox{$\vep_8-\vep_7$}, \vep_9-\vep_7, \\
\fbox{$\vep_{10}-\vep_{8}$}, \vep_{11}-\vep_{8},
\fbox{$\vep_{10}-\vep_{9}$}, \fbox{$\vep_{11}-\vep_{9}$},
\fbox{$\vep_{12}-\vep_{10}$}, \fbox{$\vep_{12}-\vep_{11}$},
\end{array}
\right\}, \\
R(w_S) &= \left\{ 
\begin{array}{l}
\vep_1, \fbox{$\vep_3$}, \fbox{$\vep_{10}$}, \vep_{12} \\
\vep_{10}-\vep_1, 
\vep_{12}-\vep_1, 
\fbox{$\vep_3-\vep_2$},
\fbox{$\vep_4-\vep_2$},
\vep_5-\vep_2,
\vep_6-\vep_2,
\vep_7-\vep_2,
\vep_8-\vep_2, 
\vep_9-\vep_2,
\vep_{10}-\vep_2,\\
\vep_{11}-\vep_2,
\vep_{12}-\vep_2, 
\vep_{10}-\vep_3,
\vep_{12}-\vep_3, 
\fbox{$\vep_5-\vep_4$},
\vep_6-\vep_4,
\vep_7-\vep_4,
\vep_8-\vep_4,
\vep_9-\vep_4,
\vep_{10}-\vep_4,\\
\vep_{11}-\vep_4,
\vep_{12}-\vep_4, 
\vep_6-\vep_5,
\vep_7-\vep_5,
\vep_8-\vep_5,
\vep_9-\vep_5,
\vep_{10}-\vep_5,
\vep_{11}-\vep_5,
\vep_{12}-\vep_5, 
\vep_8-\vep_6,\\
\vep_{10}-\vep_6,
\vep_{11}-\vep_6,
\vep_{12}-\vep_6, 
\fbox{$\vep_8-\vep_7$},
\vep_{10}-\vep_7,
\vep_{11}-\vep_7,
\vep_{12}-\vep_7, 
\fbox{$\vep_{10}-\vep_8$},
\vep_{12}-\vep_8, \\
\fbox{$\vep_{10}-\vep_9$},
\fbox{$\vep_{11}-\vep_9$},
\vep_{12}-\vep_9, 
\fbox{$\vep_{12}-\vep_{10}$},
\fbox{$\vep_{12}-\vep_{11}$}, \\
\vep_3+\vep_1,
\vep_4+\vep_1,
\vep_5+\vep_1,
\vep_6+\vep_1,
\vep_7+\vep_1,
\vep_8+\vep_1,
\vep_9+\vep_1,
\vep_{10}+\vep_1,
\vep_{11}+\vep_1,
\vep_{12}+\vep_1,\\ 
\vep_{10}+\vep_2,
\vep_{12}+\vep_2, 
\vep_4+\vep_3,
\vep_5+\vep_3,
\vep_6+\vep_3,
\vep_7+\vep_3,
\vep_8+\vep_3,
\vep_9+\vep_3,
\vep_{10}+\vep_3,
\vep_{11}+\vep_3,\\
\vep_{12}+\vep_3, 
\vep_{10}+\vep_4,
\vep_{12}+\vep_4, 
\vep_{10}+\vep_5,
\vep_{12}+\vep_5, 
\vep_{10}+\vep_6,
\vep_{12}+\vep_6, 
\vep_{10}+\vep_7,
\vep_{12}+\vep_7,\\ 
\vep_{10}+\vep_8,
\vep_{12}+\vep_8, 
\vep_{10}+\vep_9,
\vep_{12}+\vep_9, 
\vep_{11}+\vep_{10},
\vep_{12}+\vep_{10}, 
\vep_{12}+\vep_{11},
\end{array}
\right\},
\end{align*}
and $J = R(w_S)\cap P(\cc)$ consists of the outlined elements of $P(\cc)$ (which are the same as the outlined elements of $R(w_S)$).
Another $T \in \cT_{k}^\lambda$ is (again, with $T^{(0)}$ indicated by the shaded boxes)
$$
T=
\begin{matrix} 
\begin{tikzpicture}[xscale=.4,yscale=-.4]
\draw[densely dotted, thick, red!50!blue] (1.5, -.5) node[above left] {\scriptsize$r_1$} to (6.5, 4.5);
\draw[densely dotted, thick, red!50!blue] (-.5, .5) node[above left] {\scriptsize$-r_1$} to (4.5, 5.5);
\draw[densely dotted, thick,, red] (-.5, 6.5) node[above left] {\scriptsize$-r_2$} to (1.5, 8.5);
\draw[densely dotted, thick, red] (7.5, -.5) node[above left] {\scriptsize$r_2$} to (9.5, 1.5);
\filldraw[white] (0,0)--(9,0)--(9,2)--(6,2)--(6,5)--(2,5)--(2,6)--(1,6)--(1,9)--(0,9);
\filldraw[gray!70] (0,0)--(8,0)--(8,2)--(5,2)--(5,4)--(3,4)--(3,5)--(0,5);
\node at (.5,6.5) {$4$};
\node at (8.5,.5) {$2$};
\node at (1.5,5.5) {$3$};
\node at (.5,5.5) {$1$};
\node at (8.5,1.5) {$10$};
\node at (5.5,2.5) {$7$};
\node at (5.5,3.5) {$8$};
\node at (3.5,4.5) {$6$};
\node at (4.5,4.5) {$9$};
\node at (5.5,4.5) {$12$};
\node at (.5,7.5) {$5$};
\node at (.5,8.5) {$11$};
\ExA
\node[bV,red] at (8,0){};
\node[bV,red] at (0,7){};
\node[bV,red!50!blue] at (5,3){};
\node[bV,red!50!blue] at (3,4){};
\end{tikzpicture}
\end{matrix}.
$$
\end{example}

Keeping the setting of Theorem \ref{thm:partitions-to-SLRs}, Proposition \ref{prop:stdtabbijection} associates a configuration of 
$2k$ boxes to $(\cc,J)$.  This configuration can be described in terms of the data of $\lambda\in \cP^{(k)}$ as follows.
With $S_{\mathrm{max}}^{(0)}$ as defined just before \eqref{ccdefn},
let $\mathrm{rot}(\lambda/S_{\mathrm{max}}^{(0)})$ be the $180^\circ$ rotation of the skew shape $\lambda/S_{\mathrm{max}}^{(0)}$.  
Then
\begin{equation}
\hbox{the configuration of boxes $\kappa$ corresponding to 
		$(\cc,J)$} \quad\hbox{is}\quad
		\kappa=\mathrm{rot}(\lambda/S_{\mathrm{max}}^{(0)})\cup \lambda/S_{\mathrm{max}}^{(0)},
\label{eq:correspndingdoubledshape}
\end{equation}
so that it is the (disjoint) union of two skew shapes $\lambda/S_{\mathrm{max}}^{(0)}$ and $\mathrm{rot}(\lambda/S_{\mathrm{max}}^{(0)})$, 
placed 
with
\begin{enumerate}[\quad]
\item  $\mathrm{rot}(\lambda/S^{(0)})$ northwest of $\lambda/S^{(0)}$,
\item  $\lambda/S^{(0)}$ positioned so that the contents of its boxes are $(\tilde c(S^{(1)}/S^{(0)}), \ldots, \tilde c(S^{(k)}/S^{(k-1)}))$,
\item $\mathrm{rot}(\lambda/S^{(0)})$ positioned so that the contents of its boxes are $(-\tilde c(S^{(k)}/S^{(k-1)}), \ldots, -\tilde c(S^{(1)}/S^{(0)}))$,
\end{enumerate}
and with markings placed at the NE corners of the rectangles $\cB$ and $\cB'$ corresponding to $\lambda/S^{(0)}$ (in the notation of \eqref{eq:P0drawings}).
The resulting doubled skew shape is symmetric under the 180$^\circ$ rotation which sends a box on diagonal $c_i$ to a box on diagonal $-c_i$. 
In the case of Example \ref{ex:partitionstoSLRs} the corresponding configuration of boxes is
$$
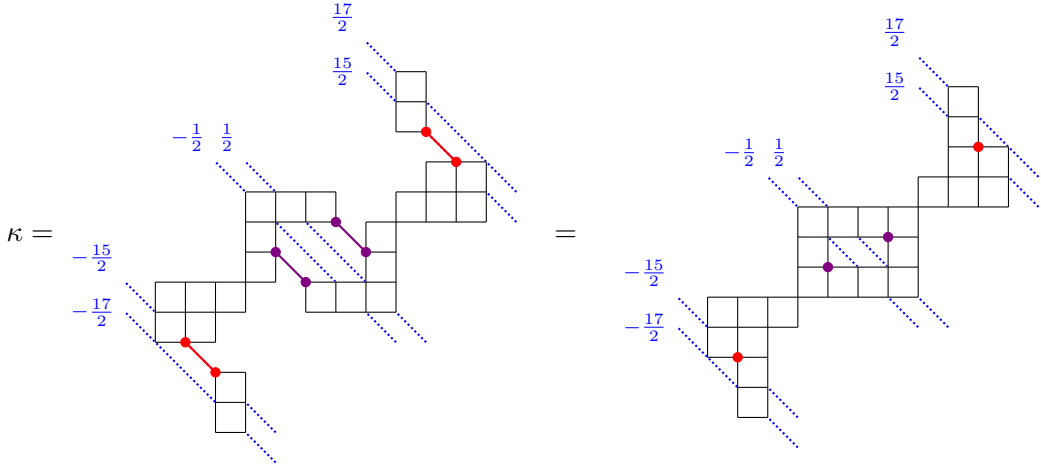
\kappa=
\begin{matrix}
\begin{tikzpicture}[xscale=.4,yscale=-.4]
\draw[blue, densely dotted, thick]  (7,6)--(1,0) node[above left]{\scriptsize $\frac12$};
\draw[blue, densely dotted, thick]  (6,6)--(0,0) node[above left]{\scriptsize $-\frac12$};
\draw[blue, densely dotted, thick]  (2,10)--(-3,5) node[left]{\scriptsize $-\frac{17}{2}$};
\draw[blue, densely dotted, thick]  (10,1)--(5,-4) node[above left]{\scriptsize $\frac{17}{2}$};
\draw[blue, densely dotted, thick]  (2,9)--(-3,4) node[above left]{\scriptsize $-\frac{15}{2}$};
\draw[blue, densely dotted, thick]  (10,2)--(5,-3) node[left]{\scriptsize $\frac{15}{2}$};
\filldraw[white] 
		(6,-3) rectangle (7,-1)
		(7,0) to ++(2,0) to ++(0,2) to ++(-3,0) to ++(0,-1) to ++(1,0) to ++(0,-1)
			(5,2) rectangle (6,4)
			(3,4) rectangle (6,5)
			(1,4) rectangle (2,2)
			(1,2) rectangle (4,1)
		(0,7) rectangle ++(1,2)
			(0,6) to ++(-2,0) to ++(0,-2) to ++(3,0) to ++(0,1) to ++(-1,0) to ++(0,1);
\draw (7,0)--(9,0) (6,1)--(9,1) (5,2)--(9,2) (5,3)--(6,3) (3,4)--(6,4) (3,5)--(6,5) (0,7)--(1,7) (0,8)--(1,8) (0,9)--(1,9);
\draw (0,7)--(0,9) (1,7)--(1,9) (3,4)--(3,5) (4,4)--(4,5) (5,2)--(5,5) (6,1)--(6,5) (7,0)--(7,2) (8,0)--(8,2) (9,0)--(9,2);
\node[bV,red] at (8,0){};
\node[bV,red!50!blue] at (5,3){};
\node[bV,red!50!blue] at (3,4){};
\node[bV,red] at (0,7){};
\pgftransformshift{\pgfpoint{-1cm}{-1cm}}
\draw (7,-2)--(8,-2) (7,-1)--(8,-1) (7,0)--(8,0) (2,2)--(5,2) (2,3)--(5,3) (2,4)--(3,4) (-1,5)--(3,5) (-1,6)--(2,6) (-1,7)--(1,7);
\draw (-1,5)--(-1,7) (0,5)--(0,7) (1,5)--(1,7) (2,2)--(2,6) (3,2)--(3,5) (4,2)--(4,3) (5,2)--(5,3)  (7,-2)--(7,0) (8,-2)--(8,0);
\node[bV,red] at (8,0){}; \draw[red,  thick] (8,0)-- +(1,1);
\node[bV,red!50!blue] at (5,3){}; \draw[red!50!blue, thick] (5,3)-- +(1,1);
\node[bV,red!50!blue] at (3,4){}; \draw[red!50!blue, thick] (3,4)-- +(1,1);
\node[bV,red] at (0,7){}; \draw[red, thick] (0,7)-- +(1,1);
\end{tikzpicture}\end{matrix}
\quad  = \quad 
\begin{matrix}\begin{tikzpicture}[xscale=.4,yscale=-.4]
	\ShapeA
\end{tikzpicture}\end{matrix}$$
This configuration of boxes also appeared in Example \ref{ex:stdtabbijection}.

For generically large $a,b,c,d$, there will be examples of $\lambda,\mu\in \cP^{(k)}$ with $\lambda \ne \mu$ and
$B_k^{\lambda}\cong B_k^\mu$ as $H_k^{\mathrm{ext}}$-modules; see Example \ref{ex:isomorphicPartitions}. This is because
the eigenvalues of $P$ on $M\otimes N$ are not sufficient to distinguish the components of 
$M\otimes N$ as a $\fgl_n$-module. It could be helpful
to further extend $H_k^\ext$ and consider an algebra $Z(U_q \fgl_n) \otimes H_k$ acting on $M\otimes N\otimes V^{\otimes k}$.

\begin{example}
\label{ex:isomorphicPartitions}
Let $a=c=6$ and $b=d=4$,
$$\lambda(k)=(11+k,10,8,8,6,6,5,3,3,1)\qquad\hbox{and}\qquad
\mu(k)=(11+k,9,9,8,7,6,4,3,2,2), \text{ i.e.}$$
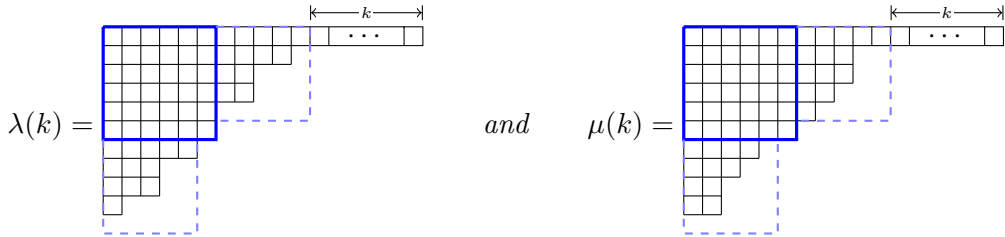
$$
\lambda(k) = 
\begin{matrix}
\begin{tikzpicture}[xscale= .25, yscale= -.25]
\Part{12,10,8,8,6,6,5,3,3,1}
\draw (12,0)--(17,0)--(17,1)--(12,1) (16,0)--(16,1);
\draw[|<->|] (11,-.75) to node[fill=white, inner sep=.5pt] {\tiny$k$} (17,-.75);
\node at (14,.5) {$\cdots$};
\draw[blue, very thick](0,0)--(0,6)--(6,6)--(6,0)--(0,0);
\draw[dashed, blue!50, thick] (6,0)--(11,0)--(11,5)--(6,5);
\draw[dashed, blue!50, thick] (0,6)--(0,11)--(5,11)--(5,6);
\end{tikzpicture}\end{matrix}
\qquad 
\text{and}
\qquad
\mu(k) = 
\begin{matrix}
\begin{tikzpicture}[xscale= .25, yscale= -.25]
\Part{12,9,9,8,7,6,4,3,2,2}
\draw (12,0)--(17,0)--(17,1)--(12,1) (16,0)--(16,1);
\draw[|<->|] (11,-.75) to node[fill=white, inner sep=.5pt] {\tiny$k$} (17,-.75);
\node at (14,.5) {$\cdots$};
\draw[blue, very thick](0,0)--(0,6)--(6,6)--(6,0)--(0,0);
\draw[dashed, blue!50, thick] (6,0)--(11,0)--(11,5)--(6,5);
\draw[dashed, blue!50, thick] (0,6)--(0,11)--(5,11)--(5,6);
\end{tikzpicture}\end{matrix}.~
$$
Then $\lambda(k)\ne \mu(k)$ but, as $H_k^\ext$-modules,
$$B_k^{\lambda(k)} \cong B_k^{\mu(k)} \cong H_k^{(z,\cc,\emptyset)}, 
\qquad\hbox{where $\cc=(11, 12, \ldots, 11+k-1)$ and $z=q^{28 + k(k+21)}$.}
$$
\end{example}

Recall from \eqref{BklambdaHeckedefn} that
\begin{equation*}
M\otimes N\otimes V^{\otimes k} \cong \bigoplus_{\lambda\in \cP^{(k)}} L(\lambda)\otimes B_k^\lambda
\qquad\hbox{as $(U_q\fg, \cH_k^{\mathrm{ext}})$-bimodules.}
\end{equation*}
A consequence of Theorem \ref{thm:calibconst}(b) is the following construction of the irreducible $H_k^\ext$-modules
$B_k^\lambda$.  Keeping the setting and notation of \eqref{eq:tableaux-paths}, 
for $\lambda\in \cP^{(k)}$ and $S \in \cT_k^\lambda$, let
\begin{equation}
\hbox{$s_jS$ be the path from $(a^c)$ to $\lambda$ that differs from $S$ only at $S^{(j)}$.}
\label{sjLdefn}
\end{equation}
The path $s_jS$ is unique if it exists:  if $S = ((a^c)\to S^{(0)}\to S^{(1)}
\to \cdots \to S^{(k)})$ then $S^{(j+1)}$ is obtained by adding a box to $S^{(j)}$,  and 
$(s_jS)^{(j)}$ is obtained by moving a box of $S^{(j)}$ to the position of the added box in $S^{(j+1)}$.  In the case that 
$j=0$, the paths $s_0S$ and $S$ satisfy $(s_0S)^{(1)}=S^{(1)}$ and the partitions $(s_0S)^{(0)}$ and $S^{(0)}$ in $\cP^{(0)}$ 
differ by the placement of one box, with 
\begin{equation}\label{eq:move-a-box}
\tilde c((s_0S)^{(1)}/(s_0S)^{(0)}) = -\tilde c(S^{(1)}/S^{(0)}),
\end{equation}
where the shifted content of a box $\tilde c(\mathrm{box})$ is as defined in \eqref{ccdefn}.

\begin{cor} Keep the conditions of 
Theorems \ref{thm:HeckeActionOnTensorSpace} and \ref{thm:partitions-to-SLRs}.
In particular, assume that the genericity conditions of \eqref{newgenconds} hold so that
$q$ is not a root of unity, $a,b,c,d\in \ZZ_{>0}$ and $(a+c)-(b+d)\not\in \{0, \pm1, \pm2\}$.
Let $\lambda\in \cP^{(k)}$.  Then $B_k^\lambda$ has a basis $\{ v_S\ |\ S\in \cT_k^\lambda\}$ 
such that the $H_k^\ext$-action is given by
\begin{align*}
P v_S &= q^{2e_0(T)} v_S, \qquad
Z_i v_S = q^{2c(S^{(i)}/S^{(i-1)})} v_S,  \\
T_i v_S &= [T_i]_{S,S} v_S + \sqrt{-([T_i]_{S,S}-q)([T_i]_{S,S}+q^{-1})}\ v_{s_iS},
\qquad\hbox{for $i=1, \dots,k-1$}, \\
Y_1v_S &= [Y_1]_{S,S}v_S+ \sqrt{-([Y_1]_{S,S}-q^{-2d})([Y_1]_{S,S}-q^{2b})}\ v_{s_0 S}, \\
X_1v_S &= [X_1]_{S,S}v_S + q^{-2c(S^{(1)}/S^{(0)})}q^{(a-c+b-d)} \sqrt{-([X_1]_{S,S}-q^{2a})([X_1]_{S,S}-q^{-2c})}\ v_{s_0S},
\end{align*}
where $v_{s_jS}=0$ if $s_jS$ does not exist, and
\begin{align*}
[T_i]_{S,S} &= \frac{q-q^{-1}}{1-q^{2(c(S^{(i)}/S^{(i-1)})-c(S^{(i+1)}/S^{(i)}))}} , \\
[Y_1]_{S,S} 
&= \frac{ (q^{2b}+q^{-2d}) - (q^{2a}+q^{-2c})q^{2(b-d)}q^{-2c(S^{(1)}/S^{(0)})}} 
{1-q^{2(a-c+b-d)}q^{-4c(S^{(1)}/S^{(0)})}}, \\
[X_1]_{S,S} 
&= 
\frac{ (q^{2a}+q^{-2c}) - (q^{2b}+q^{-2d})q^{2(a-c)}q^{-2c(S^{(1)}/S^{(0)})}} 
							{1-q^{2(a-c+b-d)}q^{-4c(S^{(1)}/S^{(0)})}}.
\end{align*}

\end{cor}


\begin{proof} The appropriate basis of $B_k^\lambda$ is the one given in Proposition \ref{prop:calibrated}
and used also in the proof of Theorem \ref{thm:partitions-to-SLRs}.  It is only necessary to convert from the notation $v_w$ in 
Theorem \ref{thm:calibconst} to the notation $v_S$ using the bijection in Corollary \ref{onStdTableaux}.
Recall from \eqref{eq:ttoqconversions} that
$$\begin{array}{c}
	a_1=q^{2a}, \qquad a_2 = q^{-2c}, \qquad
	b_1=q^{2b}, \qquad b_2=q^{-2d},	\qquad \\
	t^{\frac12} = q, \qquad 
	t_k^{\frac12} =  a_1^{\frac12}(-a_2)^{-\frac12} 
				= -i q^{a+c},  \quad \text{ and }\quad 
	t_0^{\frac12} = b_1^{\frac12}(-b_2)^{-\frac12} 
				= -i q^{b+d}. \\
\end{array}
$$
From \eqref{snormActionPW} and \eqref{Weigenvals}, 
$$\gamma_{w^{-1}(i)} v_w = \Phi(W_i)v_S 
=-q^{-(a-c+b-d)}q^{2c(S^{(i)}/S^{(i-1)})}v_S.
$$
From \eqref{eq:normalized_gens}, \eqref{BraidMurphy}, and \eqref{Heckedefn},
$Y_1 = b_1^{\frac12}(-b_2)^{\frac12}T_0=iq^{b-d}T_0$ 
and $X_1 = (a_1 + a_2) - a_1a_2 Y_1 Z_1^{-1} = q^{2a} + q^{-2c} - q^{2(a-c)}  Y_1 Z_1^{-1}$.
With these conversions, the formulas
from \eqref{snormActionT} and \eqref{snormActionX} become
\begin{align*}
T_i v_S &= T_i v_w = [T_i]_{S,S} v_S + [T_i]_{s_iS,S} v_{s_iS},
\qquad\hbox{for $i=1, \dots,k-1$, } \\
Y_1 v_S &= i q^{b-d} T_0v_w = [Y_1]_{S,S}v_S+ [Y_1]_{s_0S,S}  v_{s_0 S} \ , \qquad\hbox{and} \\
X_1 v_S & = \left(q^{2a} + q^{-2c} - q^{2(a-c)}  Y_1 Z_1^{-1}\right)v_S 
= \left(q^{2a} + q^{-2c} - q^{2(a-c)}q^{-2c(S^{(1)}/S^{(0)})}  Y_1 \right)v_S \\
&= [X_1]_{S,S} v_S - [X_1]_{s_0S,S} v_{s_0 S},
\end{align*}
with
\begin{align*}
[T_i]_{S,S} &= [T_i]_{ww} = \frac{t^{\frac12}-t^{-\frac12}}{1-\gamma_{w^{-1}(i)}\gamma^{-1}_{w^{-1}(i+1)}} 
	=\frac{q-q^{-1}}{1-q^{2(c(S^{(i)}/S^{(i-i)}) - c(S^{(i+1)}/S^{(i)}))}},\\
[Y_1]_{S,S} &= iq^{b-d}[T_0]_{ww} = iq^{b-d}\frac{ (t_0^{\frac12}-t_0^{-\frac12}) + (t_k^{\frac12} -t_k^{-\frac12})\gamma^{-1}_{w^{-1}(1)}} 
{1-\gamma^{-2}_{w^{-1}(1)}} \\
	&= iq^{b-d}(-i)\frac{ (q^{(b+d)}+q^{-(b+d)}) - (q^{(a+c)}+q^{-(a+c)})q^{a-c+b-d}q^{-2c(S^{(1)}/S^{(0)})}} 
{1-q^{2(a-c+b-d)}q^{-4c(S^{(1)}/S^{(0)})}}
	\\
	&= \frac{ (q^{2b}+q^{-2d)}) - (q^{2a}+q^{-2c})q^{2(b-d)}q^{-2c(S^{(1)}/S^{(0)})}} 
{1-q^{2(a-c+b-d)}q^{-4c(S^{(1)}/S^{(0)})}},  \text{ and }\\
[X_1]_{S,S} &= q^{2a} + q^{-2c} - q^{2(a-c)}q^{-2c(S^{(1)}/S^{(0)})}[Y_1]_{S,S} \\
&= q^{2a} + q^{-2c} - q^{2(a-c)} q^{-2c(S^{(1)}/S^{(0)})} 
\frac{ (q^{2b}+q^{-2d}) - (q^{2a}+q^{-2c})q^{2(b-d)}q^{-2c(S^{(1)}/S^{(0)})}} 
{1-q^{2(a-c+b-d)}q^{-4c(S^{(1)}/S^{(0)})}}  \\
&=\frac{ (q^{2a}+q^{-2c}) - (q^{2b}+q^{-2d})q^{2(a-c)}q^{-2c(S^{(1)}/S^{(0)})}} 
{1-q^{2(a-c+b-d)}q^{-4c(S^{(1)}/S^{(0)})}}.
\end{align*}

On the two-dimensional subspace $\mathrm{span}_\CC\{ v_S, v_{s_0S}\}$ the action of $T_0$ in the basis $\{v_S, v_{s_0S}\}$ is a symmetric matrix $[T_0]$,
and so the matrix of $Y_1$ in this basis
is $[Y_1] = iq^{b-d}[T_0]$ is also symmetric.  The action of $Z_1$ is by a diagonal matrix $[Z_1]$, so $[Z_1]^t=[Z_1]$.
Therefore, using $X_1 = Z_1 Y_1^{-1}$ from \eqref{BraidMurphy} and $X_1 = (a_1 + a_2) - a_1a_2 X_1^{-1}$ from \eqref{Heckedefn},  we have $([X_1]^{-1})^t = ([Y_1][Z_1]^{-1})^t = ([Z_1]^{-1})^t[Y_1]^t=[Z_1]^{-1}[Y_1]$ and so 
$$[Z_1][X_1]^t[Z_1]^{-1}
=[Z_1]((a_1+a_2)-a_1a_2 [Z_1]^{-1} [Y_1]) [Z_1]^{-1} = [X_1].$$
Thus
$$
[Z_1]_{S,S}[X_1]_{s_0S,S}[Z_1^{-1}]_{s_0L,s_0S} = [X_1]_{S,s_0S}
\ \ \hbox{and}\ 
-[X_1]_{S,s_0S}[X_1]_{s_0S,S}=([X_1]_{S,S}-a_1)([X_1]_{S,S}-a_2),
$$
since $[X_1]$ is a $2\times 2$ matrix with eigenvalues $a_1$ and $a_2$ (as in the proof of Theorem \ref{thm:calibconst}).
Thus
\begin{align*}
[X_1]_{s_0S,S}
&=\sqrt{([X_1]_{s_0S,S})^2}
=\sqrt{[X_1]_{S,s_0S}[Z_1]_{S,S}^{-1}[X_1]_{s_0S,S}[Z_1]_{s_0S,s_0S}} \\
&=\sqrt{[Z_1]_{S,S}^{-1}[Z_1]_{s_0Ss_0S}}\sqrt{-([X_1]_{S,S}-q^{2a})([X_1]_{S,S}-q^{-2c})}.
\end{align*}
By \eqref{eq:move-a-box}, $c((s_0S)^{(1)}/(s_0S)^{(0)}) =- c(S^{(1)}/S^{(0)}) + (a-c+b-d)$, so that 
\begin{align*}
\sqrt{[Z_1]_{S,S}^{-1}[Z_1]_{s_0S,s_0S}}
&=q^{-c(S^{(1)}/S^{(0)})}q^{c((s_0S)^{(1)}/(s_0S)^{(0)})}
=q^{-2c(S^{(1)}/S^{(0)}) + (a-c+b-d) }.
\end{align*}
Thus
$$[X_1]_{s_0S,S}
=q^{-2c(S^{(1)}/S^{(0)})}q^{(a-c+b-d)} \sqrt{-([X_1]_{S,S}-q^{2a})([X_1]_{S,S}-q^{-2c})}.$$
\end{proof}

%

\end{document}